\documentclass[a4paper,oneside,english]{amsart}
\usepackage[T1]{fontenc}
\usepackage[utf8]{inputenc}
\usepackage{babel}
\usepackage{mathtools}
\usepackage{amstext}
\usepackage{amsthm}
\usepackage{amssymb}
\usepackage{esint}
\PassOptionsToPackage{normalem}{ulem}
\usepackage{ulem}
\usepackage[unicode=true,pdfusetitle,
 bookmarks=true,bookmarksnumbered=false,bookmarksopen=false,
 breaklinks=false,pdfborder={0 0 1},backref=false,colorlinks=false]
 {hyperref}

\makeatletter

\pdfpageheight\paperheight
\pdfpagewidth\paperwidth

\numberwithin{equation}{section}
\theoremstyle{plain}
\newtheorem{thm}{\protect\theoremname}[section]
\theoremstyle{remark}
\newtheorem{rem}[thm]{\protect\remarkname}
\theoremstyle{plain}
\newtheorem{prop}[thm]{\protect\propositionname}
\theoremstyle{plain}
\newtheorem{question}[thm]{\protect\questionname}
\theoremstyle{plain}
\newtheorem{lem}[thm]{\protect\lemmaname}
\theoremstyle{plain}
\newtheorem{cor}[thm]{\protect\corollaryname}

\makeatother

\providecommand{\corollaryname}{Corollary}
\providecommand{\lemmaname}{Lemma}
\providecommand{\propositionname}{Proposition}
\providecommand{\questionname}{Open Question}
\providecommand{\remarkname}{Remark}
\providecommand{\theoremname}{Theorem}

\begin{document}
\global\long\def\bbC{\mathbb{C}}%
\global\long\def\bbN{\mathbb{N}}%
\global\long\def\bbQ{\mathbb{Q}}%
\global\long\def\bbR{\mathbb{R}}%
\global\long\def\bbZ{\mathbb{Z}}%
\global\long\def\bbS{\mathbb{S}}%

\global\long\def\bfD{{\bf D}}%

\global\long\def\bfv{{\bf v}}%

\global\long\def\calA{\mathcal{A}}%
\global\long\def\calB{\mathcal{B}}%
\global\long\def\calC{\mathcal{C}}%
\global\long\def\calD{\mathcal{D}}%
\global\long\def\calE{\mathcal{E}}%
\global\long\def\calF{\mathcal{F}}%
\global\long\def\calg{\mathcal{G}}%
\global\long\def\calH{\mathcal{H}}%
\global\long\def\calI{\mathcal{I}}%
\global\long\def\calJ{\mathcal{J}}%
\global\long\def\calK{\mathcal{K}}%
\global\long\def\calL{\mathcal{L}}%
\global\long\def\calM{\mathcal{M}}%
\global\long\def\calN{\mathcal{N}}%
\global\long\def\calO{\mathcal{O}}%
\global\long\def\calP{\mathcal{P}}%
\global\long\def\calQ{\mathcal{Q}}%
\global\long\def\calR{\mathcal{R}}%
\global\long\def\calS{\mathcal{S}}%
\global\long\def\calT{\mathcal{T}}%
\global\long\def\calU{\mathcal{U}}%
\global\long\def\calV{\mathcal{V}}%
\global\long\def\calW{\mathcal{W}}%
\global\long\def\calX{\mathcal{X}}%
\global\long\def\calY{\mathcal{Y}}%
\global\long\def\calZ{\mathcal{Z}}%

\global\long\def\frkc{\mathfrak{c}}%
\global\long\def\frkd{\mathfrak{d}}%
\global\long\def\frkp{\mathfrak{p}}%
\global\long\def\frkr{\mathfrak{r}}%

\global\long\def\alp{\alpha}%
\global\long\def\bt{\beta}%
\global\long\def\dlt{\delta}%
\global\long\def\Dlt{\Delta}%
\global\long\def\eps{\epsilon}%
\global\long\def\gmm{\gamma}%
\global\long\def\Gmm{\Gamma}%
\global\long\def\kpp{\kappa}%
\global\long\def\tht{\theta}%
\global\long\def\lmb{\lambda}%
\global\long\def\Lmb{\Lambda}%
\global\long\def\vphi{\varphi}%
\global\long\def\omg{\omega}%

\global\long\def\rd{\partial}%
\global\long\def\aleq{\lesssim}%
\global\long\def\ageq{\gtrsim}%
\global\long\def\aeq{\simeq}%

\global\long\def\peq{\mathrel{\phantom{=}}}%
\global\long\def\To{\longrightarrow}%
\global\long\def\weakto{\rightharpoonup}%
\global\long\def\embed{\hookrightarrow}%
\global\long\def\Re{\mathrm{Re}}%
\global\long\def\Im{\mathrm{Im}}%
\global\long\def\chf{\mathbf{1}}%
\global\long\def\td#1{\widetilde{#1}}%
\global\long\def\br#1{\overline{#1}}%
\global\long\def\ul#1{\underline{#1}}%
\global\long\def\wh#1{\widehat{#1}}%
\global\long\def\rng#1{\mathring{#1}}%
\global\long\def\tint#1#2{{\textstyle \int_{#1}^{#2}}}%
\global\long\def\tsum#1#2{{\textstyle \sum_{#1}^{#2}}}%

\global\long\def\lan{\langle}%
\global\long\def\ran{\rangle}%
\global\long\def\blan{\big\langle}%
\global\long\def\bran{\big\rangle}%
\global\long\def\Blan{\Big\langle}%
\global\long\def\Bran{\Big\rangle}%

\global\long\def\Mod{\mathbf{Mod}}%
\global\long\def\out{\phantom{}^{\mathrm{(out)}}}%
\global\long\def\inn{\phantom{}^{\mathrm{(in)}}}%
\global\long\def\loc{\mathrm{loc}}%
\global\long\def\sg{\mathrm{sg}}%
\global\long\def\NL{\mathrm{NL}}%
\global\long\def\rad{\mathrm{rad}}%
\global\long\def\dec{\mathrm{d}}%
\global\long\def\coer{\mathrm{c}}%

\title[Classification for HMHF and NLH]{On classification of global dynamics for energy-critical equivariant
harmonic map heat flows and radial nonlinear heat equation}
\author{Kihyun Kim}
\email{kihyun.kim@snu.ac.kr (current); khyun@ihes.fr (previous)}
\address{Department of Mathematical Sciences and Research Institute of Mathematics,
Seoul National University, 1 Gwanak-ro, Gwanak-gu, Seoul 08826,  Korea
(current); IHES, 35 route de Chartres, 91440 Bures-sur-Yvette, France
(previous)}
\author{Frank Merle}
\email{merle@ihes.fr}
\address{IHES, 35 route de Chartres, 91440 Bures-sur-Yvette, France and AGM,
CY Cergy Paris Université, 2 av.\ Adolphe Chauvin, 95302 Cergy-Pontoise
Cedex, France}
\begin{abstract}
We consider the global dynamics of finite energy solutions to energy-critical
equivariant harmonic map heat flow (HMHF) and radial nonlinear heat
equation (NLH). It is known that any finite energy equivariant solutions
to (HMHF) decompose into finitely many harmonic maps (bubbles) separated
by scales and a body map, as approaching to the maximal time of existence.
Our main result for (HMHF) gives a complete classification of their
dynamics for equivariance indices $D\geq3$; (i) they exist globally
in time, (ii) the number of bubbles and signs are determined by the
energy class of the initial data, and (iii) the scales of bubbles
are asymptotically given by a universal sequence of rates up to scaling
symmetry. In parallel, we also obtain a complete classification of
$\dot{H}^{1}$-bounded radial solutions to (NLH) in dimensions $N\geq7$,
building upon soliton resolution for such solutions. To our knowledge,
this provides the first rigorous classification of bubble tree dynamics
within symmetry. We introduce a new approach based on the energy method
that does not rely on maximum principle. The key ingredient of the
proof is a monotonicity estimate near any bubble tree configurations,
which in turn requires a delicate construction of modified multi-bubble
profiles also.
\end{abstract}

\maketitle

\tableofcontents{}

\section{Introduction}

We consider the energy-critical harmonic map heat flow under equivariant
symmetry 
\begin{equation}
\left\{ \begin{aligned}\rd_{t}v & =\rd_{rr}v+\frac{1}{r}\rd_{r}v-\frac{D^{2}\sin(2v)}{2r^{2}}\\
v(0,r) & =v_{0}(r)\in\bbR,
\end{aligned}
\quad(t,r)\in[0,T)\times(0,\infty)\right.\tag{HMHF}\label{eq:HMHF-equiv}
\end{equation}
with equivariance indices $D\in\bbZ_{\geq1}$ and the energy-critical
heat equation under radial symmetry
\begin{equation}
\left\{ \begin{aligned}\rd_{t}u & =\rd_{rr}u+\frac{N-1}{r}\rd_{r}u+|u|^{\frac{4}{N-2}}u,\\
u(0,r) & =u_{0}(r)\in\bbR,
\end{aligned}
\quad(t,r)\in[0,T)\times(0,\infty)\right.\tag{NLH}\label{eq:NLH-rad}
\end{equation}
in dimensions $N\geq3$, for finite energy initial data.

Our main results (Theorem~\ref{thm:main-HMHF} and Theorem~\ref{thm:main-NLH})
give a complete classification of the global dynamics of finite energy
solutions to \eqref{eq:HMHF-equiv} when $D\geq3$ and $\dot{H}^{1}$-bounded
solutions to \eqref{eq:NLH-rad} when $N\geq7$.

\subsection{Energy-critical equivariant harmonic map heat flow}

The \emph{harmonic map heat flow} on $\bbR^{2}$ with the $\bbS^{2}$
target is the gradient flow associated to the Dirichlet energy 
\[
E[\phi]=\frac{1}{2}\int_{\bbR^{2}}|\nabla\phi|^{2}dx,
\]
where $\bbS^{2}\embed\bbR^{3}$ is the standard unit 2-sphere and
$\phi$ is a map from $\bbR^{2}\to\bbS^{2}$. The associated Cauchy
problem is then given by 
\begin{equation}
\left\{ \begin{aligned}\rd_{t}\phi & =\Dlt\phi+\phi|\nabla\phi|^{2},\\
\phi(0,x) & =\phi_{0}(x)\in\bbS^{2},
\end{aligned}
\quad(t,x)\in[0,T)\times\bbR^{2}.\right.\label{eq:HMHF}
\end{equation}
The RHS of \eqref{eq:HMHF} is called the \emph{tension field}. In
a more general context, the harmonic map heat flow for maps from a
Riemannian manifold $\calM$ to another $\calN$ was introduced by
Eells and Sampson \cite{EellsSampson1964AJM} to study the deformation
of a given mapping into \emph{harmonic maps}, i.e., the critical points
of the Dirichlet energy. The harmonic maps are the static solutions
to the harmonic map heat flow.

The two-dimensional case ($\dim\calM=2$), such as \eqref{eq:HMHF},
is distinguished as the energy is conformally invariant. In our setting
$\calM=\bbR^{2}$, we have scaling symmetry; for any $\lmb\in(0,\infty)$
and a solution $\phi(t,x)$ to \eqref{eq:HMHF}, $\phi(\frac{t}{\lmb^{2}},\frac{x}{\lmb})$
is also a solution to \eqref{eq:HMHF}. That the underlying dimension
is $2$ implies that the energy is invariant under the scaling transform.
Hence we call \eqref{eq:HMHF} \emph{energy-critical}. By the result
of Struwe \cite{Struwe1985} (originally proved for a closed surface
$\calM$), the Cauchy problem is well-posed in the \emph{energy class},
i.e., the set of maps with finite energy.

A particularly interesting phenomenon in this two-dimensional setting
is \emph{bubbling}, whose harmonic map counterpart goes back to \cite{SacksUhlenbeck1981}.
Struwe proved in \cite{Struwe1985} that any finite energy initial
data admits a global weak solution that is regular except finitely
many spacetime points. Moreover, these singular points are characterized
by the property that suitably rescaled flow converges locally in space
to a nontrivial harmonic map as approaching to the singular time.
The energy identity, written in our context as 
\[
E[\phi(t_{2})]-E[\phi(t_{1})]=\int_{t_{1}}^{t_{2}}\int_{\bbR^{2}}|\rd_{t}\phi|^{2}dxdt,
\]
and its localized version play a key role. A series of works such
as Qing \cite{Qing1995CommAnalGeom}, Ding--Tian \cite{DingTian1995CommAnalGeom},
Wang \cite{Wang1996HoustonJM}, Qing--Tian \cite{QingTian1997CPAM},
and Lin--Wang \cite{LinWangCVPDE} provide more refined understanding
on this energy identity (e.g., the lack of compactness for Palais--Smale
sequences). In consequence, for a \emph{well-chosen time sequence},
the solution looks like a superposition of bubbles (scaled harmonic
maps) and a body map. See also Topping \cite{Topping1997JDG,Topping2004AnnMath,Topping2004MathZ}.
We refer to the book \cite{LinWang2008book} by Lin and Wang for more
detailed discussions on bubbling and also for an introduction to harmonic
map heat flows.

Recently, within equivariant symmetry (which is the setting of this
paper), Jendrej and Lawrie \cite{JendrejLawrie2023CVPDE} obtained
the full \emph{continuous-in-time} bubble decomposition; see Proposition~\ref{prop:HMHF-bubble-dec}
below for the statement. Jendrej, Lawrie, and Schlag \cite{JendrejLawrieSchlag2023arXiv}
obtained a version of continuous-in-time decomposition without symmetry.
It is worth mentioning that the work \cite{JendrejLawrie2023CVPDE}
is inspired from the works on soliton resolution for critical wave
equations; see for example \cite{DuyckaertsKenigMerle2013CambJMath,Cote2015CPAM,JiaKenig2017AJM,DuyckaertsJiaKenigMerle2017GAFA,DuyckaertsKenigMerle2023Acta,JendrejLawrie2021arXiv,JendrejLawrie2023AnnPDE}.

In this paper, we consider harmonic map heat flows under equivariant
symmetry. For $D\in\bbZ_{\geq1}$, the $D$-equivariant symmetry ansatz
\begin{equation}
\phi(t,x)=\big(\sin v(t,r)\cos(D\tht),\sin v(t,r)\sin(D\tht),\cos v(t,r)\big),\label{eq:equiv-ansatz}
\end{equation}
where $(r,\tht)$ are polar coordinates on $\bbR^{2}$, is preserved
under the flow of \eqref{eq:HMHF}. The equation then reduces to \eqref{eq:HMHF-equiv}
given at the beginning of the introduction. The energy functional
reads 
\[
E[v(t)]=2\pi\int_{0}^{\infty}\frac{1}{2}\Big((\rd_{r}v(t,r))^{2}+\frac{D^{2}\sin^{2}v(t,r)}{r^{2}}\Big)rdr
\]
and the energy identity reads
\[
E[v(t_{2})]-E[v(t_{1})]=2\pi\int_{t_{1}}^{t_{2}}\int_{0}^{\infty}(\rd_{t}v(t,r))^{2}rdrdt.
\]

It is well-known that the energy class within $D$-equivariance is
the disjoint union of the connected components $\calE_{\ell,m}$ for
$\ell,m\in\bbZ$, where 
\begin{equation}
\calE_{\ell,m}\coloneqq\{v:(0,\infty)\to\bbR:E[v]<+\infty,\ \lim_{r\to0}v(r)=\ell\pi,\ \lim_{r\to+\infty}v(r)=m\pi\}.\label{eq:def-calE-ellm}
\end{equation}
The dynamics of \eqref{eq:HMHF-equiv} is then understood separately
on each component $\calE_{\ell,m}$. Each $\calE_{\ell,m}$ is an
affine space. When $\ell=m=0$, $\calE\coloneqq\calE_{0,0}$ becomes
a vector space and we equip this space with the norm 
\[
\|v\|_{\calE}^{2}\coloneqq\|\rd_{r}v\|_{L^{2}(rdr)}^{2}+\|r^{-1}v\|_{L^{2}(rdr)}^{2}.
\]
This norm naturally induces a metric on each $\calE_{\ell,m}$.

The equation \eqref{eq:HMHF-equiv} is invariant under the transforms
$v\mapsto v+\pi$ and $v\mapsto-v$. At the map level (i.e., $\phi=(\phi_{1},\phi_{2},\phi_{3})$),
these transforms correspond to isometries on $\bbS^{2}$ $\phi\mapsto(-\phi_{1},-\phi_{2},-\phi_{3})$
(orientation reversing) and $\phi\mapsto(-\phi_{1},-\phi_{2},\phi_{3})$
(orientation preserving), respectively. The latter transform simply
corresponds to rotation about the $z$-axis by $\pi$, so the maps
$\phi$ and $(-\phi_{1},-\phi_{2},\phi_{3})$ are homotopic outside
symmetry. The disconnectedness of $\calE_{0,1}$ and $\calE_{0,-1}$
in our setting is due to the symmetry ansatz \eqref{eq:equiv-ansatz}.

There is a nontrivial $D$-equivariant harmonic map
\[
Q(r)\coloneqq2\arctan(r^{D})
\]
and it is unique in the sense that 
\[
\iota Q_{\lmb}(r)\coloneqq\iota Q\Big(\frac{r}{\lmb}\Big)
\]
for any sign $\iota\in\{\pm1\}$, scale $\lmb\in(0,\infty)$, and
their modifications by an addition of $\ell\pi$, $\ell\in\bbZ$,
generate all nontrivial $D$-equivariant harmonic maps. Note that
$Q\in\calE_{0,1}$ and $E[Q]=4\pi D$.

Recently, Jendrej and Lawrie \cite{JendrejLawrie2023CVPDE} established
the continuous-in-time bubble decomposition for $D$-equivariant harmonic
map heat flows of finite energy. It roughly says that any such solutions
decompose into the sum of \emph{scale-decoupled bubbles }(or, simply
\emph{multi-bubbles})\emph{ }and a radiation: 
\[
v(t)\approx\sum_{j=1}^{J}\iota_{j}Q_{\lmb_{j}(t)}+v^{\ast}
\]
as $t\to T$ (the maximal time of existence) with $J$ the number
of bubbles, $\lmb_{J}(t)\ll\dots\ll\lmb_{1}(t)$, and a radiation
profile $v^{\ast}$ (which becomes trivial if $T=+\infty$). For the
precise statement, see Proposition~\ref{prop:HMHF-bubble-dec} in
Section~\ref{subsec:Statements-soliton-res}.

It is then natural to consider the \emph{classification problem};
which scenarios of the number of bubbles $J$, signs $\iota_{j}$,
and scales $\lmb_{j}(t)$ can actually arise? In \cite{GustafsonNakanishiTsai2010CMP},
Gustafson, Nakanishi, and Tsai classified the dynamics when $D\geq2$,
$J=1$ (one bubble), and the initial data is close to $Q$. In particular,
when $D\geq3$, one has asymptotic stability: $T=+\infty$ and $\lmb_{1}(t)\to L_{\infty}\in(0,\infty)$.
However, the classification including multi-bubbles has remained open.

We now state our main result for \eqref{eq:HMHF-equiv}. We give a
complete classification when $D\geq3$. Further discussions on the
result will follow after we state the analogous result (Theorem~\ref{thm:main-NLH}
below) for the critical heat equation.
\begin{thm}[Classification of $D$-equivariant solutions within high equivariance]
\label{thm:main-HMHF}Let $D\geq3$. Let $\ell,m\in\bbZ$ and $v_{0}\in\calE_{\ell,m}$
be an initial data. Then, the corresponding solution $v(t)$ to \eqref{eq:HMHF-equiv}
is global and admits the decomposition 
\begin{equation}
v(t)-\ell\pi-\iota\sum_{j=1}^{J}Q_{\lmb_{j,L_{\infty}}^{\mathrm{ex}}(t)}\to0\text{ in }\calE\text{ as }t\to+\infty\label{eq:main-thm-HMHF-dec}
\end{equation}
for some real $L_{\infty}\in(0,\infty)$ and 
\begin{equation}
J=|m-\ell|\quad\text{and}\quad\iota=\mathrm{sgn}(m-\ell).\label{eq:main-thm-J-iota}
\end{equation}
Here, the rates $\lmb_{j,L_{\infty}}^{\mathrm{ex}}(t)$ are given
by 
\[
\lmb_{j,L_{\infty}}^{\mathrm{ex}}(t)\coloneqq L_{\infty}^{1+2\alpha_{j}}\cdot\frac{\beta_{j}}{t^{\alpha_{j}}}\quad\text{with}\quad\alp_{j}\coloneqq\frac{1}{2}\Big(\frac{D}{D-2}\Big)^{j-1}-\frac{1}{2}
\]
and $\beta_{j}$s are the universal constants defined by \eqref{eq:def-beta_j}
and \eqref{eq:def-kappa-explicit-int}.
\end{thm}

\subsection{Energy-critical radial nonlinear heat equation}

We also consider the \emph{energy-critical heat equation} 
\begin{equation}
\left\{ \begin{aligned}\rd_{t}u & =\Dlt u+|u|^{\frac{4}{N-2}}u,\\
u(0,x) & =u_{0}\in\bbR,
\end{aligned}
\quad(t,x)\in[0,T)\times\bbR^{N}\right.\label{eq:NLH}
\end{equation}
in dimensions $N\geq3$. \eqref{eq:NLH} is a gradient flow associated
to the energy 
\[
E[u(t)]=\int_{\bbR^{N}}\Big(\frac{1}{2}|\nabla u(t,x)|^{2}-\frac{N-2}{2N}|u(t,x)|^{\frac{2N}{N-2}}\Big)dx.
\]
The energy space is $\dot{H}^{1}=\dot{H}^{1}(\bbR^{N})$, \eqref{eq:NLH}
is well-posed in $\dot{H}^{1}$ (see \cite{Weissler1980,BrezisCazenave1996}
and also \cite{CollotMerleRaphael2017CMP}), and it has scaling symmetry:
for any $\lmb\in(0,\infty)$ and a solution $u(t,x)$ to \eqref{eq:NLH},
$\frac{1}{\lmb^{(N-2)/2}}u(\frac{t}{\lmb^{2}},\frac{x}{\lmb})$ is
also a solution to \eqref{eq:NLH}. The energy is invariant under
this scaling transform. Imposing the radial symmetry $u(t,x)=u(t,r)$
with $r=|x|$ (abuse of notation), we arrive at the equation \eqref{eq:NLH-rad}.

The $D$-equivariant harmonic map heat flow is closely related to
the radial critical heat equation in dimension $N=2D+2$. This connection
can roughly be seen from the identity 
\begin{equation}
\Big(\rd_{rr}+\frac{1}{r}\rd_{r}-\frac{D^{2}}{r^{2}}\Big)(r^{D}u)=r^{D}\Big(\rd_{rr}+\frac{N-1}{r}\rd_{r}\Big)u\label{eq:Laplacian-connection}
\end{equation}
which transforms the Laplacian for $D$-equivariant functions on $\bbR^{2}$
into that for radial functions on $\bbR^{N}$, and that $u\mapsto r^{D}u$
is an isometry from $L^{2}(r^{N-1}dr)$ to $L^{2}(rdr)$.

\eqref{eq:NLH-rad} admits a nontrivial radial static solution
\[
W(r)=\frac{1}{(1+\frac{r^{2}}{N(N-2)})^{\frac{N-2}{2}}}
\]
and it is unique in the sense that 
\[
\iota W_{\lmb}(r)\coloneqq\frac{\iota}{\lmb^{\frac{N-2}{2}}}W\Big(\frac{r}{\lmb}\Big)
\]
for any sign $\iota\in\{\pm1\}$ and scale $\lmb\in(0,\infty)$ generate
all nonzero radial static solutions in $\dot{H}^{1}$.

Soliton resolution, the analogue of bubble decomposition for \eqref{eq:HMHF-equiv},
for $\dot{H}^{1}$-bounded radial solutions to \eqref{eq:NLH-rad}
was announced by Matano and Merle in the mid 10's and the proof was
presented at a conference at IHES. As the result did not appear in
a paper form, we will state this result as an open question (see Open
Question~\ref{conj:radial-soliton-resolution}) in Section~\ref{subsec:Statements-soliton-res}.
Note that $\dot{H}^{1}$-boundedness is necessary due to the existence
of type-I blow-up.

As for \eqref{eq:HMHF-equiv}, we also consider the classification
problem for \eqref{eq:NLH-rad}. When $J=1$ and $N\geq7$, a result
in Collot--Merle--Raphaël \cite{CollotMerleRaphael2017CMP} yields
the (conditional) asymptotic stability $\lmb_{1}(t)\to L_{\infty}\in(0,\infty)$.
Our main result for \eqref{eq:NLH-rad} gives the full multi-bubble
generalization under radial symmetry, leading to a complete classification
of $\dot{H}^{1}$-bounded radial solutions to critical heat equations
when $N\geq7$.
\begin{thm}[Classification of $\dot{H}^{1}$-bounded radial solutions in large
dimensions]
\label{thm:main-NLH}Let $N\geq7$. Assume the statement in Open
Question~\ref{conj:radial-soliton-resolution} holds. Let $u(t)$
be a $\dot{H}^{1}$-bounded radial solution to \eqref{eq:NLH-rad}.
Then, $u(t)$ is global and admits a decomposition 
\begin{equation}
u(t)-\iota\sum_{j=1}^{J}(-1)^{j}W_{\lmb_{j,L_{\infty}}^{\mathrm{ex}}(t)}\to0\text{ in }\dot{H}^{1}\text{ as }t\to+\infty\label{eq:main-thm-NLH-dec}
\end{equation}
for some integer $J\ge0$, sign $\iota\in\{\pm1\}$, and real $L_{\infty}\in(0,\infty)$.
Here, the rates $\lmb_{j,L_{\infty}}^{\mathrm{ex}}(t)$ are given
by 
\[
\lmb_{j,L_{\infty}}^{\mathrm{ex}}(t)\coloneqq L_{\infty}^{1+2\alpha_{j}}\cdot\frac{\beta_{j}}{t^{\alpha_{j}}}\quad\text{with}\quad\alp_{j}\coloneqq\frac{1}{2}\Big(\frac{N-2}{N-6}\Big)^{j-1}-\frac{1}{2}
\]
and $\beta_{j}$s are the universal constants defined by \eqref{eq:def-beta_j}
and \eqref{eq:def-kappa-explicit-int}.
\end{thm}

\begin{rem}
\label{rem:Open-Question-resolved}Open Question~\ref{conj:radial-soliton-resolution}
\emph{was resolved affirmatively} in the paper \cite{Aryan2024arXiv}
one month later we posted our article on arXiv.
\end{rem}

\subsection{\label{subsec:Discussion}Discussions on the results}

We comment on Theorem~\ref{thm:main-HMHF} and Theorem~\ref{thm:main-NLH}.\smallskip{}

\emph{1. Method and novelties.} We introduce a new approach for the
classification of multi-bubble dynamics. Our proof is based solely
on the energy method and is not reliant on maximum principle. We also
deal with sign changing solutions. The crucial ingredient of the proof
is a \emph{monotonicity estimate near any multi-bubble configurations}.
This is not difficult in the one-bubble case, but our multi-bubble
analysis requires a refined construction of \emph{modified multi-bubble
profiles}. See Section~\ref{subsec:Strategy} for more details on
the proof. We believe that the approach introduced in this paper is
robust and can be applied in more general settings such as wave equations
or geometric evolution equations. \smallskip{}

\emph{2. Existence of scenarios in classification.} All the scenarios
in our main theorems exist. Thus the theorems give a \emph{complete}
classification. For \eqref{eq:HMHF-equiv}, the existence is immediate;
for any $\ell\in\bbZ$, $J\in\bbN_{\geq0}$, and $\iota\in\{\pm1\}$,
\emph{every solution} in the class $\calE_{\ell,\ell+\iota J}$ admits
the decomposition \eqref{eq:main-thm-HMHF-dec} for some $L_{\infty}\in(0,\infty)$.
The value of $L_{\infty}$ can be adjusted arbitrarily by scaling.
For \eqref{eq:NLH-rad}, del Pino, Musso, and Wei constructed each
scenario \cite{delPinoMussoWei2021AnalPDE}. We remark that our method
combined with a shooting argument will provide an alternative construction
in $\dot{H}^{1}$, but we will not address it in this paper.\smallskip{}

\emph{3. Lower equivariance.} When $D\leq2$, the uniqueness result
(or \emph{the universality of blow-up speeds}) for the scales as in
our main theorems no longer holds. Our theorems are optimal also in
this sense.

When $D=2$, Gustafson, Nakanishi, and Tsai \cite{GustafsonNakanishiTsai2010CMP}
considered the one-bubble scenario $J=1$ and proved that the asymptotic
behavior of the scale $\lmb_{1}(t)$ depends on the spatial decay
of initial data. It turns out that the scenarios such as $\lmb_{1}(t)\to0$,
$\lmb_{1}(t)\to\infty$, $\lmb_{1}(t)$ eternally oscillating, are
all possible. (Note that this asymptotic behavior is different from
the formal asymptotics in \cite{vandenBergHulshofKing2003}.) Nevertheless,
a generalization of the classification result near one bubble to the
multi-bubble setting seems possible; this will be reported in the
future.

The $D<2$ case is more dramatic. First of all, there is a finite-time
blow-up of bubbling. Chang, Ding, and Ye \cite{ChangDingYe1992} proved
the existence of finite-time blowing up $1$-equivariant HMHFs when
the domain is a ball. More direct constructions with sharp descriptions
of the blow-up dynamics were later obtained by Raphaël and Schweyer
\cite{RaphaelSchweyer2013CPAM,RaphaelSchweyer2014AnalPDE} using the
method developed in \cite{RaphaelRodnianski2012Publ.Math.,MerleRaphaelRodnianski2013InventMath}.
A non-radial construction of bubbles simultaneously blowing up at
$k\geq1$ different points was obtained in \cite{DaviladelPinoWei2020Invent}.
See also finite-time (type-II) blow-up constructions \cite{Schweyer2012JFA,delPinoMussoWei2019ActaSinica,delPinoMussoWeiZhou2020DCDS,delPinoMussoWeiZhangZhou2020arXiv,Harada2020AIHP,Harada2020AnnPDE}
for \eqref{eq:NLH} with earlier formal constructions \cite{HerreroVelazquez-preprint,HerreroVelazquezCRASP1994,FilippasHerreroVelazquez2000}.
On the other hand, van der Hout proved in \cite{vanderHout2003JDE}
that there is no finite-time blowing up bubble tree ($J\geq2$ and
$T<+\infty$) for \eqref{eq:HMHF-equiv} when $D=1$. Finally, there
is no rigidity for the asymptotic behavior of scales for global solutions
as in the $D=2$ case. The spatial decay of initial data plays a key
role; see for example \cite{FilaKing2012,delPinoMussoWei2020AnalPDE,WeiZhangZhou2023arXiv}.\smallskip{}

\emph{4. Classification problem.} We first discuss classification
problems for nonlinear heat equations on $\bbR^{N}$:
\[
\rd_{t}u=\Dlt u+|u|^{p-1}u,\qquad p>1.
\]
In the subcritical case $p<p_{s}=\frac{N+2}{N-2}$, finite-time blow-up
dynamics is classified; Giga and Kohn \cite{GigaKohn1985CPAM,GigaKohn1987Indiana}
and Giga, Matsui, and Sasayama \cite{GigaMatsuiSasayama2004Indiana}
proved that all such solutions are known to be of \emph{type-I}, i.e.,
$\|u(t)\|_{L^{\infty}}\aeq(T-t)^{-\frac{1}{p-1}}$ (the rate given
by the ODE solution). A similar result is also available for sub-conformal
nonlinear wave equations; see Merle--Zaag \cite{MerleZaag2003AJM,MerleZaag2005MathAnn}.

On the other hand, when $p$ is sufficiently large ($p$ larger than
so-called Joseph--Lundgren exponent $p_{JL}$), a finite-time blow-up
different from type-I (we call \emph{type-II} and $\|u(t)\|_{L^{\infty}}\gg(T-t)^{-\frac{1}{p-1}}$)
was found by Herrero and Velázquez \cite{HerreroVelazquez-preprint,HerreroVelazquezCRASP1994}
together with some polynomial rates of $\|u(t)\|_{L^{\infty}}$. Later,
Mizoguchi \cite{Mizoguchi2007MathAnn,Mizoguchi2011TAMS} proved that
\emph{radial} and \emph{nonnegative} (with more assumptions on $u_{0}$)
type-II blow-up solutions indeed have the rates found in \cite{HerreroVelazquez-preprint,HerreroVelazquezCRASP1994}.
See also \cite{Mizoguchi2020CPAM}.

As discussed in the previous item, (sign changing) type-II blow-up
exists in the critical case $p=p_{s}$ in lower dimensions $N\leq6$.
Restricted to radial solutions, this gap in the range of $p$ indeed
exists; Matano and Merle \cite{MatanoMerle2004CPAM} proved the non-existence
of radial type-II blow-up solutions in the range $p_{s}<p<p_{JL}$.
Moreover, in the critical case $p=p_{s}$, any radial type-II finite-time
blow-up solutions must be sign changing. Recently, Wang and Wei \cite{WangWei2021arXiv}
proved that any finite-time blow-up with non-radial $u_{0}\geq0$
when $p=p_{s}$, $N\geq7$, is type-I.

Our main results are closely related to the analysis of type-II blow-up
dynamics in the critical case $p=p_{s}$. Type-II blow-up solutions
in this case are typically constructed via concentration of solitons
(or bubbling) and our main theorems in particular classify the asymptotic
dynamics of multi-solitons. The classification in this sense is not
limited to the equations considered in this paper, but rather appears
universally. See for example the works in the one-bubble case: the
series of works \cite{MerleRaphael2005AnnMath,MerleRaphael2003GAFA,Raphael2005MathAnn,MerleRaphael2004InventMath,MerleRaphael2006JAMS,MerleRaphael2005CMP}
for $L^{2}$-critical NLS, \cite{MartelMerleRaphael2014Acta} for
$L^{2}$-critical gKdV, \cite{GustafsonNakanishiTsai2010CMP} for
\eqref{eq:HMHF-equiv}, \cite{CollotMerleRaphael2017CMP} for \eqref{eq:NLH},
and also \cite{Kim2022arXiv2}.

In a different context, Merle and Zaag \cite{MerleZaag2012AJM} classified
the behavior of solutions to 1D semi-linear wave equations near the
blow-up curves including characteristic blow-up points, which might
be seen as a classification result involving multi-solitons. Note
that Côte and Zaag \cite{CoteZaag2013CPAM} constructed all the scenarios
in the classification.

\smallskip{}

\emph{5. Some other results.} The non-radial case is more interesting
and has richer dynamics. A new scenario of global solutions with solitons
concentrating at several different points exists \cite{CortazarDelPinoMusso2020JEMS}.
See also an earlier construction \cite{MartelRaphael2018AnnSci} by
Martel and Raphaël for the construction of strongly interacting solitons
for $L^{2}$-critical NLS. For weakly interacting solitons (or weakly
interacting with the radiation), see \cite{Merle1990CMP,BourgainWang1997,Martel2005AJM}.
Finally, we refer to \cite{Jendrej2017AnalPDE,Jendrej2019AJM} for
constructions of radially symmetric two-bubbles for energy-critical
wave and Schrödinger equations.

\subsection{\label{subsec:Notation}Notation}

We introduce a few sets of notation. Apart from the basic one, we
need additional notation to deal with multi-bubble solutions and treat
both \eqref{eq:HMHF-equiv} and \eqref{eq:NLH-rad} at the same time.

\subsubsection{Basic notation}
\begin{itemize}
\item For $A\in\bbR$ and $B\geq0$, we denote $A\aleq B$ (or, $A=O(B)$)
if $|A|\leq CB$ for some universal constant $C>0$. If $A,B\geq0$,
we denote $A\aeq B$ if $A\aleq B$ and $A\ageq B$. The dependence
on parameters is written in subscripts. Any dependence on $N$ (or
$D=\frac{N-2}{2}$) will be ignored. The dependence on $J$ (the number
of bubbles) will be tracked only in Section~\ref{subsec:Construction-of-vphi}
and be ignored elsewhere.
\item We use the notation $\chf_{A}$. If $A$ is a statement, it means
that it is $1$ if $A$ is true and $0$ otherwise. If $A$ is a set,
then it denotes the indicator function on $A$; $\chf_{A}(x)=1$ if
$x\in A$ and $0$ otherwise.
\item $\chi=\chi(r)$ denotes a smooth radial cutoff function such that
$\chi(r)=1$ if $r\leq1$ and $\chi(r)=0$ if $r\geq2-\frac{1}{10}$.
We denote $\chi_{R}\coloneqq\chi(\cdot/R)$.
\item $\lan x\ran\coloneqq(1+|x|^{2})^{1/2}$.
\item We use small $o$-notation such as $o_{\alp^{\ast}\to0}(1)$ and $o_{t_{0}\to T}(1)$,
which denote some quantities going to zero as $\alp^{\ast}\to0$ and
$t_{0}\to T$, respectively. We also denote $\dlt(\alp^{\ast})=o_{\alp^{\ast}\to0}(1)$.
In Sections~\ref{sec:Spacetime-estimate}--\ref{sec:Proof-of-main-thm},
$o(1)$ denotes $o_{t\to T}(1)$.
\item For functions $g$ and $h$ identified as radial parts of functions
on $\bbR^{N}$, we denote the integral $\int g=c_{N}\int_{0}^{\infty}g(r)r^{N-1}dr$
and the inner product $\lan g,h\ran=\int gh$, where $c_{N}$ is the
volume of the unit $(N-1)$-sphere.
\item For $k\in\bbN$, denote $|f|_{k}\coloneqq\max\{|f|,|r\rd_{r}f|,\dots,|r^{k}\rd_{r}^{k}f|\}$
and $|f|_{-k}\coloneqq r^{-k}|f|_{k}$.
\end{itemize}

\subsubsection{\label{subsec:Unified-notation}Unified notation for \eqref{eq:NLH-rad}
and \eqref{eq:HMHF-equiv}.}

We will consider \eqref{eq:HMHF-equiv} and \eqref{eq:NLH-rad} simultaneously
by introducing a unified notation. These two equations are written
in a unified form 
\[
\rd_{t}u=\rd_{rr}u+\frac{N-1}{r}\rd_{r}u+r^{-(D+2)}f(v).
\]
The definitions of various objects depend on the equations. Unless
stated otherwise, we assume throughout the paper
\[
D>2.
\]

\begin{itemize}
\item For \eqref{eq:HMHF-equiv}, $v=v(t,r)$, $D\in\bbZ_{\geq1}$, and
$Q(r)=2\arctan(r^{D})$ are given. We define $u=r^{-D}v$, $N=2D+2$,
$p=\frac{N+2}{N-2}$, $f(v)=\frac{D^{2}}{2}\{2v-\sin(2v)\}$, and
$W(r)=r^{-D}Q(r)$.
\item For \eqref{eq:NLH-rad}, $u=u(t,r)$, $N\in\bbZ_{\geq3}$, $p=\frac{N+2}{N-2}$,
and $W(r)=(1+\frac{r^{2}}{N(N-2)})^{-\frac{N-2}{2}}$ are given. We
define $D=\frac{N-2}{2}$, $v=r^{D}u$, $f(v)=|v|^{\frac{4}{N-2}}v$,
and $Q(r)=r^{D}W(r)$. 
\item The functions $u,W,U$ are related to the functions $v,Q,P$ via 
\[
u=r^{-D}v,\qquad W=r^{-D}Q,\qquad U=r^{-D}P.
\]
For a radial function $\vphi(r)$, we define $\rng{\vphi}(r)$ by
the relation 
\[
\vphi(r)=r^{-D}\rng{\vphi}(r).
\]
We view $u,W,U$ and $g$ as radial parts of functions on $\bbR^{N}$
and $v,Q,P$ as radial parts of functions on $\bbR^{2}$.
\item Let $\lmb\in(0,\infty)$. For a function $\vphi$ identified as the
radial part of a function on $\bbR^{N}$ (such as $u,W,U,g$), we
define 
\[
\vphi_{\lmb}(r)\coloneqq\frac{1}{\lmb^{D}}\vphi\Big(\frac{r}{\lmb}\Big),\qquad\Lmb\vphi\coloneqq(r\rd_{r}+D)\vphi,
\]
whereas for a function $\rng{\vphi}$ identified as the radial part
of a function on $\bbR^{2}$ (such as $v,Q,P,h$), we define 
\[
\rng{\vphi}_{\lmb}(r)\coloneqq\rng{\vphi}\Big(\frac{r}{\lmb}\Big),\qquad\Lmb\rng{\vphi}\coloneqq r\rd_{r}\rng{\vphi}.
\]
\item Once $\dot{H}^{1}$-scalings and the generator $\Lmb$ are defined
by either of the two ways above, we set 
\[
\phi_{\ul{\lmb}}\coloneqq\frac{1}{\lmb^{2}}\phi_{\lmb}\quad\text{and}\quad\Lmb_{-1}\phi\coloneqq\Lmb\phi+2\phi.
\]
\item For $\lmb\in(0,\infty)$, we usually denote $y=\frac{r}{\lmb}$ as
a renormalized space variable.
\end{itemize}

\subsubsection{Notation for multi-bubbles}

Let $J\in\{1,2,\dots\}$, $\vec{\iota}=(\iota_{1},\dots,\iota_{J})\in\{\pm1\}^{J}$,
and $\vec{\lmb}=(\lmb_{1},\dots,\lmb_{J})\in(0,\infty)^{J}$ be given.
\begin{itemize}
\item We define 
\[
\br{\lmb}_{j}\coloneqq\sqrt{\lmb_{j}\lmb_{j-1}},\quad\mu_{j}\coloneqq\lmb_{j}/\lmb_{j-1},\quad\br{\iota}_{j}\coloneqq\iota_{j}\iota_{j-1}\quad\text{ for }j\in\{2,\dots,J\}.
\]
and set 
\[
\br{\lmb}_{1}\coloneqq0,\quad\br{\lmb}_{J+1}\coloneqq\infty,\quad\mu_{1}\coloneqq0,\quad\mu_{J+1}\coloneqq0.
\]
\item For a function $\phi$, we denote 
\[
\phi_{;j}\coloneqq\iota_{j}\phi_{\lmb_{j}}\qquad\text{and}\qquad\phi_{\ul{;j}}\coloneqq\frac{1}{\lmb_{j}^{2}}\phi_{;j}.
\]
\item The crucial scalar quantity in our paper is 
\[
\calD\coloneqq\sum_{j=2}^{J}\frac{\mu_{j}^{2D}}{\lmb_{j}^{2}}.
\]
\item For $\alp\in(0,\infty)$, the space of decoupled scales (by $\alp$)
is denoted by 
\[
\calP_{J}(\alp)\coloneqq\{\vec{\lmb}\in(0,\infty)^{J}:\max_{2\leq\ell\leq J}\mu_{\ell}<\alp\}.
\]
\end{itemize}

\subsubsection{\label{subsec:Universal-sequence}Sequences of universal constants
in main theorems}

The sequences $(\lmb_{j,L}^{\mathrm{ex}}(t))_{j\geq1}$, $(\alp_{j})_{j\geq1}$,
and $(\beta_{j})_{j\geq1}$ that appeared in Theorems~\ref{thm:main-HMHF}
and \ref{thm:main-NLH} are defined as follows. For $L\in(0,\infty)$
and an integer $j\geq1$, we have 
\begin{equation}
\lmb_{j,L}^{\mathrm{ex}}(t)\coloneqq L^{1+2\alpha_{j}}\cdot\frac{\beta_{j}}{t^{\alpha_{j}}},\qquad\alpha_{j}\coloneqq\frac{1}{2}\Big(\frac{D}{D-2}\Big)^{j-1}-\frac{1}{2},\label{eq:def-lmb-ex-ell}
\end{equation}
and the universal constants $\beta_{j}$ are defined by 
\begin{equation}
\beta_{1}\coloneqq1,\qquad\beta_{j}=\Big(\frac{\alpha_{j}}{|\kappa|}\Big)^{\frac{1}{D-2}}\beta_{j-1}^{\frac{D}{D-2}}\qquad\forall j\geq2\label{eq:def-beta_j}
\end{equation}
with 
\begin{equation}
\kappa\coloneqq\begin{cases}
{\displaystyle \frac{N-2}{2}\frac{W(0)\int_{0}^{\infty}W^{p}y^{N-1}dy}{\int_{0}^{\infty}(\Lmb W)^{2}y^{N-1}dy}>0} & \text{for NLH},\\
{\displaystyle -\frac{4\int_{0}^{\infty}(\Lmb Q)^{3}y^{D-1}dy}{\int_{0}^{\infty}(\Lmb Q)^{2}ydy}<0} & \text{for HMHF}.
\end{cases}\label{eq:def-kappa-explicit-int}
\end{equation}
Recall that $W(y)=(1+\frac{y^{2}}{N(N-2)})^{-\frac{N-2}{2}}$ and
$\Lmb W(y)=(\frac{N-2}{2}-\frac{y^{2}}{2N})(1+\frac{y^{2}}{N(N-2)})^{-\frac{N}{2}}$
for NLH, whereas $\Lmb Q(y)=\frac{2Dy^{D}}{1+y^{2D}}$ for HMHF. Using
our unified notation, these values are in fact equal to 
\begin{equation}
\kappa=-\frac{\lan y^{-2}f'(Q)W(0),\Lmb W\ran}{\|\Lmb W\|_{L^{2}}^{2}}.\label{eq:def-kappa-unified}
\end{equation}
For NLH, this follows from integrating by parts. For HMHF, this follows
from $\Lmb Q=D\sin(Q)$ and $W(0)=2$.

\subsection{\label{subsec:Strategy}Strategy of the proof}

We perform the modulation analysis of multi-bubble dynamics, starting
from the soliton resolution result (whose precise formulation is given
in Section~\ref{subsec:Statements-soliton-res}). As mentioned earlier,
the key ingredient of the proof, which allows us to improve soliton
resolution to the classification, is a \emph{monotonicity estimate}
near multi-bubbles\emph{. }More precisely, we will prove a coercivity
estimate for the dissipation term in the energy identity, in a vicinity
of multi-bubbles. This is not as obvious as in the one-bubble case;
the coercivity estimate will be restored after introducing \emph{modified
multi-bubble profiles}.\smallskip{}

\emph{1. Reduction to bubble tree classification.} With the unified
notation in Section~\ref{subsec:Unified-notation}, the equations
\eqref{eq:HMHF-equiv} and \eqref{eq:NLH-rad} are written as evolution
equations for \emph{radially symmetric }solutions in $\bbR^{N}$ 
\begin{equation}
\left\{ \begin{aligned}\rd_{t}u & =\Dlt u+r^{-(D+2)}f(v),\\
u(0,r) & =u_{0}(r)\in\bbR,
\end{aligned}
\quad(t,r)\in[0,T)\times(0,\infty)\right.\label{eq:main-evol-eq}
\end{equation}
where $\Dlt=\rd_{rr}+\frac{N-1}{r}\rd_{r}$ is the radial Laplacian
on $\bbR^{N}$,
\[
v=r^{D}u\qquad\text{and}\qquad f(v)=\begin{cases}
|v|^{p-1}v & \text{for NLH},\\
\frac{D^{2}}{2}(2v-\sin(2v)) & \text{for HMHF.}
\end{cases}
\]
Note that the constants $N$, $D$, and $p$ are always related by
\[
D=\frac{N-2}{2}\qquad\text{and}\qquad p=\frac{N+2}{N-2}.
\]

By soliton resolution (see Section~\ref{subsec:Statements-soliton-res})
and the invariance of $v\mapsto v\pm\pi$ for \eqref{eq:HMHF-equiv},
the proof of Theorems~\ref{thm:main-HMHF} and \ref{thm:main-NLH}
reduces to the following classifcation theorem of radial bubble trees.
Here, we denote 
\begin{equation}
\calH_{\ell,m}\coloneqq\{u=r^{-D}v:v\in\calE_{\ell,m}\},\label{eq:def-calH-ell-m}
\end{equation}
where $\ell,m\in\bbZ$ and $\calE_{\ell,m}$ was defined in \eqref{eq:def-calE-ellm}.
Note that $\calH_{0,0}$ is naturally identified with $\dot{H}_{\rad}^{1}(\bbR^{N})$.
\begin{thm}[Classification of radial bubble trees for $D>2$]
\label{thm:main}Let $D>2$ be an integer or a half-integer and $m=0$
for \eqref{eq:NLH-rad}; let $D>2$ be an integer and $m\in\bbZ$
for \eqref{eq:HMHF-equiv}. Let $u(t)$ be a radial $\calH_{0,m}$-solution
to \eqref{eq:main-evol-eq} with $T\in(0,+\infty]$ its maximal time
of existence.

Assume there exist an integer $J\geq1$, signs $\iota_{1},\dots,\iota_{J}\in\{\pm1\}$,
continuous functions $\lmb_{1}(t),\dots,\lmb_{J}(t)\in(0,\infty)$
with $\lim_{t\to T}\lmb_{1}(t)=0$ if $T<+\infty$, and a function
$z^{\ast}\in\calH_{0,m^{\ast}}$ with $m^{\ast}=0$ for \eqref{eq:NLH-rad}
and $m^{\ast}=m-\sum_{j=1}^{J}\iota_{j}$ for \eqref{eq:HMHF-equiv}
if $T<+\infty$, and $z^{\ast}\equiv0$ if $T=+\infty$ such that
\begin{equation}
\Big\| u(t)-\sum_{j=1}^{J}\iota_{j}W_{\lmb_{j}(t)}-z^{\ast}\Big\|_{\dot{H}^{1}}+\sum_{j=2}^{J}\frac{\lmb_{j}(t)}{\lmb_{j-1}(t)}\to0\qquad\text{as }t\to T.\label{eq:main-thm-decom}
\end{equation}

Then, we have the following. (i) $T=+\infty$, (ii) $\iota_{j}=(-1)^{j}\iota$
for some $\iota\in\{\pm1\}$ for NLH, and $\iota_{j}=\mathrm{sgn}(m)$
and $J=|m|$ for \eqref{eq:HMHF-equiv}, and (iii) there exists $L_{\infty}\in(0,\infty)$
such that
\begin{equation}
\lmb_{j}(t)=\lmb_{j,L_{\infty}}^{\mathrm{ex}}(t)\cdot(1+o(1)),\qquad\forall j=1,\dots,J.\label{eq:main-thm-scale}
\end{equation}
\end{thm}

In what follows, we explain the strategy of the proof of Theorem~\ref{thm:main}.
We will focus on the \eqref{eq:NLH-rad} case.\smallskip{}

\emph{2. Recap of the proof in the one-bubble case.} We first recall
the proof in the simplest case $J=1$ for \eqref{eq:NLH-rad} (for
$N\geq7$) with $\|u(t)-W_{\lmb(t)}\|_{\dot{H}^{1}}<\dlt\ll1$, contained
in \cite{CollotMerleRaphael2017CMP}. From the energy identity and
$\dot{H}^{1}$-boundedness, we have an $L^{2}$-in-time control
\[
\int_{0}^{T}\|\rd_{t}u(t)\|_{L^{2}}^{2}dt=E[u(0)]-\lim_{t\to T}E[u(t)]<+\infty.
\]
Substituting the equation \eqref{eq:NLH-rad}, writing $u(t)=W_{\lmb(t)}+g(t)$,
and using the soliton equation $\Dlt W_{\lmb}+f(W_{\lmb})=0$, the
$L_{t}^{2}$-control reads 
\[
\int_{0}^{T}\big\|-H_{\lmb(t)}g(t)+\NL_{W_{\lmb(t)}}(g(t))\big\|_{L^{2}}^{2}dt<+\infty,
\]
where $H_{\lmb}=-\Dlt-f'(W_{\lmb})$ is the linearized operator around
$W_{\lmb}$ and $\NL_{W_{\lmb}}(g)=f(W_{\lmb}+g)-f(W_{\lmb})-f'(W_{\lmb})g$
is the nonlinear term. By standard modulation theory, we may assume
that $g(t)$ satisfies some orthogonality condition which guarantees
the coercivity $\|H_{\lmb}g\|_{L^{2}}\ageq\|g\|_{\dot{H}^{2}}$. The
nonlinear term can be ignored since $\|g(t)\|_{\dot{H}^{1}}$ is small.
As a result, we have an $L_{t}^{2}$-\emph{spacetime estimate} (or,
a monotonicity estimate) for $g(t)$:
\[
\int_{0}^{T}\|g(t)\|_{\dot{H}^{2}}^{2}dt<+\infty.
\]
On the other hand, a standard modulation estimate using $r\Lmb W\in L^{2}$
(if and only if $N>6$) gives 
\[
\Big|\frac{\lmb_{t}}{\lmb}\Big|\aleq\|g(t)\|_{\dot{H}^{2}}^{2}
\]
after possibly modifying $\lmb(t)$. The previous two displays then
imply $\lmb(t)\to L_{\infty}\in(0,\infty)$ as $t\to T$. We also
have $T=+\infty$ because $\lmb(t)\not\to0$. Here, the spacetime
estimate is indispensable as it guarantees that $g(t)$ does not contribute
to the modulation dynamics. See also \cite{GustafsonNakanishiTsai2010CMP}.

\smallskip{}

\emph{3. Modified multi-bubble profiles and spacetime estimate.} To
prove Theorem~\ref{thm:main}, we will generalize the previous proof
to the multi-bubble setting. The key observation will be that after
\emph{modifying} pure multi-bubble profiles, say from $\sum_{j=1}^{J}W_{;j}$
to $U(\vec{\iota},\vec{\lmb};\cdot)$, one can restore the \emph{$L_{t}^{2}$-spacetime
estimate} for the remainder $g(t)$ of the decomposition 
\[
u(t)=U(\vec{\iota},\vec{\lmb}(t);\cdot)+g(t).
\]

A naive use of the decomposition $u(t)=\sum_{j=1}^{J}W_{;j}(t)+g(t)$
encounters a difficulty because any pure multi-bubble with $J\geq2$
cannot be a static solution. There is always a nontrivial inhomogeneous
error term in $\rd_{t}u$, say $\rd_{t}u=\Psi_{0}-H_{\vec{\lmb}}g+N(g)$
with $\Psi_{0}=\Dlt(\sum_{j=1}^{J}W_{;j})+f(\sum_{j=1}^{J}W_{;j})$
and $H_{\vec{\lmb}}g=-\Dlt g-\sum_{j=1}^{J}f'(W_{;j})g$, preventing
us to say that the linear term $H_{\vec{\lmb}}g$ is dominant in the
right hand side. A simple but crucial observation to get around this
difficulty is the following; by modifying the pure multi-bubble profiles
(from $\sum_{j=1}^{J}W_{;j}$ to $U$), one can transform $\Psi_{0}$
into $\Psi_{1}=\Dlt U+f(U)$ that is \emph{almost orthogonal} to $H_{\vec{\lmb}}g$.
Once we have $\Psi_{1}$ orthogonal to $H_{\vec{\lmb}}g$, we have
\begin{equation}
\|\rd_{t}u\|_{L^{2}}^{2}\approx\|\Psi_{1}-H_{\vec{\lmb}}g\|_{L^{2}}^{2}\ageq\|\Psi_{1}\|_{L^{2}}^{2}+\|H_{\vec{\lmb}}g\|_{L^{2}}^{2}.\label{eq:strat-3}
\end{equation}
One has $\|H_{\vec{\lmb}}g\|_{L^{2}}\aeq\|g\|_{\dot{H}^{2}}$ by imposing
suitable orthogonality conditions on $g(t)$ on each scale $\lmb_{j}(t)$.
Moreover, it turns out that (see \eqref{eq:prof-eq-main-size})
\begin{equation}
\|\Psi_{1}\|_{L^{2}}^{2}\aeq\sum_{j=2}^{J}\frac{\mu_{j}^{2D}}{\lmb_{j}^{2}}\eqqcolon\calD.\label{eq:strat-8}
\end{equation}
Therefore, in view of the energy identity, \emph{\eqref{eq:strat-3}
not only recovers the spacetime estimate for $g(t)$, but also gives
an additional control from the size of $\Psi_{1}$: }(see \eqref{eq:sptime-est})\emph{
\begin{equation}
\int_{0}^{T}\Big\{\calD+\|g(t)\|_{\dot{H}^{2}}^{2}\Big\} dt<+\infty.\label{eq:strat-4}
\end{equation}
}Note that, in the finite-time blow-up case where $g$ is not assumed
to be $\dot{H}^{1}$-small, we need an additional localization argument.

Now, the proof of the spacetime estimate reduces to the construction
of modified multi-bubble profiles $U$ such that $\Psi_{1}$ is almost
orthogonal to $H_{\vec{\lmb}}g$. In the one-bubble case, some earlier
works such as \cite{MerleRaphael2003GAFA,RaphaelSzeftel2011JAMS}
already suggest how one can modify the profile; in principle, one
can transform $\Psi_{0}$ into $\Psi_{1}$ so that $\Psi_{1}$ belongs
to an approximate kernel of the linearized operator. In our case,
carefully performing the one-bubble analysis on each bubble combined
with induction, we obtain 
\begin{equation}
\Dlt U+f(U)=-\sum_{j=2}^{J}\frac{c_{j}}{\lmb_{j}^{2}}\Lmb W_{;j}+\text{(small error)}\label{eq:strat-5}
\end{equation}
for some reals $c_{j}$ in terms of scales. This will be sufficient
because $\Lmb W_{;j}$ is an approximate kernel, and hence it is almost
orthogonal to $H_{\vec{\lmb}}g$.

In our construction of $U$, a simple linear expansion $U=\sum_{j=1}^{J}[W+\chi_{j}c_{j}S]_{;j}$
for some fixed profile $S$ and cutoffs $\chi_{j}$, which was sufficient
for constructing bubble trees in \cite{delPinoMussoWei2021AnalPDE},
\emph{does not work for classification} in large dimensions; see Remark~\ref{rem:7D-simpler-profile}.
We perform a more refined construction using recursion on $J$, in
a similar spirit of elliptic bubble tree constructions such as in
\cite{IacopettiVaira2016} for the Brezis--Nirenberg problem. See
also \cite{DengSunWei2025Duke}.\smallskip{}

\emph{4. Modulation dynamics.} Once the construction of $U$ and the
$L_{t}^{2}$-spacetime estimate of $g(t)$ are established, it remains
to study the dynamics of the scales $\lmb_{1},\dots,\lmb_{J}$ and
determine the signs $\iota_{1},\dots,\iota_{J}$. At the formal level,
$c_{j}/\lmb_{j}^{2}$ in \eqref{eq:strat-5} corresponds to the variation
$(\lmb_{j})_{t}/\lmb_{j}$ and the leading order term of $c_{j}$
can be computed by $c_{j}\approx-\lan\Psi_{0},\Lmb W_{;j}\ran/\|\Lmb W\|_{L^{2}}^{2}$.
One arrives at the formal ODE system
\begin{equation}
\frac{\lmb_{j,t}}{\lmb_{j}}\approx\kappa\frac{\br{\iota}_{j}\mu_{j}^{D}}{\lmb_{j}^{2}}.\label{eq:strat-6}
\end{equation}
Next, we recall that the scales $\lmb_{j}(t)$ satisfy the assumptions
in Theorem~\ref{thm:main}, i.e., 
\begin{equation}
\sum_{j=2}^{J}\mu_{j}(t)+\chf_{T<+\infty}\lmb_{1}(t)\to0.\label{eq:strat-7}
\end{equation}
In the formal ODE system \eqref{eq:strat-6}, we see that $\lmb_{1}$
is constant, so $\lmb_{1}\equiv L_{\infty}\in(0,\infty)$. This excludes
finite-time blow-up by \eqref{eq:strat-7}. Next, as $\kappa>0$ for
\eqref{eq:NLH-rad}, we obtain $\br{\iota}_{2}=\iota_{1}\iota_{2}=-1$
in order to have $\mu_{2}=\lmb_{2}/\lmb_{1}\to0$. With $\br{\iota}_{2}=-1$
and $\lmb_{1}\equiv L_{\infty}$, one can integrate \eqref{eq:strat-6}
for $j=2$ to obtain the asymptotics of $\lmb_{2}(t)$. This argument
can be done for all $j$ by induction.

The actual ODE system (see e.g., Proposition~\ref{prop:Modulation-Est})
has error terms. Now, the spacetime estimate \eqref{eq:strat-4} is
crucial, as it guarantees that the error terms from $g(t)$ are negligible.
Dealing with some other error terms involving different $\mu_{k}$s
requires more clever tricks. We first observe that the errors involving
high powers of $\mu_{k}$s are negligible, thanks to the additional
control in \eqref{eq:strat-4}. Having simplified the error terms,
we further introduce modified scaling parameters (which might be viewed
as sub- and supersolutions to ODE) that enjoy certain monotonic properties,
in order to integrate the ODE system; see Step 1 in Section~\ref{subsec:Integration}.

Finally, we remark that the harmonic map heat flow case is similar.
The only difference in the ODE integration is that now we have $\kappa<0$.
This gives $\br{\iota}_{j}=1$ saying that all $\iota_{j}$s are the
same. The topological restriction $u(t)\in\calE_{0,m}$ then determines
the sign and the total number of bubbles.\medskip{}

\noindent \mbox{\textbf{Organization of the paper.} }We gather some
preliminaries in Section~\ref{sec:Preliminaries}. We construct the
modified multi-bubble profiles $U$ in Section~\ref{sec:Modified-multi-bubble-profiles}
and use them to prove the key spacetime estimate in Section~\ref{sec:Spacetime-estimate}.
Finally in Section~\ref{sec:Proof-of-main-thm}, we prove modulation
estimates and integrate them to complete the proof of the main theorems.\medskip{}

\noindent \mbox{\textbf{Acknowledgements.} }K. Kim was supported
by Huawei Young Talents Programme at IHES. K. Kim was supported by
the New Faculty Startup Fund from Seoul National University, the POSCO
Science Fellowship of POSCO TJ Park Foundation, and the National Research
Foundation of Korea (NRF) grant funded by the Korea government (MSIT)
RS-2025-00523523. The authors thank anonymous referees for their careful
reading and valuable comments that improve the presentation of this
manuscript.

\section{\label{sec:Preliminaries}Preliminaries}

\subsection{\label{subsec:Statements-soliton-res}Statements of bubble decomposition
and soliton resolution}

Here, we gather the precise statements of bubble decomposition for
equivariant solutions to \eqref{eq:HMHF-equiv} and soliton resolution
for $\dot{H}^{1}$-bounded radial solutions to \eqref{eq:NLH-rad},
which will be the starting point of our analysis.
\begin{prop}[Bubble decomposition for equivariant solutions; see \cite{JendrejLawrie2023CVPDE}]
\label{prop:HMHF-bubble-dec}Let $D\in\bbZ_{\geq1}$. Let $\ell,m\in\bbZ$
and $v_{0}\in\calE_{\ell,m}$ be an initial data; let $u(t)$ be the
corresponding solution to \eqref{eq:HMHF-equiv} with $T\in(0,+\infty]$
its maximal time of existence.

\emph{(Finite-time blow-up case)} If $T<+\infty$, then there exist
an integer $J\geq1$, signs $\iota_{1},\dots,\iota_{J}\in\{\pm1\}$,
continuous functions $\lmb_{1}(t),\dots,\lmb_{J}(t)\in(0,\infty)$,
and a function $z^{\ast}\in\calE_{0,m^{\ast}}$ with $m^{\ast}=m-\ell-\sum_{j=1}^{J}\iota_{j}$
such that 
\begin{equation}
\Big\| v(t)-\ell\pi-\sum_{j=1}^{J}\iota_{j}Q_{\lmb_{j}(t)}-z^{\ast}\Big\|_{\calE}+\sum_{j=2}^{J}\frac{\lmb_{j}(t)}{\lmb_{j-1}(t)}+\frac{\lmb_{1}(t)}{\sqrt{T-t}}\to0\qquad\text{as }t\to T.\label{eq:bubble-dec-HMHF-finite}
\end{equation}

\emph{(Global case)} If $T=+\infty$, then there exist an integer
$J\geq0$, signs $\iota_{1},\dots,\iota_{J}\in\{\pm1\}$, and continuous
functions $\lmb_{1}(t),\dots,\lmb_{J}(t)\in(0,\infty)$ such that
\begin{equation}
\Big\| v(t)-\ell\pi-\sum_{j=1}^{J}\iota_{j}Q_{\lmb_{j}(t)}\Big\|_{\calE}+\sum_{j=2}^{J}\frac{\lmb_{j}(t)}{\lmb_{j-1}(t)}+\frac{\lmb_{1}(t)}{\sqrt{t}}\to0\qquad\text{as }t\to T.\label{eq:bubble-dec-HMHF-infinite}
\end{equation}
\end{prop}

As mentioned in the introduction, the following statement of soliton
resolution for $\dot{H}^{1}$-bounded solutions to \eqref{eq:NLH-rad}
was announced by Matano and Merle and the proof was presented at a
conference at IHES. However, the result did not appear in a paper
form. We state the precise version we use in this paper.
\begin{question}[Soliton resolution for $\dot{H}^{1}$-bounded radial solutions]
\label{conj:radial-soliton-resolution}Let $N\geq3$. Let $u(t)$
be a radially symmetric $\dot{H}^{1}$-solution to \eqref{eq:NLH-rad}
with $T\in(0,+\infty]$ its maximal time of existence. Assume $\sup_{t\in[0,T)}\|u(t)\|_{\dot{H}^{1}}<+\infty$.

\emph{(Finite-time blow-up case)} If $T<+\infty$, then there exist
an integer $J\geq1$, signs $\iota_{1},\dots,\iota_{J}\in\{\pm1\}$,
continuous functions $\lmb_{1}(t),\dots,\lmb_{J}(t)\in(0,\infty)$,
and a function $z^{\ast}\in\dot{H}_{\rad}^{1}$ such that 
\begin{equation}
\Big\| u(t)-\sum_{j=1}^{J}\iota_{j}W_{\lmb_{j}(t)}-z^{\ast}\Big\|_{\dot{H}^{1}}+\sum_{j=2}^{J}\frac{\lmb_{j}(t)}{\lmb_{j-1}(t)}+\frac{\lmb_{1}(t)}{\sqrt{T-t}}\to0\qquad\text{as }t\to T.\label{eq:sol-res-NLH-finite}
\end{equation}

\emph{(Global case)} If $T=+\infty$, then there exist an integer
$J\geq0$, signs $\iota_{1},\dots,\iota_{J}\in\{\pm1\}$, and continuous
functions $\lmb_{1}(t),\dots,\lmb_{J}(t)\in(0,\infty)$ such that
\begin{equation}
\Big\| u(t)-\sum_{j=1}^{J}\iota_{j}W_{\lmb_{j}(t)}\Big\|_{\dot{H}^{1}}+\sum_{j=2}^{J}\frac{\lmb_{j}(t)}{\lmb_{j-1}(t)}+\frac{\lmb_{1}(t)}{\sqrt{t}}\to0\qquad\text{as }t\to T.\label{eq:sol-res-NLH-infinite}
\end{equation}
\end{question}

As mentioned in Remark~\ref{rem:Open-Question-resolved}, Open Question~\ref{conj:radial-soliton-resolution}
\emph{was resolved affirmatively} in the paper \cite{Aryan2024arXiv}
one month later we posted our article on arXiv.

\subsection{Basic pointwise estimates}

We begin with some pointwise estimates for the nonlinearity. Recall
that $f(v)=|v|^{p-1}v$ for \eqref{eq:NLH-rad} and $f(v)=\frac{D^{2}}{2}(2v-\sin(2v))$
for \eqref{eq:HMHF-equiv}. We record the following pointwise estimates
regarding $f(v)$.
\begin{lem}
For $f(v)=|v|^{p-1}v$ with $1<p\leq2$, we have 
\begin{align}
|f(a+b)-f(a)-f'(a)b| & \aleq|a|^{p-2}|b|^{2},\qquad a\neq0,\label{eq:f-06}\\
|f'(a+b)-f'(a)| & \aleq|a|^{p-2}|b|,\qquad a\neq0,\label{eq:f-02}\\
|f'(a+b)-f'(a)| & \aleq|b|^{p-1},\label{eq:f-03}\\
|f(a+b)-f(a)-f'(a)b| & \aleq|b|^{p},\label{eq:f-04}\\
|f(a+b)-f(a)-f(b)| & \aleq|b|^{p-1}|a|,\label{eq:f-05}\\
|f'(\tsum{j=1}Ja_{j})-\tsum{j=1}Jf'(a_{j})| & \aleq_{J}\tsum{j,\ell=1}J\chf_{j\neq\ell}|a_{j}a_{\ell}|^{\frac{p-1}{2}}.\label{eq:f-07}
\end{align}
For $f(v)=\frac{D^{2}}{2}(2v-\sin(2v))$, we have 
\begin{align}
|f(a+b)-f(a)-f'(a)b| & \aleq|b|^{2},\qquad|b|\aleq1,\label{eq:f-06-HM}\\
|f'(a+b)-f'(a)| & \aleq|\sin b|,\label{eq:f-03-HM}\\
|f(a+b)-f(a)-f(b)| & \aleq|\sin a||\sin b|,\label{eq:f-05-HM}\\
|f'(\tsum{j=1}Ja_{j})-\tsum{j=1}Jf'(a_{j})| & \aleq_{J}\tsum{j,\ell=1}J|\sin a_{j}||\sin a_{\ell}|.\label{eq:f-07-HM}
\end{align}
\end{lem}

\begin{rem}
Sharper estimates are possible for \eqref{eq:HMHF-equiv}. However,
such improvements are not necessary because the estimates for \eqref{eq:HMHF-equiv}
are already better than those for \eqref{eq:NLH-rad} under $|a|,|b|\aleq1$.
\end{rem}

\begin{proof}
We only show the estimates for \eqref{eq:HMHF-equiv}. For the proof
of the estimates for \eqref{eq:NLH-rad}, one may assume $a=1$ and
Taylor expand the function $x\mapsto f(1+x)$ for \eqref{eq:f-06}--\eqref{eq:f-05},
and one may use induction to prove \eqref{eq:f-07}. See also \cite[Appendix D]{CollotMerleRaphael2017CMP}.

We turn to the proof of \eqref{eq:f-06-HM}--\eqref{eq:f-07-HM}.
\eqref{eq:f-06-HM} follows from 
\begin{align*}
 & \sin(2a+2b)-\sin(2a)-\cos(2a)2b\\
 & \quad=\sin(2a)(\cos(2b)-1)+\cos(2a)(\sin(2b)-2b)\aleq|b|^{2},\qquad|b|\aleq1.
\end{align*}
\eqref{eq:f-03-HM} follows from 
\[
\cos(2a+2b)-\cos(2a)=-2\sin(2a+b)\sin b.
\]
\eqref{eq:f-05-HM} follows from 
\begin{align*}
 & \sin(2a+2b)-\sin(2a)-\sin(2b)\\
 & \quad=2\sin b\cos(2a+b)-2\sin b\cos b=-4\sin(a+b)\sin a\sin b.
\end{align*}
\eqref{eq:f-07-HM} follows from 
\begin{align*}
 & \{\cos(2a+2b)-1\}-\{\cos(2a)-1\}-\{\cos(2b)-1\}\\
 & \quad=(1-\cos(2a))(1-\cos(2b))-\sin(2a)\sin(2b)=-4\cos(a+b)\sin a\sin b
\end{align*}
and induction. This completes the proof.
\end{proof}
\begin{lem}
For any $\lmb\in(0,\infty)$, we have 
\begin{align*}
Q_{\lmb} & \aleq1,\\
r^{-2}f'(Q_{\lmb}) & \aleq[\lan y\ran^{-2D}]_{\lmb}^{p-1}=\lmb^{-2}\lan r/\lmb\ran^{-4},\\
W_{\lmb},\Lmb W_{\lmb} & \aleq[\lan y\ran^{-2D}]_{\lmb}=\lmb^{-D}\lan r/\lmb\ran^{-2D}\qquad\text{ for NLH},\\
r^{-D}\sin Q_{\lmb}\aeq\Lmb W_{\lmb} & \aeq[\lan y\ran^{-2D}]_{\lmb}=\lmb^{-D}\lan r/\lmb\ran^{-2D}\qquad\text{ for HMHF.}
\end{align*}
\end{lem}

\begin{proof}
By scaling, we may assume $\lmb=1$. All the estimates then follow
from explicit formulas of $W$, $Q=y^{-D}W$, and $D(p-1)=2$.
\end{proof}

\subsection{\label{subsec:Formal-inversion}Formal inversion of $H$}

When we construct modified multi-bubble profiles in Section~\ref{sec:Modified-multi-bubble-profiles},
we need to invert the linearized operator 
\[
H\coloneqq-\Dlt-y^{-2}f'(Q).
\]
Under radial symmetry, the $\dot{H}^{1}$-kernel of $H$ is spanned
by $\Gmm_{1}=\Lmb W$ and a singular kernel element $\Gmm_{2}$ can
be found using the Wronskian formula: 
\begin{equation}
\Gamma_{2}(y)=\Lmb W\int_{1}^{y}\frac{dy'}{(y')^{N-1}(\Lmb W)^{2}}.\label{eq:Gmm2-def}
\end{equation}
For \eqref{eq:HMHF-equiv}, this formula defines a smooth global solution
$\Gmm_{2}$ on $(0,\infty)$. For \eqref{eq:NLH-rad}, since $\Lmb W$
vanishes at $y=y_{\ast}\coloneqq\sqrt{N(N-2)}$, we first define $\Gmm_{2}(y)$
using \eqref{eq:Gmm2-def} only for small $y$, say $y\in(0,y_{\ast})$.
Then, by the standard ODE theory, $\Gmm_{2}$ extends to a smooth
global solution on $(0,\infty)$. In view of Wronskian formula, we
have 
\[
\Gmm_{1}(\rd_{y}\Gmm_{2})-\Gmm_{2}(\rd_{y}\Gmm_{1})=\frac{1}{y^{N-1}}.
\]

Next, we claim the following bounds for any $k\in\bbN$:
\begin{equation}
\begin{aligned}|\Gmm_{1}|_{k} & \aleq_{k}\chf_{(0,1]}+\chf_{[1,\infty)}y^{-(N-2)},\\
|\Gmm_{2}|_{k} & \aleq_{k}\chf_{(0,1]}y^{-(N-2)}+\chf_{[1,\infty)}.
\end{aligned}
\label{eq:Gmm-ptwise}
\end{equation}
The bound for $\Gmm_{1}$ easily follows from the explicit formula
$\Gmm_{1}=\Lmb W$. For $\Gmm_{2}$, the estimate near the origin
(i.e., $y\ll1$) follows from the integral formula. To show the estimate
for large $y$, one notices that a variant of \eqref{eq:Gmm2-def},
after replacing $\int_{1}^{y}$ to $\int_{2y_{\ast}}^{y}$, defines
another smooth solution $\Gmm_{3}$ in the region $(y_{\ast},\infty)$.
Note that $|\Gmm_{3}|_{k}\aeq_{k}1$ for large $y$. In particular,
$\Gmm_{1}$ and $\Gmm_{3}$ are linearly independent, so $\Gmm_{2}$
is a linear combination of $\Gmm_{1}$ and $\Gmm_{3}$. This gives
the desired bound of $\Gmm_{2}$ for large $y$.

Finally, we define a right inverse of $H$ by 
\begin{equation}
[\out H^{-1}F](y)\coloneqq\Gmm_{1}\int_{0}^{y}\Gmm_{2}F\,(y')^{N-1}dy'-\Gmm_{2}\int_{0}^{y}\Gmm_{1}F\,(y')^{N-1}dy'.\label{eq:H-inv-def}
\end{equation}
When $\lan F,\Lmb W\ran=0$, $-\int_{0}^{y}$ in the second integral
can be replaced by $\int_{y}^{\infty}$: 
\begin{equation}
\out H^{-1}F=\Gmm_{1}\int_{0}^{y}\Gmm_{2}F\,(y')^{N-1}dy'+\Gmm_{2}\int_{y}^{\infty}\Gmm_{1}F\,(y')^{N-1}dy',\quad\lan F,\Lmb W\ran=0.\label{eq:H-inv-2}
\end{equation}
For any $\lmb\in(0,\infty)$, the formal right inverse of $H_{\lmb}$
is defined by the scaling relation 
\begin{equation}
\out H_{\lmb}^{-1}F_{\ul{\lmb}}=[\out H^{-1}F]_{\lmb}.\label{eq:H-lmb-inv-rel}
\end{equation}

\subsection{Coercivity estimate for $H_{\vec{\protect\lmb}}$}

We work with function spaces $\dot{H}_{\rad}^{1}$ and $\dot{H}_{\rad}^{2}$.
The initial data belong to $\dot{H}_{\rad}^{1}$ for \eqref{eq:NLH-rad},
while for \eqref{eq:HMHF-equiv} the data after subtracting the bubbles
and the asymptotic profile belong to $\dot{H}_{\rad}^{1}$. The latter
space $\dot{H}_{\rad}^{2}$ is naturally associated to the dissipation
term of the energy identity. As $N>4$, we have the inequalities of
Hardy and Rellich:
\begin{align}
\|u\|_{\dot{H}^{1}} & \ageq\|r^{-1}u\|_{L^{2}},\label{eq:Hardy}\\
\|u\|_{\dot{H}^{2}} & \ageq\|r^{-2}u\|_{L^{2}}.\label{eq:Rellich}
\end{align}

The coercivity estimate for the dissipation term will be crucial in
our analysis. First, we recall the well-known one-bubble case; for
the proof, see \cite[Appendix B]{CollotMerleRaphael2017CMP} and \cite[Appendix B]{RaphaelRodnianski2012Publ.Math.}
for \eqref{eq:NLH-rad} and \eqref{eq:HMHF-equiv}, respectively.
\begin{lem}[Coercivity estimate for $H$]
Let $\calZ\in C_{c}^{\infty}(\bbR^{N})$ be a fixed radial profile
with $\lan\calZ,\Lmb W\ran\neq0$. Then, we have 
\begin{equation}
\|Hg\|_{L^{2}}\aeq\|g\|_{\dot{H}^{2}},\qquad\forall g\in\dot{H}_{\rad}^{2}\cap\{\calZ\}^{\perp},\label{eq:H-lin-coer}
\end{equation}
where $H=-\Dlt-y^{-2}f'(Q)$ and $\perp$ is defined with respect
to the inner product $\lan\cdot,\cdot\ran$.
\end{lem}

Building upon this one-bubble case coercivity estimate, we obtain
a multi-bubble generalization of the coercivity estimate by a soft
argument.
\begin{lem}[Coercivity estimate for $H_{\vec{\lmb}}$]
Let $\calZ\in C_{c}^{\infty}(\bbR^{N})$ be a fixed radial profile
with $\lan\calZ,\Lmb W\ran\neq0$. For any $J\geq1$, there exists
$\alp_{\coer}>0$ such that the following estimate holds whenever
$\vec{\lmb}\in\calP_{J}(\alp_{\coer})$: 
\begin{equation}
\|H_{\vec{\lmb}}g\|_{L^{2}}\aeq\|g\|_{\dot{H}^{2}},\qquad\forall g\in\dot{H}_{\rad}^{2}\cap\{\calZ_{\lmb_{1}},\dots,\calZ_{\lmb_{J}}\}^{\perp},\label{eq:lin-coer}
\end{equation}
where we denoted $H_{\vec{\lmb}}=-\Dlt-r^{-2}\sum_{j=1}^{J}f'(Q_{;j})$
and $\perp$ is defined with respect to the inner product $\lan\cdot,\cdot\ran$.
\end{lem}

\begin{proof}
Suppose not. Then, there exist $J\geq1$, sequences $\alp_{n}\to0$,
$\vec{\lmb}^{(n)}\in\calP_{J}(\alp_{n})$, and $g_{n}\in\dot{H}_{\rad}^{2}\cap\{\calZ_{\lmb_{1}^{(n)}},\dots,\calZ_{\lmb_{J}^{(n)}}\}^{\perp}$
such that $\|g_{n}\|_{\dot{H}^{2}}=1$ and $\|H_{\vec{\lmb}^{(n)}}g_{n}\|_{L^{2}}\to0$.

We claim that there exists $A>1$ such that for all large $n$ 
\begin{equation}
\max_{1\leq k\leq J}\|\chf_{[A^{-1}\lmb_{k}^{(n)},A\lmb_{k}^{(n)}]}r^{-2}g_{n}\|_{L^{2}}\aeq1.\label{eq:subcoer}
\end{equation}
Indeed, as $f'(Q)\aeq\chf_{(0,1]}y^{2}+\chf_{[1,+\infty)}y^{-2}$,
we have 
\[
\chf_{(0,A^{-1}\lmb_{k}^{(n)}]}f'(Q_{;k})+\chf_{[A\lmb_{k}^{(n)},+\infty)}f'(Q_{;k})\aleq A^{-2}.
\]
Therefore, using $\|g_{n}\|_{\dot{H}^{2}}=1$, $\|H_{\vec{\lmb}^{(n)}}g_{n}\|_{L^{2}}\to0$,
and Rellich's inequality \eqref{eq:Rellich}, we have 
\begin{align*}
1=\|g_{n}\|_{\dot{H}^{2}} & \leq\|H_{\vec{\lmb}^{(n)}}g_{n}\|_{L^{2}}+\tsum{k=1}J\|r^{-2}f'(Q_{;k})g_{n}\|_{L^{2}}\\
 & \aleq o_{n\to\infty}(1)+\max_{1\leq k\leq J}\|\chf_{[A^{-1}\lmb_{k}^{(n)},A\lmb_{k}^{(n)}]}r^{-2}g_{n}\|_{L^{2}}+A^{-2}\|r^{-2}g_{n}\|_{L^{2}}\\
 & \aleq o_{n\to\infty}(1)+o_{A\to\infty}(1)+\max_{1\leq k\leq J}\|\chf_{[A^{-1}\lmb_{k}^{(n)},A\lmb_{k}^{(n)}]}r^{-2}g_{n}\|_{L^{2}}.
\end{align*}
Taking $A>1$ large, \eqref{eq:subcoer} follows.

By \eqref{eq:subcoer}, there exists a sequence $k_{n}\in\{1,\dots,J\}$
such that $\|\chf_{r\aeq\lmb_{k_{n}}^{(n)}}r^{-2}g_{n}\|_{L^{2}}\aeq1$.
We renormalize to the scale $r\aeq\lmb_{k_{n}}^{(n)}$; define $h_{n}\in\dot{H}_{\rad}^{2}$
and $\vec{\nu}^{(n)}\in\calP_{J}(\alp_{n})$ such that $g_{n}=(\lmb_{k_{n}}^{(n)})^{-1}[h_{n}]_{\lmb_{k_{n}}^{(n)}}$
and $\vec{\lmb}^{(n)}=\lmb_{k_{n}}^{(n)}\vec{\nu}^{(n)}$. Note that
we have 
\[
\|h_{n}\|_{\dot{H}^{2}}=1,\quad\|H_{\vec{\nu}^{(n)}}h_{n}\|_{L^{2}}\to0,\quad\|\chf_{y\aeq1}h_{n}\|_{L^{2}}\aeq1,\quad\lan h_{n},\calZ\ran=0.
\]
Passing to a subsequence, let $h\in\dot{H}_{\rad}^{2}$ be a weak
$\dot{H}^{2}$-limit of $h_{n}$. By the Rellich--Kondrachov theorem,
we may assume $h_{n}\to h$ strongly in $L_{\loc}^{2}$. In particular,
we have $h\neq0$ and $\lan h,\calZ\ran=0$. We claim that 
\begin{equation}
Hh=0.\label{eq:Hh=00003D0}
\end{equation}
If this claim is true, we get a contradiction to \eqref{eq:H-lin-coer}.

It remains to show \eqref{eq:Hh=00003D0}. It suffices to show $\lan Hh,\vphi\ran=0$
for all $\vphi\in C_{c}^{\infty}(0,\infty)$. As $h_{n}\weakto h$
weakly in $\dot{H}^{2}$ and $\|H_{\vec{\nu}^{(n)}}h_{n}\|_{L^{2}}\to0$,
we have 
\begin{align*}
\lan Hh,\vphi\ran=\lim_{n\to\infty}\lan Hh_{n},\vphi\ran & =\lim_{n\to\infty}\big\{\lan H_{\vec{\nu}^{(n)}}h_{n},\vphi\ran+\tsum{j\neq k_{n}}{}\lan y^{-2}f'(Q_{\nu_{j}^{(n)}})h_{n},\vphi\ran\big\}\\
 & =\lim_{n\to\infty}\tsum{j\neq k_{n}}{}\lan y^{-2}f'(Q_{\nu_{j}^{(n)}})h_{n},\vphi\ran.
\end{align*}
To estimate the last term, fix $R>1$ such that $\mathrm{supp}(\vphi)\subseteq[R^{-1},R]$.
From the pointwise bound $f'(Q_{\nu})\aleq\min\{\nu^{-2}y^{2},\nu^{2}y^{-2}\}$
and $\alp_{n}\to0$ (the ratio between scales), we obtain $\|\chf_{[R^{-1},R]}f'(Q_{\nu_{j}^{(n)}})\|_{L^{\infty}}\aleq o_{n\to\infty}(1)$.
Thus 
\[
\lan y^{-2}f'(Q_{\nu_{j}^{(n)}})h_{n},\vphi\ran\aleq o_{n\to\infty}(1)\|y^{-2}h_{n}\|_{L^{2}}\|\vphi\|_{L^{2}}\aleq o_{n\to\infty}(1),
\]
which gives \eqref{eq:Hh=00003D0}. This completes the proof.
\end{proof}

\section{\label{sec:Modified-multi-bubble-profiles}Modified multi-bubble
profiles}

As explained in Section~\ref{subsec:Strategy}, we need to construct
the modified multi-bubble profiles $U$ to decompose a nonlinear solution
$u(t)$ (near multi-bubbles) into $u(t)=U(t)+g(t)$ so that $g(t)$
enjoys an $L_{t}^{2}$-spacetime estimate. The main goal of this section
is the construction of $U$.
\begin{prop}[Modified profile]
\label{prop:Modified-Profile}There exist $0<\alp_{0}<\frac{1}{10}$
and $A_{0}>10$ such that for any $J\in\bbN$, $\iota_{1},\dots,\iota_{J}\in\{\pm1\}$,
and $(\lmb_{1},\dots,\lmb_{J})\in\calP_{J}(\alp_{0})$, there exist
modified profiles $U=U(\iota_{1},\dots,\iota_{J},\lmb_{1},\dots,\lmb_{J};r)$
with the following properties. (By Section~\ref{subsec:Notation},
the implicit constants here depend only on $N$ and $J$.)
\begin{itemize}
\item (Pointwise estimates for $U$) The profile $U$ takes the form
\begin{equation}
U=\sum_{j=1}^{J}W_{;j}+\td U.\label{eq:U-dec}
\end{equation}
For \eqref{eq:NLH-rad}, we have (recall $\br{\lmb}_{j}=\sqrt{\lmb_{j}\lmb_{j-1}}$)
\begin{align}
|U| & \aeq|W_{;j}| &  & \text{if }r\in[A_{0}\br{\lmb}_{j+1},A_{0}^{-1}\br{\lmb}_{j}],\label{eq:U-bd-1}\\
|U| & \aleq|W_{;j}| &  & \text{if }r\in[A_{0}^{-1}\br{\lmb}_{j+1},A_{0}\br{\lmb}_{j}],\label{eq:U-bd-2}\\
|U-W_{;j}| & \aleq|W_{;j+1}|+|W_{;j-1}| &  & \text{if }r\in[A_{0}^{-1}\br{\lmb}_{j+1},A_{0}\br{\lmb}_{j}].\label{eq:U-bd-3}
\end{align}
For \eqref{eq:HMHF-equiv}, we have (recall $P=r^{D}U$)
\begin{align}
r^{-D}\sin(P) & \aleq|[\lan y\ran^{-2D}]_{;j}| &  & \text{if }r\in[A_{0}^{-1}\br{\lmb}_{j+1},A_{0}\br{\lmb}_{j}],\label{eq:U-bd-2-HM}\\
r^{-D}\sin(P-Q_{;j}) & \aleq|[\lan y\ran^{-2D}]_{;j-1}|+|[\lan y\ran^{-2D}]_{;j+1}| &  & \text{if }r\in[A_{0}^{-1}\br{\lmb}_{j+1},A_{0}\br{\lmb}_{j}].\label{eq:U-bd-3-HM}
\end{align}
\item (Profile equation) $U$ solves the equation 
\begin{equation}
\Dlt U+r^{-(D+2)}f(P)=-\sum_{j=2}^{J}\br{\iota}_{j}\frac{\kappa\mu_{j}^{D}}{\lmb_{j}^{2}}\Lmb W_{;j}+\Psi\label{eq:def-Psi}
\end{equation}
and the main term has the size (recall $\calD=\sum_{j=2}^{J}\frac{\mu_{j}^{2D}}{\lmb_{j}^{2}}$)
\begin{equation}
\Big\|\sum_{j=2}^{J}\br{\iota}_{j}\frac{\kappa\mu_{j}^{D}}{\lmb_{j}^{2}}\Lmb W_{;j}\Big\|_{L^{2}}\aeq\sqrt{\calD}.\label{eq:prof-eq-main-size}
\end{equation}
\item (Estimates for the remainder terms $\td U$ and $\Psi$) Assume $(\lmb_{1},\dots,\lmb_{J})\in\calP_{J}(\alp^{\ast})$
for some $\alp^{\ast}\in(0,\alp_{0}]$. Recall that $\dlt(\alp^{\ast})$
denotes $o_{\alp^{\ast}\to0}(1)$. Then, $\td U$ satisfies the estimates
\begin{align}
\|\td U\|_{\dot{H}^{1}} & \aleq\dlt(\alp^{\ast}),\label{eq:tdU-H1}\\
\|\lmb_{j}\rd_{\lmb_{j}}\td U\cdot[\lan y\ran^{-2D}]_{\ul{;k}}\|_{L^{1}} & \aleq\chf_{j=k}(\mu_{k+1}^{D}+\mu_{k}^{D}|\log\mu_{k}|)\label{eq:tdU-weight-L1}\\
 & \peq+\chf_{j>k}\Big(\frac{\lmb_{j}}{\lmb_{k}}\Big)^{D}+\chf_{j<k}\dlt(\alp^{\ast})\Big(\frac{\lmb_{k}}{\lmb_{j}}\Big)^{D-2}\nonumber 
\end{align}
and $\Psi$ satisfies the estimates 
\begin{align}
\|\Psi\|_{L^{2}} & \aleq\dlt(\alp^{\ast})\sqrt{\calD},\label{eq:Psi-L2}\\
\|\Psi\cdot[\lan y\ran^{-2D}]_{\ul{;k}}\|_{L^{1}} & \aleq\frac{1}{\lmb_{k}^{2}}(\mu_{k+1}^{D}+\dlt(\alp^{\ast})\cdot\mu_{k}^{D})+\calD.\label{eq:Psi-inn}
\end{align}
\end{itemize}
\end{prop}

\begin{rem}
The estimate \eqref{eq:Psi-L2}, saying that the inhomogeneous error
is smaller than the main term \eqref{eq:prof-eq-main-size} in the
$L^{2}$ norm, is crucial in the proof of spacetime estimate \eqref{eq:sptime-est}.
The extra weighted $L^{1}$-estimates \eqref{eq:tdU-weight-L1} and
\eqref{eq:Psi-inn} are not fundamental in the proof of spacetime
estimate, but they will be used in the proof of modulation estimates
(Proposition~\ref{prop:Modulation-Est}) when we estimate certain
inner products.
\end{rem}

\begin{rem}
\label{rem:7D-simpler-profile}For \eqref{eq:NLH-rad} with $N=7$,
a linear approximation profile 
\[
\check{U}=W_{;1}+\sum_{j=2}^{J}[W+\br{\iota}_{j}\mu_{j}^{D}\chi_{1/\sqrt{\mu_{j}}}S]_{;j},
\]
with the corrector profile $T$ defined by 
\[
S\coloneqq\out H^{-1}\{W(0)\cdot f'(W)+\kappa\Lmb W\},
\]
can be used. This profile was also used in the bubble tree construction
of \cite{delPinoMussoWei2021AnalPDE} for all large dimensions $N\geq7$.
However, this profile turns out to be not sufficient (at least following
the method in this paper) for the classification of bubble tree dynamics
when $N\geq8$. This is because the inhomogeneous error $\check{\Psi}$
generated by $\check{U}$ at the scale $r\aeq\br{\lmb}_{j}$ will
not satisfy \eqref{eq:Psi-L2}, which is necessary in the proof of
the spacetime estimate \eqref{eq:sptime-est}. Not only one has to
change the cutoff radius to improve the profile error, but one also
needs to make more profile expansions taking into account the nonlinear
term, as the dimension gets larger.

See also \cite{DengSunWei2025Duke} for a different construction of
another modified profile in the non-radial case with non-negative
bubbles, and with different Lagrange multipliers in the profile equation.
Note that our Lagrange multiplier $\Lmb W$ is crucial to obtain the
spacetime estimate \eqref{eq:sptime-est} related to modulation theory
for the heat equation.
\end{rem}

\subsection{\label{subsec:Sketch-of-construction}Sketch of construction}

We will construct the profile $U$ by inducting on the number of bubbles,
i.e., on $J$. Our induction is bottom-up; we start from the lowest
bubble ($W_{;1}$) and add correction terms $\vphi_{2},\vphi_{3},\dots$
to the higher bubbles ($W_{;2},W_{;3},\dots$).

We sketch this inductive construction for \eqref{eq:NLH-rad}. \emph{The
discussion in this paragraph is formal}. When $J=1$, we set 
\[
U_{1}(\iota_{1},\lmb_{1};r)\coloneqq\iota_{1}W_{\lmb_{1}}=W_{;1}.
\]
Let $\alp_{0}\in(0,\frac{1}{10})$ be a small constant that will be
independent of $J$. For some $J\geq1$, suppose that for all $\ell=2,\dots,J$,
we have constructed 
\begin{align*}
\vphi_{\ell} & =\vphi_{\ell}(\iota_{1},\dots,\iota_{\ell},\lmb_{1},\dots,\lmb_{\ell};r),\\
U_{\ell} & =U_{\ell}(\iota_{1},\dots,\iota_{\ell},\lmb_{1},\dots,\lmb_{\ell};r)=W_{;1}+\tsum{j=2}{\ell}(W_{;j}+\vphi_{j})\approx\tsum{j=1}JW_{;j},
\end{align*}
where $\iota_{1},\dots,\iota_{\ell}\in\{\pm1\}$ and $(\lmb_{1},\dots\lmb_{\ell})\in\calP(\alp_{0})$,
such that 
\begin{equation}
\Dlt U_{J}+f(U_{J})\approx-\sum_{j=2}^{J}\frac{c_{j}}{\lmb_{j}^{2}}\Lmb W_{;j}\qquad\text{and}\qquad c_{j}\approx\br{\iota}_{j}\mu_{j}^{D}\kappa.\label{eq:U_J-eqn-formal}
\end{equation}
We want to construct $\vphi_{J+1}=\vphi_{J+1}(\iota_{1},\dots,\iota_{J+1},\lmb_{1},\dots,\lmb_{J+1};r)$
such that $U_{J+1}=W_{;J+1}+U_{J}+\vphi_{J+1}$ satisfies \eqref{eq:U_J-eqn-formal}
for $J+1$. The inductive hypothesis \eqref{eq:U_J-eqn-formal} and
$\Dlt W_{;J+1}+f(W_{;J+1})=0$ say that $\vphi_{J+1}$ needs to solve
\[
\Dlt\vphi_{J+1}+f(U_{J+1})-f(U_{J})-f(W_{;J+1})\approx-\frac{c_{J+1}}{\lmb_{J+1}^{2}}\Lmb W_{;J+1},
\]
where $c_{J+1}$ is a scalar to be chosen with $c_{J+1}\approx\br{\iota}_{J+1}\mu_{J+1}^{D}\kappa$.
Decomposing $f(U_{J+1})-f(U_{J})-f(W_{;J+1})$ into the zeroth order
(say $\mathrm{IH}$, standing for inhomogeneous error), linear, and
nonlinear terms (say $\NL$) in $\vphi_{J+1}$, we see that $\vphi_{J+1}$
needs to solve 
\begin{align*}
 & \Dlt\vphi_{J+1}+f'(W_{;J+1}+U_{J})\vphi_{J+1}\\
 & \quad\approx-\{f(W_{;J+1}+U_{J})-f(W_{;J+1})-f(U_{J})\}-\NL-\frac{c_{J+1}}{\lmb_{J+1}^{2}}\Lmb W_{;J+1}\\
 & \quad\approx-\mathrm{IH}-\NL-\frac{c_{J+1}}{\lmb_{J+1}^{2}}\Lmb W_{;J+1}.
\end{align*}
Next, we further approximate the linear operator in the LHS by $\Dlt+f'(W_{;J+1})=-H_{\lmb_{J+1}}$.
Then, we obtain the equation
\[
H_{\lmb_{J+1}}\vphi_{J+1}\approx\mathrm{IH}+\mathrm{SL}+\NL+\frac{c_{J+1}}{\lmb_{J+1}^{2}}\Lmb W_{;J+1},
\]
where $\mathrm{SL}=\{f'(W_{;J+1}+U_{J})-f'(W_{;J+1})\}\vphi_{J+1}$
is a small linear term. To solve this equation, we need to invert
the operator $H_{\lmb_{J+1}}$, which requires the RHS of the above
display to be orthogonal to the radial kernel element $\Lmb W_{\lmb_{J+1}}$.
This determines 
\[
c_{J+1}=-\lan\mathrm{IH}+\mathrm{SL}+\NL,\Lmb W_{;J+1}\ran/\|\Lmb W\|_{L^{2}}^{2}.
\]
It turns out that $c_{J+1}$ can be approximated by (using $U_{J}\approx W_{;J}(0)$
at the scale of $W_{;J+1}$) 
\[
c_{J+1}\approx-\frac{\lan\mathrm{IH},\Lmb W_{;J+1}\ran}{\|\Lmb W\|_{L^{2}}^{2}}\approx-\frac{\lan f'(W_{;J+1})W_{;J}(0),\Lmb W_{;J+1}\ran}{\|\Lmb W\|_{L^{2}}^{2}}=\br{\iota}_{J+1}\mu_{J+1}^{D}\kappa.
\]
With this $c_{J+1}$, we define $\vphi_{J+1}$ by solving the fixed
point problem
\begin{equation}
\vphi_{J+1}=H_{\lmb_{J+1}}^{-1}\{\mathrm{IH}+\mathrm{SL}+\NL+\frac{c_{J+1}}{\lmb_{J+1}^{2}}\Lmb W_{;J+1}\}.\label{eq:vphi_J+1-formal}
\end{equation}
In fact, we have $\vphi_{J+1}\approx\br{\iota}_{j}\mu_{j}^{D}S$ (introduced
in Remark~\ref{rem:7D-simpler-profile}) at the leading order, but
we will stick to the fixed point problem to generate higher order
expansions of the profile at once.

For the rigorous construction, we add a cutoff in the argument because
the approximation $\Dlt+f'(W_{;J+1}+U_{J})\approx\Dlt+f'(W_{;J+1})$
is valid only in the region where $f'(U_{J})\ll\frac{1}{r^{2}}$ (in
view of Hardy's inequality), i.e., $r\not\aeq\lmb_{j}$ for all $j=1,\dots,J$.
In particular, a cutoff in $r\leq R_{J+1}$ with $R_{J+1}\ll\lmb_{J}$
will be sufficient. We will choose $R_{J+1}=\dlt_{0}\lmb_{J}$ for
some small universal constant $0<\dlt_{0}\ll1$.

\subsection{\label{subsec:Notation-for-rigorous-construction}Notation for rigorous
construction}

Let $0<\dlt_{0},\alp_{0}<\frac{1}{10}$ and $\ell\geq2$. Define 
\[
\calQ_{\ell}(\alp_{0})\coloneqq\{(\iota_{1},\dots,\iota_{\ell},\lmb_{1},\dots,\lmb_{\ell};r)\in\{\pm1\}^{\ell}\times\calP_{\ell}(\alp_{0})\times(0,\infty):r<2\dlt_{0}\lmb_{\ell-1}\}.
\]

First, for a function $\vphi:\calQ_{\ell}(\alp_{0})\to\bbR$ and $k\in\bbN$,
we define 
\[
\|\vphi\|_{X^{k}}:\{\pm1\}^{\ell}\times\calP_{\ell}(\alp_{0})\to[0,\infty]
\]
by 
\[
\|\vphi\|_{X^{k}}(\iota_{1},\dots,\iota_{\ell},\lmb_{1},\dots,\lmb_{\ell})\coloneqq\inf\{K\geq0:|\vphi(\iota_{1},\dots,\iota_{\ell},\lmb_{1},\dots,\lmb_{\ell};\cdot)|_{k}\leq K\omg_{\ell}\},
\]
where (recall $\br{\lmb}_{\ell}=\sqrt{\lmb_{\ell}\lmb_{\ell-1}}$)
\begin{equation}
\omg_{\ell}(r)\coloneqq\omg(\lmb_{\ell-1},\lmb_{\ell};r)\coloneqq\chf_{(0,\lmb_{\ell}]}\frac{r^{2}}{\lmb_{\ell}^{2}}+\chf_{[\lmb_{\ell},\br{\lmb}_{\ell}]}\frac{\lmb_{\ell}^{2}}{r^{2}}+\chf_{[\br{\lmb}_{\ell},2\dlt_{0}\lmb_{\ell-1}]}\frac{\lmb_{\ell-1}^{D-2}\lmb_{\ell}^{D}}{r^{2D-2}}.\label{eq:def-omg-ell}
\end{equation}
We omit the superscript $k$ of $X^{k}$ if $k=0$. 

Next, for any $k'\in\{0,1\}$, define 
\[
\calX_{\ell}^{k'}\coloneqq\{\vphi:\calQ_{\ell}(\alp_{0})\to\bbR:\|\vphi\|_{\calX_{\ell}^{k'}}<+\infty\text{ and }\vphi\text{ is scaling invariant}\},
\]
where (we omit $k'$ of $\calX_{\ell}^{k'}$ if $k'=0$)
\begin{align*}
\|\vphi\|_{\calX_{\ell}} & \coloneqq\sup_{\{\pm1\}^{\ell}\times\calP_{\ell}(\alp_{0})}\lmb_{\ell-1}^{D}\|\vphi\|_{X^{2}},\\
\|\vphi\|_{\calX_{\ell}^{1}} & \coloneqq\max\Big\{\|\vphi\|_{\calX_{\ell}},\max_{1\leq j\leq\ell}\sup_{\{\pm1\}^{\ell}\times\calP_{\ell}(\alp_{0})}\Big(\max\{\lmb_{\ell-1}^{D},\lmb_{j}^{D}\}\|\lmb_{j}\rd_{\lmb_{j}}\vphi\|_{X^{1}}\Big)\Big\},
\end{align*}
and $\vphi$ being scaling invariant means 
\[
\vphi(\iota\iota_{1},\dots,\iota\iota_{\ell},\lmb\lmb_{1},\dots,\lmb\lmb_{\ell};\cdot)=\iota[\vphi(\iota_{1},\dots,\iota_{\ell},\lmb_{1},\dots,\lmb_{\ell};\cdot)]_{\lmb}.
\]

Finally, for any $j\geq1$, when the profiles $\vphi_{\ell}\in\calX_{\ell}$
are defined for all $\ell=2,\dots,j$, we denote 
\begin{equation}
\begin{aligned}U_{j}(\iota_{1},\dots,\iota_{j},\lmb_{1},\dots,\lmb_{j};r) & \coloneqq W_{;1}+\tsum{\ell=2}j(W_{;\ell}+\chi_{\dlt_{0}\lmb_{\ell-1}}\vphi_{\ell}),\\
P_{j}(\iota_{1},\dots,\iota_{j},\lmb_{1},\dots,\lmb_{j};r) & =Q_{;1}+\tsum{\ell=2}j(Q_{;\ell}+\chi_{\dlt_{0}\lmb_{\ell-1}}\rng{\vphi}_{\ell}),
\end{aligned}
\label{eq:def-Pj}
\end{equation}
where $W_{;\ell}=\iota_{\ell}W_{\lmb_{\ell}}(r)$ and $\vphi_{\ell}=\vphi_{\ell}(\iota_{1},\dots,\iota_{\ell},\lmb_{1},\dots,\lmb_{\ell};r)$.
We set $\vphi_{1}\equiv0$. Note that $U_{1}=W_{;1}$ and $P_{1}=Q_{;1}$.

\subsection{\label{subsec:Construction-of-vphi}Construction of $\protect\vphi_{j}$}

In this subsection, we construct the profiles $\vphi_{j}$ for all
$j\geq2$ (recall $\vphi_{1}\equiv0$). The main goal of this subsection
is Corollary~\ref{cor:Profiles-vphi}. \emph{All the implicit constants
in this subsection are independent of $J$ and depend only on $N$.}

As explained in Section~\ref{subsec:Sketch-of-construction}, we
will construct $\vphi_{J+1}$ inductively by a fixed point argument
(cf. \eqref{eq:vphi_J+1-formal}). The following lemma gives a preparation
for this.
\begin{lem}[Quantitative estimates for fixed point argument]
\label{lem:contract-recur}There exist constants $0<\alp_{0},\dlt_{0}<\frac{1}{10}$
and $K_{0}>10$ with the following properties. Let $J\geq1$ and suppose
that the profiles $\vphi_{\ell}\in\calX_{\ell}^{1}$ are constructed
for all $\ell=2,\dots,J$ with the uniform bounds 
\begin{align}
\|\vphi_{\ell}\|_{X^{1}} & \leq K_{0}\lmb_{\ell-1}^{-D},\label{eq:vphi_j-bd-1}\\
\|\lmb_{j}\rd_{\lmb_{j}}\vphi_{\ell}\|_{X} & \leq K_{0}\min\{\lmb_{j}^{-D},\lmb_{\ell-1}^{-D}\},\quad\forall j=1,\dots,\ell.\label{eq:vphi_j-bd-2}
\end{align}
Define $U_{J},P_{J}$ by \eqref{eq:def-Pj}. Then, the mapping $\Phi:\calX_{J+1}\to\calX_{J+1}$
defined by 
\[
\Phi:\vphi\mapsto\chf_{(0,2\dlt_{0}\lmb_{J})}\out H_{\lmb_{J+1}}^{-1}\{\chi_{\dlt_{0}\lmb_{J}}\calR_{J+1}^{(0)}+\calR_{J+1}^{(1)}[\vphi]-\frkr_{J+1}[\vphi]\Lmb W_{\ul{;J+1}}\},
\]
where the operator $\out H_{\lmb}^{-1}$ is defined in Section~\ref{subsec:Formal-inversion}
and 
\begin{align*}
\calR_{J+1}^{(0)} & \coloneqq r^{-(D+2)}\{f(Q_{;J+1}+P_{J})-f(Q_{;J+1})-f(P_{J})\},\\
\calR_{J+1}^{(1)}[\vphi] & \coloneqq r^{-(D+2)}\{f(Q_{;J+1}+P_{J}+\chi_{\dlt_{0}\lmb_{J}}\rng{\vphi})-f(Q_{;J+1}+P_{J})-f'(Q_{;J+1})\chi_{\dlt_{0}\lmb_{J}}\rng{\vphi}\},\\
\frkr_{J+1}[\vphi] & \coloneqq\lan\chi_{\dlt_{0}\lmb_{J}}\calR_{J+1}^{(0)}+\calR_{J+1}^{(1)}[\vphi],\Lmb W_{;J+1}\ran/\|\Lmb W\|_{L^{2}}^{2},
\end{align*}
satisfies the following quantitative estimates:
\begin{enumerate}
\item (Contraction mapping on $\calX_{J+1}$) For any $\vphi,\vphi'\in\calX_{J+1}$
with $\|\vphi\|_{\calX_{J+1}},\|\vphi'\|_{\calX_{J+1}}\leq K_{0}$,
we have 
\begin{align}
\|\Phi[\vphi]\|_{X^{2}} & \leq K_{0}\lmb_{J}^{-D},\label{eq:Phi-map-1}\\
\|\Phi[\vphi]-\Phi[\vphi']\|_{X^{2}} & \leq\frac{1}{2}\|\vphi-\vphi'\|_{X}.\label{eq:Phi-map-2}
\end{align}
\item (Uniform bound and difference estimate for $\lmb_{j}\rd_{\lmb_{j}}\Phi[\vphi]$)
For any $\vphi,\vphi'\in\calX_{J+1}^{1}$ with $\|\vphi\|_{\calX_{J+1}^{1}},\|\vphi'\|_{\calX_{J+1}^{1}}\leq K_{0}$,
we have for $j=1,\dots,J$, 
\begin{align}
\|\lmb_{j}\rd_{\lmb_{j}}\Phi[\vphi]\|_{X^{2}} & \leq K_{0}\lmb_{j}^{-D},\label{eq:Phi-lmb-deriv-1}\\
\|\lmb_{j}\rd_{\lmb_{j}}\{\Phi[\vphi]-\Phi[\vphi']\}\|_{X^{2}} & \leq\frac{1}{2}\|\lmb_{j}\rd_{\lmb_{j}}(\vphi-\vphi')\|_{X}+\frac{1}{\lmb_{J+1}^{D-2}}\|\vphi-\vphi'\|_{X}^{p-1},\label{eq:Phi-lmb-deriv-2}
\end{align}
and we have for $j=J+1$ 
\begin{equation}
\|\lmb_{J+1}\rd_{\lmb_{J+1}}\Phi[\vphi]\|_{X^{1}}\leq K_{0}\lmb_{J}^{-D}.\label{eq:Phi-lmb-deriv-3}
\end{equation}
\item (Estimates for $\frkr_{J+1}[\vphi]$) We have
\begin{align}
|\frkr_{J+1}[\vphi]| & \aleq\mu_{J+1}^{D},\label{eq:frkr-hat-1}\\
|\frkr_{J+1}[\vphi]+\kappa\br{\iota}_{J+1}\mu_{J+1}^{D}| & \aleq(\chf_{J\geq2}\mu_{J}^{D}+\mu_{J+1})\mu_{J+1}^{D}.\label{eq:frkr-hat-2}
\end{align}
\end{enumerate}
\end{lem}

\begin{rem}
The estimate \eqref{eq:frkr-hat-2} is not necessary in the construction
of the profiles $\vphi_{j}$ themselves, but it will be used in the
proof of Proposition~\ref{prop:Modified-Profile}.
\end{rem}

\begin{proof}
\textbf{Step 0.} Remarks on the proof.

The constants $\alp_{0},\dlt_{0},K_{0}$ will be chosen in the following
order: 
\begin{equation}
0<\alp_{0}\ll\dlt_{0}\ll K_{0}^{-1}\ll1.\label{eq:profile-param-dep}
\end{equation}
By the definition of $\calX_{\ell}^{k'}$, we assume for $\vphi$
and $\vphi'$ 
\begin{align*}
|\vphi|,|\vphi'| & \leq K_{0}\lmb_{J}^{-D}\omg_{J+1},\\
|\lmb_{j}\rd_{\lmb_{j}}\vphi|,|\lmb_{j}\rd_{\lmb_{j}}\vphi'| & \leq K_{0}\lmb_{j}^{-D}\omg_{J+1},\qquad\forall j=1,\dots,J.
\end{align*}
Finally, recall that $A\aleq B$ means $|A|\leq CB$.

\textbf{Step 1.} Preliminary pointwise bounds.

We claim the following pointwise bounds on $(0,2\dlt_{0}\lmb_{J}]$:
\begin{align}
r^{-2}\omg_{J+1} & \aeq\chf_{(0,\br{\lmb}_{J+1}]}[\lan y\ran^{-2D}]_{\lmb_{J+1}}^{p-1}+\chf_{[\br{\lmb}_{J+1},2\dlt_{0}\lmb_{J}]}[\lan y\ran^{-2D}]_{\lmb_{J}}^{p-2}[\lan y\ran^{-2D}]_{\lmb_{J+1}},\label{eq:1-1}\\
|U_{J}| & \aeq\lmb_{J}^{-D}\aeq[\lan y\ran^{-2D}]_{\lmb_{J}},\label{eq:1-2}\\
U_{J}-W_{;J} & \aleq\chf_{J\geq2}\lmb_{J-1}^{-D}\aeq\chf_{J\geq2}[\lan y\ran^{-2D}]_{\lmb_{J-1}},\label{eq:1-3}\\
\lmb_{j}\rd_{\lmb_{j}}U_{J} & \aleq\lmb_{j}^{-D},\qquad\forall j=1,\dots,J,\label{eq:1-4}\\
r^{2}\omg_{J+1}^{p-1} & \aleq\lmb_{J}^{2}\mu_{J+1}^{p},\label{eq:1-5}\\
r^{2}\omg_{J+1}^{p-2} & \aleq\dlt_{0}^{2}\lmb_{J+1}^{2}\mu_{J+1}^{-D}.\label{eq:1-6}
\end{align}
In this step, all the estimates will be restricted on $(0,2\dlt_{0}\lmb_{J}]$.
Now \eqref{eq:1-1} follows from the estimates $[\lan y\ran^{-2D}]_{\lmb_{J+1}}\aeq\chf_{(0,\lmb_{J+1}]}\lmb_{J+1}^{-D}+\chf_{[\lmb_{J+1},2\dlt_{0}\lmb_{J}]}\lmb_{J+1}^{D}r^{-2D}$,
$[\lan y\ran^{-2D}]_{\lmb_{J}}\aeq\lmb_{J}^{-D}$, and the relation
$(p-1)D=2$. To show \eqref{eq:1-2}, \eqref{eq:1-3}, and \eqref{eq:1-4},
we note the identities 
\[
U_{J}-W_{;J}=\sum_{\ell=1}^{J-1}W_{;\ell}+\sum_{\ell=2}^{J}\vphi_{\ell}\quad\text{and}\quad\lmb_{j}\rd_{\lmb_{j}}U_{J}=-\Lmb W_{;j}+\sum_{\ell=j}^{J}\lmb_{j}\rd_{\lmb_{j}}\vphi_{\ell}.
\]
Using $|W_{;\ell}|\aeq\lmb_{\ell}^{-D}$, $|\vphi_{\ell}|\leq K_{0}\lmb_{\ell-1}^{-D}(2\dlt_{0}\lmb_{J}/\lmb_{\ell})^{2}$
from \eqref{eq:vphi_j-bd-1}, and $K_{0}\dlt_{0}^{2}\ll1$ from parameter
dependence \eqref{eq:profile-param-dep}, we have \eqref{eq:1-2}
and \eqref{eq:1-3}. Using $\Lmb W_{;j}\aleq\lmb_{j}^{-D}$, $|\lmb_{j}\rd_{\lmb_{j}}\vphi_{\ell}|\leq K_{0}\lmb_{j}^{-D}(2\dlt_{0}\lmb_{J}/\lmb_{\ell})^{2}$
from \eqref{eq:vphi_j-bd-2}, and $K_{0}\dlt_{0}^{2}\ll1$ from parameter
dependence \eqref{eq:profile-param-dep}, we obtain \eqref{eq:1-4}.
To show \eqref{eq:1-5} and \eqref{eq:1-6}, observe that $r^{2}\omg_{J+1}^{p-1}$
and $r^{2}\omg_{J+1}^{p-2}$ achieve their maximum at $r=\br{\lmb}_{J+1}$
and $r=2\dlt_{0}\lmb_{J}$, respectively. Evaluating at these points
give \eqref{eq:1-5} and \eqref{eq:1-6}.

\textbf{Step 2.} Estimates for $\calR_{J+1}^{(0)}$.

In this step, we estimate the inhomogeneous error term $\calR_{J+1}^{(0)}$
(cf. $\mathrm{IH}$ in Section~\ref{subsec:Sketch-of-construction}):
\[
\calR_{J+1}^{(0)}=r^{-(D+2)}\{f(Q_{;J+1}+P_{J})-f(Q_{;J+1})-f(P_{J})\}.
\]
We claim 
\begin{align}
\chf_{(0,2\dlt_{0}\lmb_{J}]}\calR_{J+1}^{(0)} & \aleq\lmb_{J}^{-D}\cdot r^{-2}\omg_{J+1},\label{eq:R0-01}\\
\chf_{(0,2\dlt_{0}\lmb_{J}]}\lmb_{j}\rd_{\lmb_{j}}\calR_{J+1}^{(0)} & \aleq\lmb_{j}^{-D}\cdot r^{-2}\omg_{J+1},\qquad\forall j=1,\dots,J.\label{eq:R0-02}
\end{align}
First, \eqref{eq:R0-01} follows from applying \eqref{eq:f-05}, \eqref{eq:f-05-HM},
\eqref{eq:1-2}, and \eqref{eq:1-1}:
\begin{align*}
 & \chf_{(0,2\dlt_{0}\lmb_{J}]}\calR_{J+1}^{(0)}\\
 & \aleq r^{-(D+2)}\{\chf_{(0,\br{\lmb}_{J+1}]}|Q_{;J+1}|^{p-1}|P_{J}|+\chf_{[\br{\lmb}_{J+1},2\dlt_{0}\lmb_{J}]}|P_{J}|^{p-1}|Q_{;J+1}|\}\\
 & \aleq\chf_{(0,\br{\lmb}_{J+1}]}[\lan y\ran^{-2D}]_{\lmb_{J+1}}^{p-1}[\lan y\ran^{-2D}]_{\lmb_{J}}+\chf_{[\br{\lmb}_{J+1},2\dlt_{0}\lmb_{J}]}[\lan y\ran^{-2D}]_{\lmb_{J}}^{p-1}[\lan y\ran^{-2D}]_{\lmb_{J+1}}\\
 & \aeq\lmb_{J}^{-D}\cdot r^{-2}\omg_{J+1}.
\end{align*}
To show \eqref{eq:R0-02}, we note the identity 
\begin{align*}
 & \lmb_{j}\rd_{\lmb_{j}}(\chi_{\dlt_{0}\lmb_{J}}\calR_{J+1}^{(0)})\\
 & =\chi_{\dlt_{0}\lmb_{J}}r^{-2}\{f'(Q_{;J+1}+P_{J})-f'(P_{J})\}(\lmb_{j}\rd_{\lmb_{j}}U_{J})-\chf_{j=J}(r\rd_{r}\chi_{\dlt_{0}\lmb_{J}})\calR_{J+1}^{(0)}.
\end{align*}
For the first term, we use \eqref{eq:f-02}, \eqref{eq:f-03-HM},
\eqref{eq:1-2}, \eqref{eq:1-4}, and \eqref{eq:1-1} to obtain 
\begin{align*}
 & \chi_{\dlt_{0}\lmb_{J}}r^{-2}\{f'(Q_{;J+1}+P_{J})-f'(P_{J})\}(\lmb_{j}\rd_{\lmb_{j}}U_{J})\\
 & \aleq r^{-2}\{\chf_{(0,\br{\lmb}_{J+1}]}|Q_{;J+1}|^{p-1}+\chf_{[\br{\lmb}_{J+1},2\dlt_{0}\lmb_{J}]}|P_{J}|^{p-2}|Q_{;J+1}|\}\lmb_{j}^{-D}\\
 & \aleq\{\chf_{(0,\br{\lmb}_{J+1}]}[\lan y\ran^{-2D}]_{\lmb_{J+1}}^{p-1}+\chf_{[\br{\lmb}_{J+1},2\dlt_{0}\lmb_{J}]}[\lan y\ran^{-2D}]_{\lmb_{J}}^{p-2}[\lan y\ran^{-2D}]_{\lmb_{J+1}}\}\lmb_{j}^{-D}\\
 & \aleq\lmb_{j}^{-D}\cdot r^{-2}\omg_{J+1}.
\end{align*}
For the second term, \eqref{eq:R0-01} implies 
\[
\chf_{j=J}(r\rd_{r}\chi_{\dlt_{0}\lmb_{J}})\calR_{J+1}^{(0)}\aleq\chf_{j=J}\lmb_{J}^{-D}\cdot r^{-2}\omg_{J+1}\aleq\lmb_{j}^{-D}\cdot r^{-2}\omg_{J+1}.
\]
The previous two displays give \eqref{eq:R0-02}.

\textbf{Step 3.} Estimates for $\calR_{J+1}^{(1)}$.

The term $\calR_{J+1}^{(1)}$ corresponds to the sum of small linear
terms and nonlinear terms (cf. $\mathrm{SL}+\mathrm{NL}$ in Section~\ref{subsec:Sketch-of-construction}).
Recall 
\[
\calR_{J+1}^{(1)}[\vphi]=r^{-(D+2)}\{f(Q_{;J+1}+P_{J}+\chi_{\dlt_{0}\lmb_{J}}\rng{\vphi})-f(Q_{;J+1}+P_{J})-f'(Q_{;J+1})\chi_{\dlt_{0}\lmb_{J}}\rng{\vphi}\}.
\]
We claim 
\begin{align}
\calR_{J+1}^{(1)}[\vphi] & \aleq\{K_{0}\lmb_{J}^{-(D+2)}+K_{0}^{p}\mu_{J+1}^{p}\lmb_{J}^{-D}r^{-2}\}\omg_{J+1},\label{eq:R1-01}\\
\calR_{J+1}^{(1)}[\vphi]-\calR_{J+1}^{(1)}[\vphi'] & \aleq\dlt_{0}^{2}\|\vphi-\vphi'\|_{X}\cdot r^{-2}\omg_{J+1},\label{eq:R1-02}
\end{align}
and for $j=1,\dots J$ 
\begin{align}
\lmb_{j}\rd_{\lmb_{j}}(\calR_{J+1}^{(1)}[\vphi]) & \aleq K_{0}\dlt_{0}^{2}\lmb_{j}^{-D}\cdot r^{-2}\omg_{J+1},\label{eq:R1-03}\\
\lmb_{j}\rd_{\lmb_{j}}(\calR_{J+1}^{(1)}[\vphi]-\calR_{J+1}^{(1)}[\vphi']) & \aleq\dlt_{0}^{2}\{\|\lmb_{j}\rd_{\lmb_{j}}(\vphi-\vphi')\|_{X}\label{eq:R1-04}\\
 & \qquad+K_{0}\lmb_{J+1}^{-(D-2)}\|\vphi-\vphi'\|_{X}^{p-1}\}\cdot r^{-2}\omg_{J+1}.\nonumber 
\end{align}
The proof is relegated to Appendix~\ref{sec:Proof-of-R1}.

\textbf{Step 4.} Estimates for $\frkr_{J+1}[\vphi]$; proof of \eqref{eq:frkr-hat-1}
and \eqref{eq:frkr-hat-2}.

We first show \eqref{eq:frkr-hat-1}. Note that 
\[
\lan r^{-2}\omg_{J+1},\Lmb W_{;J+1}\ran\aleq\lmb_{J+1}^{D}\quad\text{and}\quad\lan\omg_{J+1},\Lmb W_{;J+1}\ran\aleq\lmb_{J+1}^{D+2}|\log\mu_{J+1}|.
\]
Combining these estimates with \eqref{eq:R0-01} and \eqref{eq:R1-01},
we have 
\begin{align*}
\lan\chi_{\dlt_{0}\lmb_{J}}\calR_{J+1}^{(0)},\Lmb W_{;J+1}\ran & \aleq\mu_{J+1}^{D},\\
\lan\calR_{J+1}^{(1)}[\vphi],\Lmb W_{;J+1}\ran & \aleq K_{0}\mu_{J+1}^{D+2}|\log\mu_{J+1}|+K_{0}^{p}\mu_{J+1}^{D+p}\aleq K_{0}^{p}\mu_{J+1}^{D+p}.
\end{align*}
This completes the proof of \eqref{eq:frkr-hat-1}.

We turn to show \eqref{eq:frkr-hat-2}. By the previous display, the
contribution of $\calR_{J+1}^{(1)}[\vphi]$ is already safe; it suffices
to consider $\calR_{J+1}^{(0)}$. Note the identity 
\begin{align*}
\chi_{\dlt_{0}\lmb_{J}}\calR_{J+1}^{(0)} & =\chi_{\dlt_{0}\lmb_{J}}r^{-2}f'(Q_{;J+1})\cdot W_{;J}(0)+\chi_{\dlt_{0}\lmb_{J}}r^{-2}f'(Q_{;J+1})\{U_{J}-W_{;J}(0)\}\\
 & \peq+\chi_{\dlt_{0}\lmb_{J}}r^{-(D+2)}\{f(Q_{;J+1}+P_{J})-f(Q_{;J+1})-f'(Q_{;J+1})P_{J}-f(P_{J})\}.
\end{align*}
The first term produces the leading term; we have (recall \eqref{eq:def-kappa-unified})
\begin{align*}
\lan r^{-2}f'(Q_{;J+1})\cdot W_{;J}(0),\Lmb W_{;J+1}\ran & =\br{\iota}_{J+1}\mu_{J+1}^{D}\lan r^{-2}f'(Q)\cdot W(0),\Lmb W\ran\\
 & =-\kappa\br{\iota}_{J+1}\mu_{J+1}^{D}\|\Lmb W\|_{L^{2}}^{2}
\end{align*}
and 
\begin{align*}
 & \lan r^{-2}f'(Q_{;J+1})\cdot(1-\chi_{\dlt_{0}\lmb_{J}})W_{;J}(0),\Lmb W_{;J+1}\ran\\
 & \aleq\lmb_{J}^{-D}\|\chf_{[\dlt_{0}\lmb_{J},+\infty)}[\lan y\ran^{-2D}]_{\lmb_{J+1}}^{p}\|_{L^{1}}\aleq\lmb_{J}^{-D}\lmb_{J+1}^{D+2}\|\chf_{[\dlt_{0}\lmb_{J},+\infty)}r^{-(N+2)}\|_{L^{1}}\aleq\dlt_{0}^{-2}\mu_{J+1}^{D+2}.
\end{align*}
The remaining terms are considered as errors. First, we use \eqref{eq:1-3}
to have 
\begin{align*}
 & \chi_{\dlt_{0}\lmb_{J}}r^{-2}f'(Q_{;J+1})\{U_{J}-W_{;J}(0)\}\\
 & =\chi_{\dlt_{0}\lmb_{J}}r^{-2}f'(Q_{;J+1})\{(U_{J}-W_{;J})+(W_{;J}-W_{;J}(0))\}\\
 & \aleq\chf_{(0,2\dlt_{0}\lmb_{J}]}[\lan y\ran^{-2D}]_{\lmb_{J+1}}^{p-1}\{\chf_{J\geq2}\lmb_{J-1}^{-D}+\lmb_{J}^{-(D+2)}r^{2}\}
\end{align*}
and hence 
\[
\lan\chi_{\dlt_{0}\lmb_{J}}f'(Q_{;J+1})\{U_{J}-W_{;J}(0)\},\Lmb W_{;J+1}\ran\aleq\chf_{J\geq2}\mu_{J}^{D}\mu_{J+1}^{D}+\mu_{J+1}^{D+2}|\log\mu_{J+1}|.
\]
Next, we recall \eqref{eq:1-2}, apply \eqref{eq:f-04} and \eqref{eq:f-06-HM}
on $(0,\br{\lmb}_{J+1}]$, and apply \eqref{eq:f-05} and \eqref{eq:f-05-HM}
on $[\br{\lmb}_{J+1},2\dlt_{0}\lmb_{J}]$ to obtain 
\begin{align*}
 & \chi_{\dlt_{0}\lmb_{J}}r^{-(D+2)}\{f(Q_{;J+1}+P_{J})-f(Q_{;J+1})-f'(Q_{;J+1})P_{J}-f(P_{J})\}\\
 & \aleq\chf_{(0,\br{\lmb}_{J+1}]}[\lan y\ran^{-2D}]_{\lmb_{J}}^{p}+\chf_{[\br{\lmb}_{J+1},2\dlt_{0}\lmb_{J}]}[\lan y\ran^{-2D}]_{\lmb_{J}}^{p-1}[\lan y\ran^{-2D}]_{\lmb_{J+1}}\\
 & \aleq\chf_{(0,\br{\lmb}_{J+1}]}\lmb_{J}^{-(D+2)}+\chf_{[\br{\lmb}_{J+1},2\dlt_{0}\lmb_{J}]}\lmb_{J}^{-2}\lmb_{J+1}^{D}r^{-2D}
\end{align*}
and hence 
\[
\lan\chi_{\dlt_{0}\lmb_{J}}r^{-D-2}\{f(Q_{;J+1}+P_{J})-f(Q_{;J+1})-f'(Q_{;J+1})P_{J}-f(P_{J})\},\Lmb W_{;J+1}\ran\aleq\mu_{J+1}^{D+1}.
\]
This completes the proof of \eqref{eq:frkr-hat-2}.

\textbf{Step 5.} Completion of the proof.

In this step, we prove \eqref{eq:Phi-map-1}--\eqref{eq:Phi-lmb-deriv-3}.
Note that we have a pointwise estimate 
\[
r^{-2}\omg_{J+1}+|\lan r^{-2}\omg_{J+1},\Lmb W_{;J+1}\ran\Lmb W_{\ul{;J+1}}|\aleq r^{-2}\omg_{J+1}+\chf_{[2\dlt_{0}\lmb_{J},+\infty)}\lmb_{J+1}^{2D-2}r^{-2D}.
\]
Taking $\out H_{\lmb_{J+1}}^{-1}$ using \eqref{eq:H-inv-2} and \eqref{eq:H-lmb-inv-rel},
and applying the pointwise bounds \eqref{eq:Gmm-ptwise}, we have
for any function $F$ with $F\aleq r^{-2}\omg_{J+1}$ 
\begin{equation}
\chf_{(0,2\dlt_{0}\lmb_{J}]}\Big|\out H_{\lmb_{J+1}}^{-1}\Big\{ F-\frac{\lan F,\Lmb W_{;J+1}\ran\Lmb W_{\ul{;J+1}}}{\|\Lmb W\|_{L^{2}}^{2}}\Big\}\Big|_{2}\aleq\omg_{J+1}.\label{eq:cont-01}
\end{equation}

Now the estimates \eqref{eq:R0-01}, \eqref{eq:R1-01} after replacing
$\lmb_{J}^{-2}\leq\dlt_{0}^{2}r^{-2}$, and \eqref{eq:R1-02} imply
\begin{align*}
|\Phi[\vphi]|_{2} & \aleq(1+K_{0}\dlt_{0}^{2})\lmb_{J}^{-D}\omg_{J+1},\\
|\Phi[\vphi]-\Phi[\vphi']|_{2} & \aleq\dlt_{0}^{2}\|\vphi-\vphi'\|_{X}\omg_{J+1}.
\end{align*}
For $j\in\{1,\dots,J\}$, the estimates \eqref{eq:R0-02}, \eqref{eq:R1-03},
and \eqref{eq:R1-04} imply 
\begin{align*}
|\lmb_{j}\rd_{\lmb_{j}}\Phi[\vphi]|_{2} & \aleq(1+K_{0}\dlt_{0}^{2})\lmb_{j}^{-D}\omg_{J+1},\\
|\lmb_{j}\rd_{\lmb_{j}}\{\Phi[\vphi]-\Phi[\vphi']\}|_{2} & \aleq\dlt_{0}^{2}\{\|\lmb_{j}\rd_{\lmb_{j}}(\vphi-\vphi')\|_{X}+K_{0}\lmb_{J+1}^{-(D-2)}\|\vphi-\vphi'\|_{X}^{p-1}\}\omg_{J+1}.
\end{align*}
Finally when $j=J+1$, the scaling invariance of $\Phi[\vphi]$ implies
\begin{equation}
\lmb_{J+1}\rd_{\lmb_{J+1}}\Phi[\vphi]=-\Lmb\Phi[\vphi]-\sum_{j=1}^{J}\lmb_{j}\rd_{\lmb_{j}}\Phi[\vphi],\label{eq:scale-01}
\end{equation}
so using $\sum_{j=1}^{J}\lmb_{j}^{-D}\aleq\lmb_{J}^{-D}$ we have
\[
|\lmb_{J+1}\rd_{\lmb_{J+1}}\Phi[\vphi]|_{1}\aleq(1+K_{0}\dlt_{0}^{2})\lmb_{J}^{-D}\omg_{J+1}.
\]
Applying the parameter dependence \eqref{eq:profile-param-dep} completes
the proof of \eqref{eq:Phi-map-1}--\eqref{eq:Phi-lmb-deriv-3}.
\end{proof}
\begin{lem}[Recursive construction of $\vphi_{J+1}$]
\label{lem:hat-vphi-J+1-construction}Let the constants $\alp_{0},\dlt_{0},K_{0}$
be fixed by Lemma~\ref{lem:contract-recur} and assume the hypothesis
as in Lemma~\ref{lem:contract-recur} for some $J\geq1$. Then, $\Phi$
has a unique fixed point $\vphi_{J+1}\in\calX_{J+1}^{1}$ such that
\begin{align}
\|\vphi_{J+1}\|_{X^{2}} & \leq K_{0}\lmb_{J}^{-D},\label{eq:hat-vphi-J+1-bd}\\
\|\lmb_{j}\rd_{\lmb_{j}}\vphi_{J+1}\|_{X^{1}} & \leq K_{0}\min\{\lmb_{j}^{-D},\lmb_{J}^{-D}\},\qquad\forall j=1,\dots,J+1.\label{eq:hat-vphi-J+1-deriv}
\end{align}
\end{lem}

\begin{proof}
We proceed similarly as in the proof of the Banach fixed point theorem.
First, choose the initial state $\vphi^{(0)}\equiv0$ and define the
iterates $\vphi^{(n+1)}\coloneqq\Phi[\vphi^{(n)}]$ inductively. By
\eqref{eq:Phi-map-1} and \eqref{eq:Phi-map-2}, the sequence $\vphi^{(n+1)}$
converges to $\vphi_{J+1}\in\calX_{J+1}$, which is the unique fixed
point of $\Phi$ satisfying \eqref{eq:hat-vphi-J+1-bd}. Moreover,
we have a quantitative estimate $\|\vphi^{(n+1)}-\vphi^{(n)}\|_{\calX_{J+1}}\leq K_{0}2^{-n}$
for all $n\in\bbN$.

It remains to show that $\vphi_{J+1}$ is differentiable in each $\lmb_{j}$
with the uniform bounds \eqref{eq:hat-vphi-J+1-deriv}. By \eqref{eq:Phi-lmb-deriv-1}
and \eqref{eq:Phi-lmb-deriv-3}, the iterates $\vphi^{(n+1)}$ satisfy
the uniform bound $\|\lmb_{j}\rd_{\lmb_{j}}\vphi^{(n+1)}\|_{X^{1}}\leq K_{0}\min\{\lmb_{j}^{-D},\lmb_{J}^{-D}\}$
for $j=1,\dots,J+1$. By scaling invariance, it suffices to show that
$\lmb_{j}\rd_{\lmb_{j}}\vphi^{(n+1)}$ converges locally uniformly
in $\lmb_{1},\dots,\lmb_{J+1}$ for all $j=1,\dots,J$. We substitute
the estimate $\|\vphi^{(n+1)}-\vphi^{(n)}\|_{\calX_{J+1}}\leq K_{0}2^{-n}$
into \eqref{eq:Phi-lmb-deriv-2} to obtain 
\[
\|\lmb_{j}\rd_{\lmb_{j}}(\vphi^{(n+2)}-\vphi^{(n+1)})\|_{X^{1}}\leq\frac{1}{2}\|\lmb_{j}\rd_{\lmb_{j}}(\vphi^{(n+1)}-\vphi^{(n)})\|_{X^{1}}+K_{0}^{p-1}\lmb_{J+1}^{-D}2^{-n(p-1)}.
\]
Choosing the constants $2^{-(p-1)}<a<1$ and $C_{0}>0$ such that
$aC_{0}>2^{-(p-1)}C_{0}+1$, we have for any $n\in\bbN$ 
\begin{align*}
 & \|\lmb_{j}\rd_{\lmb_{j}}(\vphi^{(n+2)}-\vphi^{(n+1)})\|_{X^{1}}+C_{0}K_{0}^{p-1}\lmb_{J+1}^{-D}\cdot2^{-(n+1)(p-1)}\\
 & \leq\frac{1}{2}\|\lmb_{j}\rd_{\lmb_{j}}(\vphi^{(n+1)}-\vphi^{(n)})\|_{X^{1}}+(2^{-(p-1)}C_{0}+1)K_{0}^{p-1}\lmb_{J+1}^{-D}\cdot2^{-n(p-1)}\\
 & \leq a\cdot(\|\lmb_{j}\rd_{\lmb_{j}}(\vphi^{(n+1)}-\vphi^{(n)})\|_{X^{1}}+C_{0}K_{0}^{p-1}\lmb_{J+1}^{-D}\cdot2^{-n(p-1)}).
\end{align*}
This inequality shows that $\|\lmb_{j}\rd_{\lmb_{j}}(\vphi^{(n+1)}-\vphi^{(n)})\|_{X^{1}}$
is exponentially decreasing in $n$ and hence $\lmb_{j}\rd_{\lmb_{j}}\vphi^{(n+1)}$
converges locally uniformly in $\lmb_{1},\dots,\lmb_{J+1}$. This
completes the proof of \eqref{eq:hat-vphi-J+1-deriv}.
\end{proof}
By recursion, we have constructed the profiles $\vphi_{\ell}\in\calX_{\ell}^{1}$
for all $\ell\geq2$:
\begin{cor}[The profiles $\vphi_{\ell}$]
\label{cor:Profiles-vphi}There exist $0<\alp_{0},\dlt_{0}<\frac{1}{10}$
and profiles $\vphi_{\ell}\in\calX_{\ell}^{1}$ for all $\ell\geq2$
with the following properties.
\begin{enumerate}
\item (Estimates for $\vphi_{\ell}$) Recall the definition of $\omg_{\ell}$
from \eqref{eq:def-omg-ell}. We have 
\begin{align}
|\vphi_{\ell}|_{2} & \aleq\lmb_{\ell-1}^{-D}\omg_{\ell},\label{eq:vphi-bd}\\
|\lmb_{j}\rd_{\lmb_{j}}\vphi_{\ell}|_{1} & \aleq\min\{\lmb_{\ell-1}^{-D},\lmb_{j}^{-D}\}\omg_{\ell},\label{eq:vphi-deriv}
\end{align}
\item (Equation for $\vphi_{\ell}$) $\vphi_{\ell}$ solves on the region
$(0,2\dlt_{0}\lmb_{\ell-1})$
\begin{equation}
H_{\lmb_{\ell}}\vphi_{\ell}=-\frac{\frkr_{\ell}}{\lmb_{\ell}^{2}}\Lmb W_{;\ell}+\chi_{\dlt\lmb_{\ell-1}}\calR_{\ell}^{(0)}+\calR_{\ell}^{(1)},\label{eq:vphi-eqn}
\end{equation}
where 
\begin{align*}
\frkr_{\ell} & =\lan\chi_{\dlt\lmb_{\ell-1}}\calR_{\ell}^{(0)}+\calR_{\ell}^{(1)},\Lmb W_{;\ell}\ran/\|\Lmb W\|_{L^{2}}^{2},\\
\calR_{\ell}^{(0)} & =r^{-(D+2)}\{f(Q_{;\ell}+P_{\ell-1})-f(Q_{;\ell})-f(P_{\ell-1})\},\\
\calR_{\ell}^{(1)} & =r^{-(D+2)}\{f(Q_{;\ell}+P_{\ell-1}+\chi_{\dlt\lmb_{\ell-1}}\rng{\vphi}_{\ell})-f(Q_{;\ell}+P_{\ell-1})-f'(Q_{;\ell})\chi_{\dlt\lmb_{\ell-1}}\rng{\vphi}_{\ell}\}.
\end{align*}
\item (Estimates for $\frkr_{\ell}$) We have 
\begin{align}
|\frkr_{\ell}| & \aleq\mu_{\ell}^{D},\label{eq:frkr-1}\\
|\frkr_{\ell}+\kpp\br{\iota}_{\ell}\mu_{\ell}^{D}| & \aleq\mu_{\ell}^{D+1}+\chf_{\ell\geq3}\mu_{\ell}^{D}\mu_{\ell-1}^{D},\label{eq:frkr-2}
\end{align}
\end{enumerate}
\end{cor}

\begin{proof}
Fix $\alp_{0},\dlt_{0},K_{0}$ as in Lemma~\ref{lem:contract-recur}.
By Lemma~\ref{lem:hat-vphi-J+1-construction}, for all $\ell\geq2$,
we can recursively construct the profiles $\vphi_{\ell}\in\calX_{\ell}^{1}$
such that the statements with $J+1=\ell$ in Lemma~\ref{lem:contract-recur}
are true. As $K_{0}$ is fixed, the statement can be written without
keeping track of $K_{0}$ in the implicit constants.
\end{proof}

\subsection{Proof of Proposition~\ref{prop:Modified-Profile}}

In this subsection, we prove Proposition~\ref{prop:Modified-Profile}
using the corrector profiles $\vphi_{\ell}$ constructed in Corollary~\ref{cor:Profiles-vphi}.
Fix the constants $\alp_{0},\dlt_{0}>0$ as in Corollary~\ref{cor:Profiles-vphi}.
Note that (cf. \eqref{eq:def-Pj}) for each $j\geq1$ the modified
$j$ bubbles profile $U_{j}$ is defined by 
\begin{align*}
U_{j}=U_{j}(\iota_{1},\dots,\iota_{j},\lmb_{1},\dots,\lmb_{j};r) & =W_{;1}+\tsum{\ell=2}j(W_{;\ell}+\chi_{\dlt_{0}\lmb_{\ell-1}}\vphi_{\ell}).
\end{align*}
From now on, we fix $J\in\bbN$ and allow implicit constants to depend
also on $J$. 
\begin{lem}
\label{lem:A_0}There exists $A_{0}>10$ such that \eqref{eq:U-bd-1}--\eqref{eq:U-bd-3}
for \eqref{eq:NLH-rad} and \eqref{eq:U-bd-2-HM}--\eqref{eq:U-bd-3-HM}
for \eqref{eq:HMHF-equiv} hold.
\end{lem}

\begin{proof}
First, observe using \eqref{eq:vphi-bd} that 
\begin{align*}
\chi_{\dlt_{0}\lmb_{\ell-1}}\vphi_{\ell}\aleq\lmb_{\ell-1}^{-D}\omg_{\ell}\aleq\chf_{(0,2\dlt_{0}\lmb_{\ell-1}]}\lmb_{\ell-1}^{-D}\min\{1,\lmb_{\ell-1}^{D}\lmb_{\ell}^{D}r^{-2D}\}\quad\\
\aleq\min\{|[\lan y\ran^{-2D}]_{;\ell-1}|,|[\lan y\ran^{-2D}]_{;\ell}|\}.
\end{align*}
Thus we obtain 
\begin{align*}
U-W_{;j} & \aleq\sum_{\ell\neq j}|[\lan y\ran^{-2D}]_{;\ell}|,\qquad\text{for \eqref{eq:NLH-rad}},\\
r^{-D}\sin(P-Q_{;j}) & \aleq\sum_{\ell\neq j}|[\lan y\ran^{-2D}]_{;\ell}|,\qquad\text{for \eqref{eq:HMHF-equiv}}.
\end{align*}
We then note the estimates 
\begin{equation}
\sum_{\ell>j}|[\lan y\ran^{-2D}]_{;\ell}|\aleq\sum_{\ell>j}\lmb_{\ell}^{D}r^{-2D}\aleq\lmb_{j+1}^{D}r^{-2D}\quad\text{and}\quad\sum_{\ell<j}|[\lan y\ran^{-2D}]_{;\ell}|\aleq\sum_{\ell<j}\lmb_{\ell}^{-D}\aleq\lmb_{j-1}^{-D}.\label{eq:Q-ell-01}
\end{equation}
Now, observe for $r\in[A_{0}\br{\lmb}_{j+1},A_{0}^{-1}\br{\lmb}_{j}]$
that 
\begin{align*}
\chf_{[A_{0}\br{\lmb}_{j+1},\lmb_{j}]}\{\lmb_{j+1}^{D}r^{-2D}+\lmb_{j-1}^{-D}\} & \aleq(A_{0}^{-2D}+\mu_{j}^{D})\lmb_{j}^{-D}\aeq(\mu_{j}^{D}+A_{0}^{-2D})|[\lan y\ran^{-2D}]_{;j}|,\\
\chf_{[\lmb_{j},A_{0}^{-1}\br{\lmb}_{j}]}\{\lmb_{j+1}^{D}r^{-2D}+\lmb_{j-1}^{-D}\} & \aleq(\mu_{j+1}^{D}+A_{0}^{-2D})\lmb_{j}^{D}r^{-2D}\aeq(\mu_{j+1}^{D}+A_{0}^{-2D})|[\lan y\ran^{-2D}]_{;j}|.
\end{align*}
Taking $A_{0}>10$ large, we obtain \eqref{eq:U-bd-1}. Once $A_{0}$
is fixed, we can ignore the dependence on $A_{0}$; the estimates
\eqref{eq:U-bd-2}--\eqref{eq:U-bd-3} and \eqref{eq:U-bd-2-HM}--\eqref{eq:U-bd-3-HM}
easily follow from \eqref{eq:Q-ell-01}.
\end{proof}
\begin{lem}
\label{lem:corrector}The corrector $\td U$ satisfies the estimates
\eqref{eq:tdU-H1} and \eqref{eq:tdU-weight-L1}.
\end{lem}

\begin{proof}
For the $\dot{H}^{1}$-estimate \eqref{eq:tdU-H1}, we use \eqref{eq:vphi-bd}
and $D>2$ to have 
\[
\|\chi_{\dlt_{0}\lmb_{\ell-1}}\vphi_{\ell}\|_{\dot{H}^{1}}\aleq\lmb_{\ell-1}^{-D}\|r^{-1}\omg_{\ell}\|_{L^{2}}\aleq\lmb_{\ell-1}^{-D}\lmb_{\ell}^{2}\br{\lmb}_{\ell}^{D-2}\aeq\mu_{\ell}^{\frac{D+2}{2}}=\dlt(\alp^{\ast}).
\]
Summing this over $\ell=2,\dots,J$ gives \eqref{eq:tdU-H1}.

For the $\lmb$-derivative estimate \eqref{eq:tdU-weight-L1}, we
use \eqref{eq:vphi-bd} and \eqref{eq:vphi-deriv} to have 
\[
\lmb_{j}\rd_{\lmb_{j}}(\chi_{\dlt_{0}\lmb_{\ell-1}}\vphi_{\ell})\aleq\chf_{j\leq\ell}\min\{\lmb_{\ell-1}^{-D},\lmb_{j}^{-D}\}\omg_{\ell}.
\]
Multiplying the above by $[\lan y\ran^{-2D}]_{\ul{;k}}$ and summing
over $\ell$, we obtain a preliminary bound 
\begin{equation}
\|\lmb_{j}\rd_{\lmb_{j}}\td U\cdot[\lan y\ran^{-2D}]_{\ul{;k}}\|_{L^{1}}\aleq\lmb_{j-1}^{-D}\|\omg_{j}\cdot[\lan y\ran^{-2D}]_{\ul{;k}}\|_{L^{1}}+\lmb_{j}^{-D}\sum_{\ell\geq j+1}\|\omg_{\ell}\cdot[\lan y\ran^{-2D}]_{\ul{;k}}\|_{L^{1}}.\label{eq:lmb_td_U-01}
\end{equation}
Now we estimate $\|\omg_{\ell}\cdot[\lan y\ran^{-2D}]_{\ul{;k}}\|_{L^{1}}$.
If $\ell\leq k$, then we use the bound $\chf_{(0,2\dlt_{0}\lmb_{\ell-1}]}[\lan y\ran^{-2D}]_{\ul{;k}}\aleq\lmb_{k}^{D-2}r^{-2D}$
to have 
\[
\chf_{k\geq\ell}\|\omg_{\ell}\cdot[\lan y\ran^{-2D}]_{\ul{;k}}\|_{L^{1}}\aleq\lmb_{\ell}^{2}\lmb_{k}^{D-2}|\log(\br{\lmb_{\ell}}/\lmb_{\ell})|\aleq\lmb_{\ell}^{2}\lmb_{k}^{D-2}|\log\mu_{\ell}|.
\]
If $\ell\geq k+1$, then we use the bound $\chf_{(0,2\dlt_{0}\lmb_{\ell-1}]}[\lan y\ran^{-2D}]_{\ul{;k}}\aleq\lmb_{k}^{-(D+2)}$
to have 
\[
\chf_{k\leq\ell-1}\|\omg_{\ell}\cdot[\lan y\ran^{-2D}]_{\ul{;k}}\|_{L^{1}}\aleq\lmb_{\ell-1}^{D-2}\lmb_{\ell}^{D}\lmb_{k}^{-(D+2)}(2\dlt_{0}\lmb_{\ell-1})^{4}\aleq\lmb_{\ell-1}^{D+2}\lmb_{\ell}^{D}\lmb_{k}^{-(D+2)}.
\]
Substituting these two bounds into \eqref{eq:lmb_td_U-01}, we obtain
\begin{align*}
 & \chf_{j\geq k+1}\|\lmb_{j}\rd_{\lmb_{j}}\td U\cdot[\lan y\ran^{-2D}]_{\ul{;k}}\|_{L^{1}}\\
 & \quad\aleq\lmb_{j-1}^{2}\lmb_{j}^{D}\lmb_{k}^{-(D+2)}+\lmb_{j}^{-D}\sum_{\ell\geq j+1}\lmb_{\ell-1}^{D+2}\lmb_{\ell}^{D}\lmb_{k}^{-(D+2)}\aeq\lmb_{j-1}^{2}\lmb_{j}^{D}\lmb_{k}^{-(D+2)}
\end{align*}
and 
\begin{align*}
 & \chf_{j\leq k}\|\lmb_{j}\rd_{\lmb_{j}}\td U\cdot[\lan y\ran^{-2D}]_{\ul{;k}}\|_{L^{1}}\\
 & \aleq\lmb_{j-1}^{-D}\lmb_{j}^{2}\lmb_{k}^{D-2}|\log\mu_{j}|+\lmb_{j}^{-D}\Big(\sum_{\ell=j+1}^{k}\lmb_{\ell}^{2}\lmb_{k}^{D-2}|\log\mu_{\ell}|+\sum_{\ell\geq k+1}\lmb_{\ell-1}^{D+2}\lmb_{\ell}^{D}\lmb_{k}^{-(D+2)}\Big)\\
 & \aeq\lmb_{j-1}^{-D}\lmb_{j}^{2}\lmb_{k}^{D-2}|\log\mu_{j}|+\lmb_{j}^{-D}(\chf_{j\leq k-1}\lmb_{j+1}^{2}\lmb_{k}^{D-2}|\log\mu_{j+1}|+\lmb_{k+1}^{D}).
\end{align*}
Therefore, we have proved 
\begin{align*}
 & \|\lmb_{j}\rd_{\lmb_{j}}\td U\cdot[\lan y\ran^{-2D}]_{\ul{;k}}\|_{L^{1}}\\
 & \quad\aleq\left\{ \begin{aligned} & \Big(\frac{\lmb_{j-1}}{\lmb_{k}}\Big)^{2}\Big(\frac{\lmb_{j}}{\lmb_{k}}\Big)^{D} &  & \text{if }j\geq k+1,\\
 & \mu_{k+1}^{D}+\mu_{k}^{D}|\log\mu_{k}| &  & \text{if }j=k,\\
 & (\mu_{j+1}^{2}|\log\mu_{j+1}|+\mu_{j}^{D}|\log\mu_{j}|)\cdot\Big(\frac{\lmb_{k}}{\lmb_{j}}\Big)^{D-2} &  & \text{if }j\leq k-1,
\end{aligned}
\right.
\end{align*}
which implies \eqref{eq:tdU-weight-L1}. This completes the proof.
\end{proof}
\begin{lem}
\label{lem:main-term-size}The estimate \eqref{eq:prof-eq-main-size}
holds, i.e., 
\[
\Big\|\sum_{j=2}^{J}\br{\iota}_{j}\frac{\kappa\mu_{j}^{D}}{\lmb_{j}^{2}}\Lmb W_{;j}\Big\|_{L^{2}}^{2}\aeq\calD.
\]
\end{lem}

\begin{proof}
This estimate easily follows from the almost orthogonality among $\Lmb W_{;j}$s
with $\Lmb W\aleq\lan y\ran^{-2D}$:
\begin{equation}
\chf_{j\neq k}\|\tfrac{1}{\lmb_{j}}[\lan y\ran^{-2D}]_{;j}\cdot\tfrac{1}{\lmb_{k}}[\lan y\ran^{-2D}]_{;k}\|_{L^{1}}\aeq\chf_{j>k}\Big(\frac{\lmb_{j}}{\lmb_{k}}\Big)^{D-1}+\chf_{j<k}\Big(\frac{\lmb_{k}}{\lmb_{j}}\Big)^{D-1}\aleq\dlt(\alp^{\ast}).\label{eq:almost-orthog}
\end{equation}
This completes the proof.
\end{proof}
\begin{lem}
\label{lem:Psi}The inhomogeneous error term
\[
\Psi=\Dlt U+r^{-(D+2)}f(P)+\sum_{j=2}^{J}\br{\iota}_{j}\frac{\kappa\mu_{j}^{D}}{\lmb_{j}^{2}}\Lmb W_{;j}
\]
satisfies the estimates \eqref{eq:Psi-L2} and \eqref{eq:Psi-inn}.
\end{lem}

\begin{proof}
For $\ell\in\{2,\dots,J\}$, define 
\begin{equation}
\Psi_{\ell}\coloneqq\Dlt U_{\ell}+r^{-(D+2)}f(P_{\ell})+\sum_{j=2}^{\ell}\br{\iota}_{j}\frac{\kappa\mu_{j}^{D}}{\lmb_{j}^{2}}\Lmb W_{;j}.\label{eq:Psi-ell-def}
\end{equation}
Define $\Psi_{1}\equiv0$. We need to estimate $\Psi_{J}$.

\textbf{Step 1.} Recursive expression of $\Psi_{\ell}$.

In this step, we derive a recursive expression for $\Psi_{\ell}$
(see \eqref{eq:Psi-ell-exp} below) for $\ell\geq2$. Decomposing
$U_{\ell}=U_{\ell-1}+W_{;\ell}+\chi_{\dlt_{0}\lmb_{\ell-1}}\vphi_{\ell}$
and applying $\Dlt W_{;\ell}+r^{-(D+2)}f(Q_{;\ell})=0$ and the definition
of $\Psi_{\ell-1}$, we obtain 
\[
\Psi_{\ell}=\Dlt(\chi_{\dlt_{0}\lmb_{\ell-1}}\vphi_{\ell})+r^{-(D+2)}\{f(P_{\ell})-f(P_{\ell-1})-f(Q_{;\ell})\}+\br{\iota}_{\ell}\frac{\kappa\mu_{\ell}^{D}}{\lmb_{\ell}^{2}}\Lmb W_{;\ell}+\Psi_{\ell-1}.
\]
Now we recall the definitions of $\calR_{\ell}^{(0)}$ and $\calR_{\ell}^{(1)}$
from Corollary~\ref{cor:Profiles-vphi}, and note 
\[
r^{-(D+2)}\{f(P_{\ell})-f(P_{\ell-1})-f(Q_{;\ell})\}=r^{-2}f'(Q_{;\ell})\chi_{\dlt_{0}\lmb_{\ell-1}}\vphi_{\ell}+\calR_{\ell}^{(0)}+\calR_{\ell}^{(1)}.
\]
By the above display and the equation \eqref{eq:vphi-eqn} for $\vphi_{\ell}$,
we arrive at 
\begin{equation}
\begin{aligned}\Psi_{\ell} & =\Big\{\br{\iota}_{\ell}\frac{\kappa\mu_{\ell}^{D}}{\lmb_{\ell}^{2}}\Lmb W_{;\ell}+\frac{\frkr_{\ell}}{\lmb_{\ell}^{2}}\chi_{\dlt_{0}\lmb_{\ell-1}}\Lmb W_{;\ell}\Big\}\\
 & \peq+(1-\chi_{\dlt_{0}\lmb_{\ell-1}}^{2})\calR_{\ell}^{(0)}+(1-\chi_{\dlt_{0}\lmb_{\ell-1}})\calR_{\ell}^{(1)}+[\Dlt,\chi_{\dlt_{0}\lmb_{\ell-1}}]\vphi_{\ell}+\Psi_{\ell-1}.
\end{aligned}
\label{eq:Psi-ell-exp}
\end{equation}

\textbf{Step 2.} Estimates.

In this step, we show the estimates \eqref{eq:Psi-L2} and \eqref{eq:Psi-inn}.
We will estimate $\Psi$ using 
\[
|\Psi|\leq\sum_{\ell=2}^{J}|\Psi_{\ell}-\Psi_{\ell-1}|
\]
and the expression \eqref{eq:Psi-ell-exp}.

First, by \eqref{eq:frkr-2}, we have 
\[
\Big|\frac{\frkr_{\ell}+\kappa\br{\iota}_{\ell}\mu_{\ell}^{D}}{\lmb_{\ell}^{2}}\chi_{\dlt_{0}\lmb_{\ell-1}}\Lmb W_{;\ell}\Big|\aleq\frac{\dlt(\alp^{\ast})\mu_{\ell}^{D}}{\lmb_{\ell}^{2}}[\lan y\ran^{-2D}]_{\lmb_{\ell}}.
\]
Hence we have the $L^{2}$-estimate
\[
\sum_{\ell=2}^{J}\Big\|\frac{\frkr_{\ell}+\kappa\br{\iota}_{\ell}\mu_{\ell}^{D}}{\lmb_{\ell}^{2}}\chi_{\dlt_{0}\lmb_{\ell-1}}\Lmb W_{;\ell}\Big\|_{L^{2}}\aleq\sum_{\ell=2}^{J}\frac{\dlt(\alp^{\ast})\mu_{\ell}^{D}}{\lmb_{\ell}}\aleq\dlt(\alp^{\ast})\sqrt{\calD}.
\]
For the weighted $L^{1}$-estimate, fix any $k\in\{1,\dots,J\}$ and
note the estimate 
\begin{equation}
\|[\lan y\ran^{-2D}]_{;\ell}[\lan y\ran^{-2D}]_{\ul{;k}}\|_{L^{1}}\aeq\chf_{\ell<k}\frac{\lmb_{k}^{D-2}}{\lmb_{\ell}^{D-2}}+\chf_{\ell=k}+\chf_{\ell>k}\frac{\lmb_{\ell}^{D}}{\lmb_{k}^{D}}.\label{eq:Qell-Qk}
\end{equation}
Hence we have 
\[
\Big\|\frac{\dlt(\alp^{\ast})\mu_{\ell}^{D}}{\lmb_{\ell}^{2}}[\lan y\ran^{-2D}]_{;\ell}[\lan y\ran^{-2D}]_{\ul{;k}}\Big\|_{L^{1}}\aleq\dlt(\alp^{\ast})\Big\{\chf_{\ell<k}\frac{\mu_{\ell}^{D}\mu_{k}^{D}}{\lmb_{k}^{2}}+\chf_{\ell=k}\frac{\mu_{k}^{D}}{\lmb_{k}^{2}}+\chf_{\ell>k}\frac{\mu_{\ell}^{D}}{\lmb_{\ell}}\frac{\mu_{k+1}^{D}}{\lmb_{k+1}}\Big\}.
\]
Hence we obtain the weighted $L^{1}$-estimate
\[
\sum_{\ell=2}^{J}\Big\|\frac{\frkr_{\ell}+\kappa\br{\iota}_{\ell}\mu_{\ell}^{D}}{\lmb_{\ell}^{2}}\chi_{\dlt_{0}\lmb_{\ell-1}}\Lmb W_{;\ell}\cdot[\lan y\ran^{-2D}]_{\ul{;k}}\Big\|_{L^{1}}\aleq\dlt(\alp^{\ast})\Big(\frac{\mu_{k+1}^{D}+\mu_{k}^{D}}{\lmb_{k}^{2}}+\calD\Big).
\]

Next, we note that 
\[
\Big|\frac{\kappa\br{\iota}_{\ell}\mu_{\ell}^{D}}{\lmb_{\ell}^{2}}(1-\chi_{\dlt_{0}\lmb_{\ell-1}})\Lmb W_{;\ell}\Big|\aleq\chf_{[\dlt_{0}\lmb_{\ell-1},+\infty)}\mu_{\ell}^{D}\frac{\lmb_{\ell}^{D-2}}{r^{2D}}.
\]
We have $L^{2}$-estimate 
\[
\sum_{\ell=2}^{J}\Big\|\frac{\kappa\br{\iota}_{\ell}\mu_{\ell}^{D}}{\lmb_{\ell}^{2}}(1-\chi_{\dlt_{0}\lmb_{\ell-1}})\Lmb W_{;\ell}\Big\|_{L^{2}}\aleq\sum_{\ell=2}^{J}\frac{\mu_{\ell}^{2D-1}}{\lmb_{\ell}}\aleq\dlt(\alp^{\ast})\sqrt{\calD}.
\]
For the weighted $L^{1}$-estimate, fix any $k\in\{1,\dots,J\}$ and
observe that
\begin{align*}
\Big\|\chf_{[\dlt_{0}\lmb_{\ell-1},+\infty)}\mu_{\ell}^{D}\frac{\lmb_{\ell}^{D-2}}{r^{2D}}\cdot[\lan y\ran^{-2D}]_{\ul{;k}}\Big\|_{L^{1}} & \aleq\chf_{\ell\geq k+1}\mu_{\ell}^{D}\frac{\lmb_{\ell}^{D-2}}{\lmb_{k}^{D}}+\chf_{\ell<k+1}\mu_{\ell}^{D}\frac{\lmb_{\ell}^{D-2}\lmb_{k}^{D-2}}{\lmb_{\ell-1}^{2D-2}}\\
 & \aleq\chf_{\ell\geq k+1}\frac{\mu_{\ell}^{D}\mu_{k+1}^{D}}{\lmb_{\ell}\lmb_{k+1}}+\chf_{\ell<k+1}\frac{\mu_{\ell}^{3D-2}}{\lmb_{\ell}^{2}}\aleq\calD
\end{align*}
so we have 
\[
\sum_{\ell=2}^{J}\Big\|\frac{\kappa\br{\iota}_{\ell}\mu_{\ell}^{D}}{\lmb_{\ell}^{2}}(1-\chi_{\dlt_{0}\lmb_{\ell-1}})\Lmb W_{;\ell}\cdot[\lan y\ran^{-2D}]_{\ul{;k}}\Big\|_{L^{1}}\aleq\calD.
\]

We estimate the contribution of $(1-\chi_{\dlt_{0}\lmb_{\ell-1}}^{2})\calR_{\ell}^{(0)}$.
Recall from the proof of Lemma~\ref{lem:A_0} that $|U_{\ell-1}|\aleq\sum_{j=1}^{\ell-1}|[\lan y\ran^{-2D}]_{;j}|$
for \eqref{eq:NLH-rad} and $|r^{-D}\sin(P_{\ell-1})|\aleq\sum_{j=1}^{\ell-1}|[\lan y\ran^{-2D}]_{;j}|$
for \eqref{eq:HMHF-equiv}. Using \eqref{eq:f-05} and \eqref{eq:f-05-HM},
we have a pointwise estimate
\[
\calR_{\ell}^{(0)}\aleq\sum_{j=1}^{\ell-1}|[\lan y\ran^{-2D}]_{;j}|^{p-1}|[\lan y\ran^{-2D}]_{;\ell}|\aleq\sum_{j=1}^{\ell-1}\min\Big\{\frac{1}{\lmb_{j}^{2}},\frac{\lmb_{j}^{2}}{r^{4}}\Big\}\frac{\lmb_{\ell}^{D}}{r^{2D}}.
\]
For the $L^{2}$-estimate, we observe for $j\leq\ell-1$ 
\[
\Big\|\chf_{[\dlt_{0}\lmb_{\ell-1},+\infty)}\min\Big\{\frac{1}{\lmb_{j}^{2}},\frac{\lmb_{j}^{2}}{r^{4}}\Big\}\frac{\lmb_{\ell}^{D}}{r^{2D}}\Big\|_{L^{2}}\aleq\frac{\lmb_{\ell}^{D}}{\lmb_{j}^{2}\lmb_{\ell-1}^{D-1}}\aleq\frac{\mu_{\ell}^{D+1}}{\lmb_{\ell}},
\]
so summing over $j=1,\dots,\ell-1$ and $\ell=2,\dots,J$ gives 
\[
\sum_{\ell=2}^{J}\|\chf_{[\dlt_{0}\lmb_{\ell-1},+\infty)}\calR_{\ell}^{(0)}\|_{L^{2}}\aleq\sum_{\ell=2}^{J}\frac{\mu_{\ell}^{D+1}}{\lmb_{\ell}}\aleq\dlt(\alp^{\ast})\sqrt{\calD}.
\]
For the weighted $L^{1}$-estimate, fix any $k\in\{1,\dots,J\}$ and
observe for $j\leq\ell-1$
\begin{align*}
 & \Big\|\chf_{[\dlt_{0}\lmb_{\ell-1},+\infty)}\min\Big\{\frac{1}{\lmb_{j}^{2}},\frac{\lmb_{j}^{2}}{r^{4}}\Big\}\frac{\lmb_{\ell}^{D}}{r^{2D}}\cdot[\lan y\ran^{-2D}]_{\ul{;k}}\Big\|_{L^{1}}\\
 & \quad\aleq\chf_{k\leq j}\frac{\lmb_{\ell}^{D}}{\lmb_{k}^{D+2}}+\chf_{j<k<\ell}\frac{\lmb_{\ell}^{D}}{\lmb_{j}^{2}\lmb_{k}^{D}}+\chf_{k\geq\ell}\frac{\lmb_{k}^{D-2}\lmb_{\ell}^{D}}{\lmb_{j}^{2}\lmb_{\ell-1}^{2D-2}}\\
 & \quad\aleq\chf_{k\leq j}\frac{\mu_{k+1}^{D}}{\lmb_{k}^{2}}+\chf_{j<k<\ell}\frac{\dlt(\alp^{\ast})\mu_{k+1}^{D}}{\lmb_{k}^{2}}+\chf_{k\geq\ell}\frac{\mu_{\ell}^{D}\mu_{k}^{D}}{\lmb_{k}^{2}}.
\end{align*}
Thus we have proved for $j\leq\ell-1$
\[
\Big\|\chf_{[\dlt_{0}\lmb_{\ell-1},+\infty)}\min\Big\{\frac{1}{\lmb_{j}^{2}},\frac{\lmb_{j}^{2}}{r^{4}}\Big\}\frac{\lmb_{\ell}^{D}}{r^{2D}}\cdot[\lan y\ran^{-2D}]_{\ul{;k}}\Big\|_{L^{1}}\aleq\frac{\mu_{k+1}^{D}+\dlt(\alp^{\ast})\mu_{k}^{D}}{\lmb_{k}^{2}}.
\]
Summing this estimate over $j=1,\dots,\ell-1$ and $\ell=2,\dots,J$
gives 
\[
\sum_{\ell=2}^{J}\Big\|\chf_{[\dlt_{0}\lmb_{\ell-1},+\infty)}\calR_{\ell}^{(0)}\cdot[\lan y\ran^{-2D}]_{\ul{;k}}\Big\|_{L^{1}}\aleq\frac{\mu_{k+1}^{D}+\dlt(\alp^{\ast})\mu_{k}^{D}}{\lmb_{k}^{2}}.
\]

Finally, we estimate $(1-\chi_{\dlt_{0}\lmb_{\ell-1}})\calR_{\ell}^{(1)}$
and $[\Dlt,\chi_{\dlt_{0}\lmb_{\ell-1}}]\vphi_{\ell}$. We notice
that this term is supported on $[\dlt_{0}\lmb_{\ell-1},2\dlt_{0}\lmb_{\ell-1}]$.
We can use the bound \eqref{eq:R1-01} for $\calR_{\ell}^{(1)}$ with
$K_{0}\aeq1$ and $J+1=\ell$, and the bound \eqref{eq:vphi-bd} for
$\vphi_{\ell}$ to obtain 
\begin{align*}
 & |(1-\chi_{\dlt_{0}\lmb_{\ell-1}})\calR_{\ell}^{(1)}|+|[\Dlt,\chi_{\dlt_{0}\lmb_{\ell-1}}]\vphi_{\ell}|\\
 & \quad\aleq\chf_{[\dlt_{0}\lmb_{\ell-1},2\dlt_{0}\lmb_{\ell-1}]}r^{-2}\lmb_{\ell-1}^{-D}{}_{\ell}\aleq\chf_{[\dlt_{0}\lmb_{\ell-1},2\dlt_{0}\lmb_{\ell-1}]}\frac{\lmb_{\ell}^{D}}{\lmb_{\ell-1}^{2D+2}}.
\end{align*}
Thus we have the $L^{2}$-estimate 
\[
\|(1-\chi_{\dlt_{0}\lmb_{\ell-1}})\calR_{\ell}^{(1)}\|_{L^{2}}+\|[\Dlt,\chi_{\dlt_{0}\lmb_{\ell-1}}]\vphi_{\ell}\|_{L^{2}}\aleq\frac{\lmb_{\ell}^{D}}{\lmb_{\ell-1}^{D+1}}\aeq\frac{\mu_{\ell}^{D+1}}{\lmb_{\ell}}\aleq\dlt(\alp^{\ast})\sqrt{\calD}.
\]
For the weighted $L^{1}$-estimate, fix any $k\in\{1,\dots,J\}$ and
observe 
\[
\Big\|\chf_{[\dlt_{0}\lmb_{\ell-1},2\dlt_{0}\lmb_{\ell-1}]}\frac{\lmb_{\ell}^{D}}{\lmb_{\ell-1}^{2D+2}}\cdot[\lan y\ran^{-2D}]_{\ul{;k}}\Big\|_{L^{1}}\aleq\chf_{k<\ell}\frac{\lmb_{\ell}^{D}}{\lmb_{k}^{D+2}}+\chf_{k\geq\ell}\frac{\lmb_{\ell}^{D}\lmb_{k}^{D-2}}{\lmb_{\ell-1}^{2D}}\aleq\frac{\mu_{k+1}^{D}+\dlt(\alp^{\ast})\mu_{k}^{D}}{\lmb_{k}^{2}}.
\]
Summing this over $\ell=2,\dots,J$ gives 
\[
\sum_{\ell=2}^{J}\Big\|\big\{(1-\chi_{\dlt_{0}\lmb_{\ell-1}})\calR_{\ell}^{(1)}+[\Dlt,\chi_{\dlt_{0}\lmb_{\ell-1}}]\vphi_{\ell}\big\}\cdot[\lan y\ran^{-2D}]_{\ul{;k}}\Big\|_{L^{1}}\aleq\frac{\mu_{k+1}^{D}+\dlt(\alp^{\ast})\mu_{k}^{D}}{\lmb_{k}^{2}}.
\]
This completes the proof.
\end{proof}
The proof of Proposition~\ref{prop:Modified-Profile} follows from
Lemmas~\ref{lem:A_0}, \ref{lem:corrector}, \ref{lem:main-term-size},
and \ref{lem:Psi}.

\section{\label{sec:Spacetime-estimate}Spacetime estimate}

Let $u(t)$ be a solution as in the assumption of Theorem~\ref{thm:main}.
Note that $J\geq1$ and $\vec{\iota}=(\iota_{1},\dots,\iota_{J})$
are fixed. The goal of this section is to decompose 
\[
u(t)=U(\vec{\iota},\vec{\lmb}(t);\cdot)+g(t)
\]
(the new scales $\vec{\lmb}(t)=(\lmb_{1}(t),\dots,\lmb_{J}(t))$ will
be different from the scales introduced in the assumption of Theorem~\ref{thm:main}
in general) such that $g(t)$ satisfies an $L^{2}$ in time control:
\[
\int_{t_{0}}^{T}\|g(t)\|_{\dot{H}^{2}}^{2}dt<+\infty.
\]
As explained in Section~\ref{subsec:Strategy}, this $L_{t}^{2}$-control
is crucial when we integrate modulation equations in the next section;
it guarantees that nonlinear terms of $g(t)$ do not contribute to
the dynamics of $\vec{\lmb}(t)$. Moreover, the modified profile $U(\vec{\iota},\vec{\lmb};\cdot)$
was introduced to achieve this $L^{2}$ in time control. Due to its
construction, $\mu_{j}(t)$s also turn out to enjoy some integrability
in time.

To fix the new scales $\vec{\lmb}(t)$ for each time, we will impose
$J$ orthogonality conditions $\lan g(t),\calZ_{\lmb_{k}(t)}\ran=0$,
$k=1,\dots,J$. We fix a function $\calZ$ such that 
\begin{equation}
\calZ\in C_{\rad}^{\infty},\quad\lan\calZ,\Lmb W\ran=1,\quad\mathrm{supp}(\calZ)\subseteq[R_{0}^{-1},R_{0}]\text{ for some }R_{0}>10.\label{eq:calZ-prop}
\end{equation}
One may explicitly choose $\calZ=c(1-\chi_{R_{0}^{-1}})\chi_{R_{0}/2}\Lmb W$
for some $R_{0}>10$ and a constant $c$. The following is the main
result of this section.
\begin{prop}[Decomposition and spacetime estimate]
\label{prop:spacetime-est}Let $u(t)$ satisfy the assumption of
Theorem~\ref{thm:main}. Then, there exist $t_{0}<T$ and a $C^{1}$-curve\footnote{Recall $\alp_{0}$ from Proposition~\ref{prop:Modified-Profile}.}
$\vec{\lmb}:[t_{0},T)\to\calP_{J}(\alp_{0})$ such that the following
holds on the time interval $[t_{0},T)$:
\begin{itemize}
\item (Decomposition) We still have \eqref{eq:main-thm-decom} with the
new choice of parameters $\vec{\lmb}(t)$. The refined decomposition
of $u(t)$ as 
\begin{equation}
u(t)=U(\vec{\iota},\vec{\lmb}(t);\cdot)+g(t)\eqqcolon U(t)+g(t)\label{eq:new-decomp}
\end{equation}
satisfies the orthogonality conditions
\begin{equation}
\lan g(t),\calZ_{\lmb_{k}(t)}\ran=0\qquad\forall k\in\{1,\dots,J\},\label{eq:orthog}
\end{equation}
where $U(\vec{\iota},\vec{\lmb};r)$ is the modified multi-bubble
profile defined in Proposition~\ref{prop:Modified-Profile}.
\item (Qualitative smallness) We have 
\begin{align}
\|\lan r/\lmb_{1}(t)\ran^{-1}g(t)\|_{\dot{H}^{1}} & =o(1),\label{eq:g_H1loc_o(1)}\\
\max_{2\leq j\leq J}\mu_{j}(t) & =o(1).\label{eq:mu_j_o(1)}
\end{align}
Moreover, we have 
\begin{align}
\|\chi_{4r_{0}}g(t)\|_{\dot{H}^{1}} & =\dlt(r_{0})+o(1),\label{eq:chi-r0-g}\\
\|g(t)\|_{\dot{H}^{1}} & =o(1)\qquad\text{if }T=+\infty.\label{eq:global_g_H1_o(1)}
\end{align}
\item (Spacetime estimate) Recall that $\calD=\sum_{j=2}^{J}\mu_{j}^{2D}/\lmb_{j}^{2}$.
We have 
\begin{equation}
\int_{t_{0}}^{T}\Big\{\calD+\|g(t)\|_{\dot{H}^{2}}^{2}\Big\} dt<+\infty.\label{eq:sptime-est}
\end{equation}
\end{itemize}
\end{prop}

We remark that the decomposition and qualitative smallness parts of
the above proposition are standard. The heart of this proposition
is the spacetime estimate \eqref{eq:sptime-est}.

\subsection{Decomposition}

First, we state a decomposition lemma near radial multi-bubbles, which
will follow from the implicit function theorem. For $m\in\bbZ$, $v\in\calH_{0,m}$
(see \eqref{eq:def-calH-ell-m}), $\vec{\nu}\in(0,+\infty)^{J}$,
and $\dlt>0$, we denote 
\begin{align*}
B_{\dlt}(v) & \coloneqq\{u\in\calH_{0,m}:\|u-v\|_{\dot{H}^{1}}<\dlt\},\\
B_{\dlt}(\vec{\nu}) & \coloneqq\{\vec{\lmb}\in(0,+\infty)^{J}:\mathrm{dist}(\vec{\nu},\vec{\lmb})<\dlt\},\\
\mathrm{dist}(\vec{\nu},\vec{\lmb}) & \coloneqq\max_{1\leq j\leq J}|\log(\nu_{j}/\lmb_{j})|.
\end{align*}
We also denote 
\[
W_{\vec{\iota},\vec{\nu}}\coloneqq\sum_{j=1}^{J}\iota_{j}W_{\nu_{j}}.
\]
The proof of the following lemma is similar to \cite[Lemma 4.11]{JendrejLawrie2021arXiv}
and \cite[Lemma B.1]{DuyckaertsKenigMerle2023Acta}.
\begin{lem}[Decomposition near $W_{\vec{\iota},\vec{\nu}}$]
\label{lem:decom-near-multi-bubble}Let $J\ge1$. There exist small
constants $\alp_{\dec}\in(0,\alp_{0})$ and $\dlt_{\dec},\eta_{\dec}>0$
with the following property. For any $\vec{\iota}\in\{\pm1\}^{J}$
and $\vec{\nu}\in\calP_{J}(\alp_{\dec})$, there exists a $C^{1}$-function
$G^{\vec{\iota},\vec{\nu}}:B_{\dlt_{\dec}}(W_{\vec{\iota},\vec{\nu}})\to B_{\eta_{\dec}}(\vec{\nu})$
such that:
\begin{itemize}
\item For all $u\in B_{\dlt_{\dec}}(W_{\vec{\iota},\vec{\nu}})$ and $k\in\{1,\dots,J\}$,
we have $\lan u-U(\vec{\iota},G^{\vec{\iota},\vec{\nu}}(u);\cdot),\calZ_{G_{k}^{\vec{\iota},\vec{\nu}}(u)}\ran=0$.
\item If $\lan u-U(\vec{\iota},\vec{\lmb};\cdot),\calZ_{\lmb_{k}}\ran=0$
for all $k\in\{1,\dots,J\}$ for some $u\in B_{\dlt_{\dec}}(W_{\vec{\iota},\vec{\nu}})$
and $\vec{\lmb}\in B_{\eta_{\dec}}(\vec{\nu})$, then $\vec{\lmb}=G^{\vec{\iota},\vec{\nu}}(u)$.
\item For any $\eta>0$, there exist $\alp_{1}\in(0,\alp_{\dec}]$ and $\dlt_{1}\in(0,\dlt_{\dec}]$
such that $\mathrm{dist}(\vec{\nu},G^{\vec{\iota},\vec{\nu}}(u))\leq\eta$
whenever $\vec{\nu}\in\calP(\alp_{1})$ and $u\in B_{\dlt_{1}}(W_{\vec{\iota},\vec{\nu}})$.
\end{itemize}
\end{lem}

\begin{proof}
We only sketch the proof. Let $\vec{\iota}\in\{\pm1\}^{J}$ and $m_{\Sigma}=\sum_{j=1}^{J}\iota_{j}$.
We consider the function $F^{\vec{\iota}}:\calP_{J}(\alp_{0})\times\calH_{0,m_{\Sigma}}\to\bbR^{J}$
whose components $F_{k}^{\vec{\iota}}$, $k=1,\dots,J$, are defined
by 
\[
F_{k}^{\vec{\iota}}(\vec{\lmb};u)=\lan u-U(\vec{\iota},\vec{\lmb};\cdot),\iota_{k}\calZ_{\ul{\lmb_{k}}}\ran.
\]
Now let $\alp_{1}\in(0,\alp_{0})$, $\vec{\nu}\in\calP_{J}(\alp_{1})$,
and $u\in B_{\dlt_{1}}(W_{\vec{\iota},\vec{\nu}})$. First, by \eqref{eq:tdU-H1},
we have 
\[
F_{k}^{\vec{\iota}}(\vec{\nu};W_{\vec{\iota},\vec{\nu}})=-\lan\td U,\iota_{k}\calZ_{\ul{\nu_{k}}}\ran\aleq\dlt(\alp_{1}).
\]
Next, by $\lan\Lmb W,\calZ\ran=1$, \eqref{eq:almost-orthog}, and
\eqref{eq:tdU-weight-L1}, we have 
\begin{align*}
 & \lmb_{j}\rd_{\lmb_{j}}F_{k}^{\vec{\iota}}(\vec{\lmb};u)\\
 & =-\lan\iota_{j}\Lmb W_{\lmb_{j}},\iota_{k}\calZ_{\ul{\lmb_{k}}}\ran+\lan\lmb_{j}\rd_{\lmb_{j}}\td U,\iota_{k}\calZ_{\ul{\lmb_{k}}}\ran+\chf_{j=k}\lan u-U,\iota_{k}\Lmb_{-1}\calZ_{\ul{\lmb_{k}}}\ran\\
 & =-(\chf_{j=k}+\dlt(\alp_{1}))+O(\|u-U\|_{\dot{H}^{1}})\\
 & =-\chf_{j=k}+\dlt(\alp_{1}+\dlt_{1}).
\end{align*}
We also note that $F_{k}$ is $C^{1}$ in $u$ with the Lipschitz
bound 
\[
F_{k}^{\vec{\iota}}(\vec{\lmb};u)-F_{k}^{\vec{\iota}}(\vec{\lmb};u')=\lan u-u',\iota_{k}\calZ_{\ul{\lmb_{k}}}\ran\aleq\|u-u'\|_{\dot{H}^{1}}.
\]
By the previous three displays, one can apply the standard implicit
function theorem to conclude the proof. See \cite[Lemma 4.11]{JendrejLawrie2021arXiv}
and \cite[Lemma B.1]{DuyckaertsKenigMerle2023Acta} for more details.
This completes the sketch of the proof.
\end{proof}
\begin{cor}[Decomposition of solution]
\label{cor:Decomp-Sol}Let $u(t)$ satisfy the assumption of Theorem~\ref{thm:main}.
Then, there exist $t_{\dec}\in(0,T)$ and a $C^{1}$-curve $\vec{\lmb}:[t_{\dec},T)\to\calP_{J}(\alp_{0})$
such that on the time interval $[t_{\dec},T)$ the decomposition and
qualitative smallness parts of Proposition~\ref{prop:spacetime-est}
hold.
\end{cor}

\begin{proof}
\textbf{Step 1.} Construction of the curve $\vec{\lmb}(t)$.

Let $u(t)$ be as in the assumption of Theorem~\ref{thm:main} with
$\vec{\lmb}(t)$ replaced by $\vec{\nu}(t)$. Let the constants $\alp_{1}$
and $\dlt_{1}$ be chosen by the third item of Lemma~\ref{lem:decom-near-multi-bubble}
with $\eta=\eta_{\dec}/3$. Choose $r_{0}>0$ small if $T<+\infty$
and $r_{0}=+\infty$ if $T=+\infty$ so that $\|\chi_{r_{0}}z^{\ast}\|_{\dot{H}^{1}}<\frac{1}{2}\dlt_{1}$.
Choose $t_{\dec}$ sufficiently close to $T$ such that $\vec{\nu}(t)\in\calP(\alp_{1})$
and $\|u(t)-W_{\vec{\iota},\vec{\nu}(t)}-z^{\ast}\|_{\dot{H}^{1}}<\frac{1}{2}\dlt_{1}$
for all $t\in[t_{\dec},T)$. In particular, we have $\|\td u(t)-W_{\vec{\iota},\vec{\nu}(t)}\|_{\dot{H}^{1}}\leq\|u(t)-W_{\vec{\iota},\vec{\nu}(t)}-z^{\ast}\|_{\dot{H}^{1}}+\|\chi_{r_{0}}z^{\ast}\|_{\dot{H}^{1}}<\dlt_{1}$
for all $t\in[t_{\dec},T)$, where $\td u(t)=u(t)-(1-\chi_{r_{0}})z^{\ast}$.
For each $t\in[t_{\dec},T)$, choose an open interval $I^{t}\ni t$
such that $\td u(\tau)\in B_{\dlt_{1}}(W_{\vec{\iota},\vec{\nu}(t)})$
and $\vec{\nu}(\tau)\in B_{\eta}(\vec{\nu}(t))$ for all $\tau\in I^{t}$.
Then, we apply Lemma~\ref{lem:decom-near-multi-bubble} to obtain
$G^{t}:I^{t}\to B_{\eta_{\dec}}(\vec{\nu})$ defined by $G^{t}(\tau)=G^{\vec{\iota},\vec{\nu}(t)}(\td u(\tau))$.
Note that each $G^{t}$ is $C^{1}$.

We claim that $\{G^{t}\}_{t\in[t_{\dec},T)}$ is a compatible family.
Suppose $\tau\in I^{t_{1}}\cap I^{t_{2}}$ for some $t_{1},t_{2}\in[t_{\dec},T)$.
By the choice of $\alp_{1}$ and $\dlt_{1}$, we have $\mathrm{dist}(G^{t_{j}}(\tau),\vec{\nu}(t_{j}))<\eta$
for $j\in\{1,2\}$. By the definition of $I^{t_{j}}$, we have $\td u(\tau)\in B_{\dlt_{1}}(W_{\vec{\iota},\vec{\nu}(t_{j})})$
and $\mathrm{dist}(\vec{\nu}(t_{j}),\vec{\nu}(\tau))<\eta$. By triangle
inequality we have $\mathrm{dist}(G^{t_{2}}(\tau),\vec{\nu}(t_{1}))<3\eta=\eta_{\dec}$.
Hence $G^{t_{2}}(\tau)=G^{t_{1}}(\tau)$ by uniqueness.

As $\{G^{t}\}_{t\in[t_{\dec},T)}$ is a compatible family, the definition
$\vec{\lmb}(\tau)=G^{t}(\tau)$ for any $\tau\in[t_{1},T)$ and $t$
with $\tau\in I^{t}$ is well-defined.

\textbf{Step 2.} Completion of the proof.

First, we show that \eqref{eq:main-thm-decom} holds for the new parameters
$\vec{\lmb}(t)$. It suffices to show $\mathrm{dist}(\vec{\lmb}(t),\vec{\nu}(t))\to0$.
If $T<+\infty$, this follows from 
\[
\mathrm{dist}(\vec{\lmb}(t),\vec{\nu}(t))\aeq|F_{k}^{\vec{\iota}}(\vec{\lmb}(t);\td u(t))-F_{k}^{\vec{\iota}}(\vec{\nu}(t);\td u(t))|=|F_{k}^{\vec{\iota}}(\vec{\nu}(t);\td u(t))|
\]
and 
\[
F_{k}^{\vec{\iota}}(\vec{\nu}(t);\td u(t))\aleq\|\chf_{r\aleq\nu_{1}(t)}r^{-1}|\td u(t)-U(\vec{\iota},\vec{\nu}(t);\cdot)|\|_{L^{2}}\aleq o_{t\to T}(1),
\]
where the latter is a consequence of $\vec{\nu}(t)\in\calP_{J}(o_{t\to T}(1))$,
$\nu_{1}(t)\to0$, and the fact that $\calZ$ is compactly supported.
If $T=+\infty$, this is easier because $\td u(t)=u(t)$ and we directly
have $\|u(t)-U(\vec{\iota},\vec{\nu}(t);\cdot)\|_{\dot{H}^{1}}\to0$
without localization.

Now, the definition of $G^{t}$ says that the orthogonality conditions
\eqref{eq:orthog} hold for the function $\td u(t)$. If $T=+\infty$,
then $\td u(t)=u(t)$ so we have \eqref{eq:orthog}. If $T<+\infty$,
using again the fact that $\calZ$ is compactly supported and $\lmb_{1}(t)\to0$,
we have \eqref{eq:orthog} for $u(t)$, possibly after taking $t_{\dec}$
closer to $T$. Finally, the smallness estimates easily follow from
\eqref{eq:main-thm-decom}. This completes the proof.
\end{proof}

\subsection{Spacetime estimate}

Let $u(t)$ be as in the assumption of Theorem~\ref{thm:main}; let
$t_{\dec}<T$ and decompose $u(t)$ on $[t_{\dec},T)$ according to
Corollary~\ref{cor:Decomp-Sol}. To complete the proof of Proposition~\ref{prop:spacetime-est},
it remains to show the spacetime estimate \eqref{eq:sptime-est}.
First, we show the interior spacetime estimate:
\begin{lem}[Interior spacetime estimate]
\label{lem:interior-sptime}There exist $t_{0}\in[t_{\dec},T)$ and
$r_{0}>0$ such that 
\begin{equation}
\int_{t_{0}}^{T}\Big\{\calD+\|\chi_{2r_{0}}g(t)\|_{\dot{H}^{2}}^{2}\Big\} dt<+\infty.\label{eq:interior-sptime}
\end{equation}
If $T=+\infty$, one can choose $r_{0}=+\infty$, i.e., \eqref{eq:sptime-est}
holds.
\end{lem}

\begin{proof}
As $u(t)$ is energy bounded (as in the assumption of Theorem~\ref{thm:main}),
for any $t_{0}\in[t_{\dec},T)$ the energy identity yields 
\[
\int_{t_{0}}^{T}\|u_{t}\|_{L^{2}}^{2}dt=E[u(t_{0})]-\lim_{t\to T}E[u(t)]<+\infty.
\]
Hence, it suffices to show that there exist $r_{0}>0$, $t_{0}\in[t_{\dec},T)$,
and $C>0$ such that for all $t\in[t_{0},T)$ we have 
\begin{equation}
\left\{ \begin{aligned}\calD+\|\chi_{2r_{0}}g\|_{\dot{H}^{2}}^{2} & \leq C\Big(\int\chi_{2r_{0}}^{2}u_{t}^{2}+1\Big) & \text{if }T<+\infty,\\
\calD+\|g\|_{\dot{H}^{2}}^{2} & \leq C\int u_{t}^{2} & \text{if }T=+\infty.
\end{aligned}
\right.\label{eq:sptime-2-2}
\end{equation}
We will prove this estimate only in the finite-time blow-up case $T<+\infty$.
The global case $T=+\infty$ is easier; one can assume $r_{0}=+\infty$
(i.e., no cutoff) in the argument below and use $\|g(t)\|_{\dot{H}^{1}}\to0$
as $t\to+\infty$ directly.

Let $T<+\infty$. The parameters $r_{0}$ and $t_{0}$ will be chosen
in the order
\[
0<T-t_{0}\ll r_{0}\ll1.
\]
We begin with the identity 
\[
u_{t}=\Dlt u+r^{-(D+2)}f(v)=\Dlt(U+g)+r^{-(D+2)}f(P+\rng g).
\]
We then apply \eqref{eq:def-Psi} and introduce (note that $r^{-(D+2)}f'(P)\rng g=r^{-2}f'(P)g$)
\begin{align*}
-H_{U}g & =\Dlt g+r^{-2}f'(P)g,\\
\NL_{U}(g) & =r^{-(D+2)}\{f(P+\rng g)-f(P)-f'(P)\rng g\}
\end{align*}
to obtain 
\begin{equation}
u_{t}=-\sum_{j=2}^{J}\br{\iota}_{j}\frac{\kappa\mu_{j}^{D}}{\lmb_{j}^{2}}\Lmb W_{;j}-H_{U}g+\NL_{U}(g)+\Psi.\label{eq:sptime-2-3}
\end{equation}
Multiplying the above by $\chi_{2r_{0}}$, we arrive at the identity
\begin{align}
\chi_{2r_{0}}u_{t} & =-\sum_{j=2}^{J}\br{\iota}_{j}\frac{\kappa\mu_{j}^{D}}{\lmb_{j}^{2}}\Lmb W_{;j}-H_{U}(\chi_{2r_{0}}g)+\calR,\label{eq:sptime-2-4}\\
\calR & \coloneqq\sum_{j=2}^{J}\br{\iota}_{j}\frac{\kappa\mu_{j}^{D}}{\lmb_{j}^{2}}(1-\chi_{2r_{0}})\Lmb W_{;j}-[\Dlt,\chi_{2r_{0}}]g+\chi_{2r_{0}}\NL_{U}(g)+\Psi.\label{eq:sptime-2-5}
\end{align}
We will show below the following:
\begin{align}
 & \|\calR\|_{L^{2}}\aleq O_{r_{0}}(1)+o_{t_{0}\to T}(\sqrt{\calD})+o_{r_{0}\to0}(\|\chi_{2r_{0}}g\|_{\dot{H}^{2}}),\label{eq:sptime-2-6}\\
 & \|H_{U}(\chi_{2r_{0}}g)\|_{L^{2}}\aeq\|\chi_{2r_{0}}g\|_{\dot{H}^{2}},\label{eq:sptime-2-10}\\
 & \Blan\frac{\kappa\mu_{j}^{D}}{\lmb_{j}^{2}}\Lmb W_{;j},H_{U}(\chi_{2r_{0}}g)\Bran\aleq o_{t_{0}\to T}(1)\cdot\sqrt{\calD}\|\chi_{2r_{0}}g\|_{\dot{H}^{2}}.\label{eq:sptime-2-7}
\end{align}
We assume \eqref{eq:sptime-2-6}--\eqref{eq:sptime-2-7} and finish
the proof of \eqref{eq:sptime-2-2}. From the mixed term estimate
\eqref{eq:sptime-2-7}, we have 
\begin{align*}
 & \Big\|-\sum_{j=2}^{J}\br{\iota}_{j}\frac{\kappa\mu_{j}^{D}}{\lmb_{j}^{2}}\Lmb W_{;j}-H_{U}(\chi_{2r_{0}}g)\Big\|_{L^{2}}^{2}\\
 & \quad\geq\Big\|\sum_{j=2}^{J}\br{\iota}_{j}\frac{\kappa\mu_{j}^{D}}{\lmb_{j}^{2}}\Lmb W_{;j}\Big\|_{L^{2}}^{2}+\|H_{U}(\chi_{2r_{0}}g)\|_{L^{2}}^{2}-o_{t_{0}\to T}(\sqrt{\calD}\|\chi_{2r_{0}}g\|_{\dot{H}^{2}}).
\end{align*}
Applying \eqref{eq:prof-eq-main-size} and the coercivity \eqref{eq:sptime-2-10},
there exists $c>0$ such that 
\[
\Big\|-\sum_{j=2}^{J}\br{\iota}_{j}\frac{\kappa\mu_{j}^{D}}{\lmb_{j}^{2}}\Lmb W_{;j}-H_{U}(\chi_{2r_{0}}g)\Big\|_{L^{2}}^{2}\geq c\Big(\calD+\|\chi_{2r_{0}}g\|_{\dot{H}^{2}}^{2}\Big).
\]
Taking the square root and applying \eqref{eq:sptime-2-6} and \eqref{eq:sptime-2-4},
we obtain 
\[
\sqrt{\calD}+\|\chi_{2r_{0}}g\|_{\dot{H}^{2}}\aleq\|\chi_{2r_{0}}u_{t}\|_{L^{2}}+O_{r_{0}}(1)+o_{t_{0}\to T}(\sqrt{\calD})+o_{r_{0}\to0}(\|\chi_{2r_{0}}g\|_{\dot{H}^{2}}).
\]
Absorbing the last two errors into the LHS, we obtain \eqref{eq:sptime-2-2}
and hence \eqref{eq:interior-sptime}. Therefore, it remains to show
\eqref{eq:sptime-2-6}--\eqref{eq:sptime-2-7}.\smallskip{}

\emph{\uline{Proof of \mbox{\eqref{eq:sptime-2-6}}.}} Recall the
definition of $\calR$ from \eqref{eq:sptime-2-5}. The first term
can be estimated by 
\[
\Big\|\br{\iota}_{j}\frac{\kappa\mu_{j}^{D}}{\lmb_{j}^{2}}(1-\chi_{2r_{0}})\Lmb W_{;j}\Big\|_{L^{2}}\aleq\frac{1}{r_{0}^{D-1}}\mu_{j}^{D}\lmb_{j}^{D-2}\aleq o_{t_{0}\to T}(1).
\]
For the second term, we use $|[\Dlt,\chi_{2r_{0}}]g|\aleq r_{0}^{-1}|g|_{-1}$
and $\|g\|_{\dot{H}^{1}}\aleq1$ to have 
\[
\|[\Dlt,\chi_{2r_{0}}]g\|_{L^{2}}\aleq r_{0}^{-1}\||g|_{-1}\|_{L^{2}}\aleq_{r_{0}}1.
\]
For the third term, we use \eqref{eq:f-04} and \eqref{eq:f-06-HM},
the embeddings $\dot{H}^{1}\embed L^{\frac{2N}{N-2}}$ and $\dot{H}^{2}\embed L^{\frac{2N}{N-4}}$,
the numerology $(p-1)(\frac{1}{2}-\frac{1}{N})+(\frac{1}{2}-\frac{2}{N})=\frac{1}{2}$,
and \eqref{eq:chi-r0-g} to obtain 
\[
\|\chi_{2r_{0}}\NL_{U}(g)\|_{L^{2}}\aleq\|\chi_{2r_{0}}|g|^{p}\|_{L^{2}}\aleq\|\chi_{4r_{0}}g\|_{\dot{H}^{1}}^{p-1}\|\chi_{2r_{0}}g\|_{\dot{H}^{2}}\aleq o_{r_{0}\to0}(\|\chi_{2r_{0}}g\|_{\dot{H}^{2}}).
\]
Finally, $\|\Psi\|_{L^{2}}$ was already estimated in \eqref{eq:Psi-L2}.
This completes the proof of \eqref{eq:sptime-2-6}.\smallskip{}

\emph{\uline{Proof of \mbox{\eqref{eq:sptime-2-10}}.}} Note that
$\lmb_{1}(t)\to0$ implies $\lan\chi_{2r_{0}}g,\calZ_{\lmb_{k}}\ran=\lan g,\calZ_{\lmb_{k}}\ran=0$
for all $t$ close to $T$. Thus the linear coercivity estimate \eqref{eq:lin-coer}
gives 
\[
\|H_{\vec{\lmb}}(\chi_{2r_{0}}g)\|_{L^{2}}\aeq\|\chi_{2r_{0}}g\|_{\dot{H}^{2}},
\]
where $H_{\vec{\lmb}}=-\Dlt-r^{-2}\sum_{j=1}^{J}f'(Q_{;j})$. Therefore,
it suffices to show 
\[
\|H_{\vec{\lmb}}(\chi_{2r_{0}}g)-H_{U}(\chi_{2r_{0}}g)\|_{L^{2}}=o_{t_{0}\to T}(\|\chi_{2r_{0}}g\|_{\dot{H}^{2}}).
\]
Note that 
\[
H_{\vec{\lmb}}(\chi_{2r_{0}}g)-H_{U}(\chi_{2r_{0}}g)=r^{-2}\{f'(P)-\tsum{j=1}Jf'(Q_{;j})\}\chi_{2r_{0}}g.
\]
First, by \eqref{eq:f-03} and \eqref{eq:f-03-HM}, we have 
\[
r^{-2}|f'(P)-f'(\tsum{j=1}JQ_{;j})|\aleq|\td U|^{p-1}.
\]
Next, by \eqref{eq:f-07} and \eqref{eq:f-07-HM}, we have 
\[
|f'(\tsum{j=1}JQ_{;j})-\tsum{j=1}Jf'(Q_{;j})|\aleq\sum_{j,\ell=1}^{J}\chf_{j\neq\ell}|[\lan y\ran^{-2D}]_{;j}[\lan y\ran^{-2D}]_{;\ell}|^{\frac{p-1}{2}}.
\]
By the previous two estimates and the numerology $(\frac{p-1}{2})(\frac{N-2}{N})+(\frac{1}{2}-\frac{2}{N})=\frac{1}{2}$,
we obtain 
\begin{align*}
 & \|H_{\vec{\lmb}}(\chi_{2r_{0}}g)-H_{U}(\chi_{2r_{0}}g)\|_{L^{2}}\\
 & \quad\aleq\Big\{\|\td U\|_{\dot{H}^{1}}^{p-1}+\sum_{j,\ell=1}^{J}\chf_{j\neq\ell}\|[\lan y\ran^{-2D}]_{;j}[\lan y\ran^{-2D}]_{;\ell}\|_{L^{\frac{N}{N-2}}}^{\frac{p-1}{2}}\Big\}\|\chi_{2r_{0}}g\|_{\dot{H}^{2}}.
\end{align*}
It remains to show that the terms in the curly bracket are $o_{t_{0}\to T}(1)$.
For $\td U$, this follows from \eqref{eq:tdU-H1} and \eqref{eq:mu_j_o(1)}.
For the mixed term, scaling the following estimate (for $\lmb<\frac{1}{2}$)
\begin{align*}
 & \|\lan y\ran^{-2D}[\lan y\ran^{-2D}]_{\lmb}\|_{L^{\frac{N}{N-2}}}\\
 & \quad\aeq\|\chf_{(0,\lmb]}\lmb^{-\frac{N}{2}}+\chf_{[\lmb,1]}\lmb^{\frac{N}{2}}r^{-N}+\chf_{[1,+\infty)}\lmb^{\frac{N}{2}}r^{-2N}\|_{L^{1}}^{\frac{N-2}{N}}\aeq\lmb^{D}|\log\lmb|^{\frac{N-2}{N}}
\end{align*}
and applying \eqref{eq:mu_j_o(1)} show $\chf_{j\neq\ell}\|[\lan y\ran^{-2D}]_{;j}[\lan y\ran^{-2D}]_{;\ell}\|_{L^{\frac{N}{N-2}}}=o_{t_{0}\to T}(1)$.
This completes the proof of \eqref{eq:sptime-2-10}.\smallskip{}

\emph{\uline{Proof of \mbox{\eqref{eq:sptime-2-7}}.}} We claim
that 
\begin{equation}
\|\{f'(P)-f'(Q_{;k})\}\cdot[\lan y\ran^{-2D}]_{\ul{;k}}\|_{L^{2}}\aleq\frac{\mu_{k+1}^{D/2}+o(1)\cdot\mu_{k}^{D/2}}{\lmb_{k}}.\label{eq:h_k-est-04}
\end{equation}
Assuming this claim, \eqref{eq:sptime-2-7} follows from $H_{\lmb_{j}}\Lmb W_{;j}=0$
and $\Lmb W_{\ul{;j}}\aleq|[\lan y\ran^{-2D}]_{\ul{;j}}|$: 
\begin{align*}
\Blan\frac{\kappa\mu_{j}^{D}}{\lmb_{j}^{2}}\Lmb W_{;j},H_{U}(\chi_{2r_{0}}g)\Bran & =-\Blan\frac{\kappa\mu_{j}^{D}}{\lmb_{j}^{2}}\Lmb W_{;j},r^{-2}\{f'(P)-f'(Q_{;j})\}(\chi_{2r_{0}}g)\Bran\\
 & \aleq\mu_{j}^{D}\|\{f'(P)-f'(Q_{;j})\}\cdot[\lan y\ran^{-2D}]_{\ul{;j}}\|_{L^{2}}\|r^{-2}\chi_{2r_{0}}g\|_{L^{2}}\\
 & \aleq o_{t_{0}\to T}(1)\cdot\sqrt{\calD}\|\chi_{2r_{0}}g\|_{\dot{H}^{2}}.
\end{align*}
It remains to show \eqref{eq:h_k-est-04}. We separate into three
regions: (i) $r\in[\br{\lmb}_{k+1},\br{\lmb}_{k}]$, (ii) $r\in[\br{\lmb}_{j+1},\br{\lmb}_{j}]$
for some $j<k$, and (iii) $r\in[\br{\lmb}_{j+1},\br{\lmb}_{j}]$
for some $j>k$.

\emph{\uline{Case (i).}} In this region, using \eqref{eq:f-02},
\eqref{eq:U-bd-3}, and \eqref{eq:f-03-HM}, \eqref{eq:U-bd-3-HM},
we have 
\begin{align*}
\{f'(P)-f'(Q_{;k})\}\cdot[\lan y\ran^{-2D}]_{\ul{;k}} & \aleq\frac{r^{2}}{\lmb_{k}^{2}}|[\lan y\ran^{-2D}]_{;k}|^{p-1}(|[\lan y\ran^{-2D}]_{;k-1}|+|[\lan y\ran^{-2D}]_{;k+1}|)\\
 & \aleq\min\Big\{\frac{r^{2}}{\lmb_{k}^{4}},\frac{1}{r^{2}}\Big\}\Big(\frac{1}{\lmb_{k-1}^{D}}+\frac{\lmb_{k+1}^{D}}{r^{2D}}\Big).
\end{align*}
Thus 
\begin{align*}
 & \|\chf_{[\br{\lmb}_{k+1},\br{\lmb}_{k}]}|[\lan y\ran^{-2D}]_{\ul{;k}}|^{p-1}(|[\lan y\ran^{-2D}]_{;k-1}|+|[\lan y\ran^{-2D}]_{;k+1}|)\|_{L^{2}}\\
 & \aleq\frac{\br{\lmb}_{k}^{D-1}}{\lmb_{k-1}^{D}}+\Big\{\chf_{N=7}\frac{\lmb_{k+1}^{5/2}}{\lmb_{k}^{7/2}}+\chf_{N=8}\frac{\lmb_{k+1}^{3}|\log\mu_{k+1}|}{\lmb_{k}^{4}}+\chf_{N\geq9}\frac{\lmb_{k+1}^{D}}{\lmb_{k}^{4}\br{\lmb}_{k+1}^{D-3}}\Big\}\\
 & \aeq\frac{\mu_{k}^{\frac{D+1}{2}}}{\lmb_{k}}+\Big\{\chf_{N=7}\frac{\mu_{k+1}^{5/2}}{\lmb_{k}}+\chf_{N=8}\frac{\mu_{k+1}^{3}|\log\mu_{k+1}|}{\lmb_{k}}+\chf_{N\geq9}\frac{\mu_{k+1}^{\frac{D+3}{2}}}{\lmb_{k}}\Big\}.
\end{align*}

\emph{\uline{Case (ii).}} In this region, we simply use \eqref{eq:U-bd-2}
and \eqref{eq:U-bd-2-HM} to have
\[
\{f'(P)-f'(Q_{;k})\}\cdot[\lan y\ran^{-2D}]_{\ul{;k}}\aleq r^{2}|[\lan y\ran^{-2D}]_{;j}|^{p-1}|[\lan y\ran^{-2D}]_{\ul{;k}}|\aleq\min\Big\{\frac{r^{2}}{\lmb_{j}^{2}},\frac{\lmb_{j}^{2}}{r^{2}}\Big\}\frac{\lmb_{k}^{D-2}}{r^{2D}}.
\]
Hence we have 
\begin{align*}
 & \sum_{j<k}\|\chf_{[\br{\lmb}_{j+1},\br{\lmb}_{j}]}\{f'(P)-f'(Q_{;k})\}\cdot[\lan y\ran^{-2D}]_{\ul{;k}}\|_{L^{2}}\\
 & \aleq\sum_{j<k}\Big\{\chf_{N=7}\frac{\lmb_{k}^{1/2}}{\lmb_{j}^{3/2}}+\chf_{N=8}\frac{\lmb_{k}|\log\mu_{j+1}|}{\lmb_{j}^{2}}+\chf_{N\geq9}\frac{\lmb_{k}^{D-2}}{\lmb_{j}^{2}\br{\lmb}_{j+1}^{D-3}}\Big\}\\
 & \aleq\chf_{N=7}\frac{\mu_{k}^{3/2}}{\lmb_{k}}+\chf_{N=8}\frac{\mu_{k}^{2}|\log\mu_{k+1}|}{\lmb_{k}}+\chf_{N\geq9}\frac{\mu_{k}^{\frac{D+1}{2}}}{\lmb_{k}}.
\end{align*}

\emph{\uline{Case (iii).}} In this region, similarly as in the
previous case, we have 
\[
\{f'(P)-f'(Q_{;k})\}\cdot[\lan y\ran^{-2D}]_{\ul{;k}}\aleq r^{2}|[\lan y\ran^{-2D}]_{;j}|^{p-1}|[\lan y\ran^{-2D}]_{\ul{;k}}|\aleq\frac{\lmb_{j}^{2}}{\lmb_{k}^{D+2}r^{2}}.
\]
Hence we have 
\[
\sum_{j>k}\|\chf_{[\br{\lmb}_{j+1},\br{\lmb}_{j}]}\{f'(P)-f'(Q_{;k})\}\cdot[\lan y\ran^{-2D}]_{\ul{;k}}\|_{L^{2}}\aleq\sum_{j>k}\frac{\lmb_{j}^{2}\br{\lmb}_{j}^{D-1}}{\lmb_{k}^{D+2}}\aleq\frac{\mu_{k+1}^{\frac{D+3}{2}}}{\lmb_{k}}.
\]
This completes the proof of \eqref{eq:h_k-est-04} and hence \eqref{eq:sptime-2-7}.
This completes the proof.
\end{proof}
By the previous lemma, it remains to consider the case of $T<+\infty$
to conclude the proof of \eqref{eq:sptime-est}. This is achieved
by the following:
\begin{lem}[Exterior spacetime estimate]
\label{lem:exterior-sptime}Let $T<+\infty$; choose $r_{0}>0$ and
$t_{0}\in[t_{\dec},T)$ as in Lemma~\ref{lem:interior-sptime}. We
have 
\begin{equation}
\int_{t_{0}}^{T}\|(1-\chi_{r_{0}})g(t)\|_{\dot{H}^{2}}^{2}<+\infty.\label{eq:exterior-sptime}
\end{equation}
\end{lem}

\begin{proof}
Using the nonlinear energy identity, crude bounds $f(\rng u)\aleq|\rng u|^{p}\aleq1$
and $|U(t)|_{2}\aleq r^{-D}$ (For the $|\cdot|_{2}$-bound, one can
use the $X^{2}$-bounds of $\vphi_{\ell}$ in Corollary~\ref{cor:Profiles-vphi}),
we have 
\begin{align*}
\|(1-\chi_{r_{0}})\Dlt u\|_{L_{[t_{0},T)}^{2}L^{2}}^{2} & \aleq1+\|(1-\chi_{r_{0}})r^{-(D+2)}f(\rng u)\|_{L_{[t_{0},T)}^{2}L^{2}}^{2}\aleq1+r_{0}^{-2}|T-t_{0}|,\\
\|(1-\chi_{r_{0}})\Dlt U\|_{L_{[t_{0},T)}^{2}L^{2}}^{2} & \aleq r_{0}^{-2}|T-t_{0}|,
\end{align*}
and hence 
\[
\|(1-\chi_{r_{0}})\Dlt g\|_{L_{[t_{0},T)}^{2}L^{2}}\aleq1+r_{0}^{-1}\sqrt{T-t_{0}}<+\infty
\]
Combining the above estimate with \eqref{eq:interior-sptime}, we
obtain
\[
\|(1-\chi_{r_{0}})g\|_{L_{[t_{0},T)}^{2}\dot{H}^{2}}\aleq\|(1-\chi_{r_{0}})\Dlt g\|_{L_{[t_{0},T)}^{2}L^{2}}+\|\chf_{[r_{0},2r_{0}]}r^{-1}|g|_{-1}\|_{L_{[t_{0},T)}^{2}L^{2}}<+\infty
\]
as desired.
\end{proof}
Lemmas~\ref{lem:interior-sptime} and \ref{lem:exterior-sptime}
give \eqref{eq:sptime-est}. This completes the proof of Proposition~\ref{prop:spacetime-est}.

\section{\label{sec:Proof-of-main-thm}Proof of Theorem~\ref{thm:main}}

In this section, we fix a solution $u(t)$ as in the assumption of
Theorem~\ref{thm:main} and decompose $u(t)$ on $[t_{0},T)$ according
to Proposition~\ref{prop:spacetime-est}. In this last section, we
will prove the modulation estimates (the controls on the time variation
of $\lmb_{j}$) and integrate them to complete the proof of Theorem~\ref{thm:main}.

\subsection{Modulation estimates}

The goal of this subsection is to prove the \emph{modulation estimates}
(Proposition~\ref{prop:Modulation-Est}). We introduce a further
set of notation. For $\psi\in\{\calZ,\Lmb W,\Lmb_{-1}\Lmb W\}$ and
$j,k\in\{1,\dots,J\}$, we denote 
\begin{equation}
\left\{ \begin{aligned}\calA_{jk}^{(\psi)} & \coloneqq\lan\Lmb W_{;j},\psi_{\ul{;k}}\ran, & g_{k}^{(\psi)} & \coloneqq\lan g,-\Lmb_{-1}\psi_{\ul{;k}}\ran,\\
\calB_{jk}^{(\psi)} & \coloneqq\lan-\lmb_{j}\rd_{\lmb_{j}}\td U,\psi_{\ul{;k}}\ran, & h_{k}^{(\psi)} & \coloneqq\lan-H_{U}g+\NL_{U}(g)+\Psi,\psi_{\ul{;k}}\ran.
\end{aligned}
\right.\label{eq:def-calA-B-etc}
\end{equation}
(Regarding the terms in $h_{k}^{(\psi)}$, see \eqref{eq:sptime-2-3}).
We estimate these inner products in Lemma~\ref{lem:Mod-Prelim} below.
On the other hand, due to the vector-valued nature of modulation estimates,
we need some preliminary estimates to solve a linear system. These
are also provided in Lemma~\ref{lem:Mod-Prelim}.
\begin{lem}[Preliminary estimates]
\label{lem:Mod-Prelim}The following estimates hold.
\begin{itemize}
\item (Smallness) We have 
\begin{equation}
\chf_{j\neq k}|\calA_{jk}^{(\psi)}|+|\calB_{jk}^{(\psi)}|+|g_{k}^{(\psi)}|=o(1).\label{eq:ABg-small}
\end{equation}
\item (Size of $\calA_{jk}^{(\psi)}$) We have 
\begin{equation}
\calA_{kk}^{(\calZ)}=1,\quad\calA_{kk}^{(\Lmb W)}=\|\Lmb W\|_{L^{2}}^{2},\quad|\calA_{kk}^{(\Lmb_{-1}\Lmb W)}|\aleq1\label{eq:calA-diag}
\end{equation}
and 
\begin{equation}
\chf_{j\neq k}|\calA_{jk}^{(\psi)}|\aleq\chf_{j>k}\Big(\frac{\lmb_{j}}{\lmb_{k}}\Big)^{D}+\chf_{j<k}\Big(\frac{\lmb_{k}}{\lmb_{j}}\Big)^{D-2}.\label{eq:calA-off-diag}
\end{equation}
\item (Size of $\calB_{jk}^{(\psi)}$) We have 
\begin{equation}
|\calB_{jk}^{(\psi)}|\aleq\chf_{j>k}\Big(\frac{\lmb_{j}}{\lmb_{k}}\Big)^{D}+\chf_{j=k}(\mu_{k+1}^{D}+\mu_{k}^{D}|\log\mu_{k}|)+\chf_{j<k}o(1)\Big(\frac{\lmb_{k}}{\lmb_{j}}\Big)^{D-2}.\label{eq:calB-est}
\end{equation}
\item (Size of $h_{k}^{(\psi)}$) We have 
\begin{align}
|h_{k}^{(\calZ)}|+|h_{k}^{(\Lmb_{-1}\Lmb W)}| & \aleq\frac{\|g\|_{\dot{H}^{2}}}{\lmb_{k}}+\frac{\mu_{k+1}^{D}+o(1)\cdot\mu_{k}^{D}}{\lmb_{k}^{2}}+\calD+\|g\|_{\dot{H}^{2}}^{2},\label{eq:h_k-est-1}\\
|h_{k}^{(\Lmb W)}| & \aleq\frac{\mu_{k+1}^{D}+o(1)\cdot\mu_{k}^{D}}{\lmb_{k}^{2}}+\calD+\|g\|_{\dot{H}^{2}}^{2}.\label{eq:h_k-est-2}
\end{align}
\item (Off-diagonal summation estimates) We have 
\begin{align}
\sum_{j=1}^{J}\Big\{\chf_{j>k}\Big(\frac{\lmb_{j}}{\lmb_{k}}\Big)^{D}+\chf_{j<k}\Big(\frac{\lmb_{k}}{\lmb_{j}}\Big)^{D-2}\Big\}\frac{1}{\lmb_{j}} & \aleq\frac{\mu_{k}^{D-1}+\mu_{k+1}^{D-1}}{\lmb_{k}},\label{eq:off-diag-1}\\
\sum_{j=1}^{J}\Big\{\chf_{j>k}\Big(\frac{\lmb_{j}}{\lmb_{k}}\Big)^{D}+\chf_{j<k}\Big(\frac{\lmb_{k}}{\lmb_{j}}\Big)^{D-2}\Big\}\frac{\mu_{j+1}^{D}}{\lmb_{j}^{2}} & \aleq o\Big(\frac{\mu_{k}^{D}}{\lmb_{k}^{2}}+\calD\Big),\label{eq:off-diag-2}\\
\sum_{j=1}^{J}\Big\{\chf_{j>k}\Big(\frac{\lmb_{j}}{\lmb_{k}}\Big)^{D}+\chf_{j<k}\Big(\frac{\lmb_{k}}{\lmb_{j}}\Big)^{D-2}\Big\}\frac{\mu_{j}^{D}}{\lmb_{j}^{2}} & \aleq o\Big(\frac{\mu_{k+1}^{D}+\mu_{k}^{D}}{\lmb_{k}^{2}}+\calD\Big).\label{eq:off-diag-3}
\end{align}
\end{itemize}
\end{lem}

\begin{proof}
\textbf{Step 1.} Estimates for $\calA_{jk}^{(\psi)}$, $\calB_{jk}^{(\psi)}$,
and $g_{k}^{(\psi)}$.

We show \eqref{eq:ABg-small}--\eqref{eq:calB-est}. The estimate
\eqref{eq:calA-diag} is immediate from $\lan\Lmb W,\calZ\ran=1$
and rescaling. The estimate \eqref{eq:calA-off-diag} follows from
$\psi\aleq\lan y\ran^{-2D}$ and \eqref{eq:almost-orthog}. The estimate
\eqref{eq:calB-est} follows from $\psi\aleq\lan y\ran^{-2D}$ and
\eqref{eq:tdU-weight-L1}. Finally, we use $\|y\lan y\ran^{1/4}\lan y\ran^{-2D}\|_{L^{2}}\aleq1$
from $N\geq7$ and \eqref{eq:g_H1loc_o(1)} to have 
\[
g_{k}^{(\psi)}\aleq\|r^{-1}\lan\tfrac{r}{\lmb_{k}}\ran^{-1/4}g\|_{L^{2}}\|r\lan\tfrac{r}{\lmb_{k}}\ran^{1/4}\psi_{\ul{;k}}\|_{L^{2}}\aleq\|r^{-1}\lan\tfrac{r}{\lmb_{1}}\ran^{-1/4}g\|_{L^{2}}\|y\lan y\ran^{1/4}\psi\|_{L^{2}}\aleq o(1),
\]
completing the proof of \eqref{eq:ABg-small}.

\textbf{Step 2.} Estimates for $h_{k}$.

In this step, we prove \eqref{eq:h_k-est-1} and \eqref{eq:h_k-est-2}.
First, the contribution of $\Psi$ is easily estimated from $\lan\Psi,\psi_{\ul{;k}}\ran\aleq\|\Psi\cdot[\lan y\ran^{-2D}]_{\ul{;k}}\|_{L^{1}}$
and \eqref{eq:Psi-inn}. Next, when $\psi\in\{\calZ,\Lmb_{-1}\Lmb W\}$,
we can simply estimate 
\[
\|H_{U}g\cdot[\lan y\ran^{-2D}]_{\ul{;k}}\|_{L^{1}}\leq\|H_{U}g\|_{L^{2}}\|[\lan y\ran^{-2D}]_{\ul{;k}}\|_{L^{2}}\aleq\frac{1}{\lmb_{k}}\|g\|_{\dot{H}^{2}}.
\]
When $\psi=\Lmb W$, we use $H_{\lmb_{k}}\Lmb W_{;k}=0$ and \eqref{eq:h_k-est-04}
to obtain 
\begin{align*}
\lan-H_{U}g,\Lmb W_{\ul{;k}}\ran & =\lan r^{-2}\{f'(P)-f'(Q_{;k})\}g,\Lmb W_{\ul{;k}}\ran\\
 & \aleq\|\{f'(P)-f'(Q_{;k})\}\Lmb W_{\ul{;k}}\|_{L^{2}}\|r^{-2}g\|_{L^{2}}\aleq\frac{\mu_{k+1}^{D/2}+o(1)\mu_{k}^{D/2}}{\lmb_{k}}\|g\|_{\dot{H}^{2}}.
\end{align*}
Finally, for the nonlinear term, as $\psi\aleq\lan y\ran^{-2D}$,
it suffices to show 
\begin{equation}
\|\NL_{U}(g)\cdot[\lan y\ran^{-2D}]_{\ul{;k}}\|_{L^{1}}\aleq\frac{\mu_{k+1}^{D}+o(1)\cdot\mu_{k}^{D}}{\lmb_{k}^{2}}+\|g\|_{\dot{H}^{2}}^{2}.\label{eq:h_k-est-05}
\end{equation}

We prove \eqref{eq:h_k-est-05}. It is more convenient to separate
the \eqref{eq:HMHF-equiv} case and the \eqref{eq:NLH-rad} case.
For \eqref{eq:HMHF-equiv}, we use \eqref{eq:f-06-HM} to have 
\begin{align*}
\NL_{U}(g)\cdot[\lan y\ran^{-2D}]_{\ul{;k}} & \aleq r^{-(D+2)}|\rng g|^{2}\cdot[\lan y\ran^{-2D}]_{\ul{;k}}\aleq r^{-4}|g|^{2},\\
\|\NL_{U}(g)\cdot[\lan y\ran^{-2D}]_{\ul{;k}}\|_{L^{1}} & \aleq\|r^{-2}g\|_{L^{2}}^{2}\aleq\|g\|_{\dot{H}^{2}}^{2}.
\end{align*}
For \eqref{eq:NLH-rad}, as $p<2$, the previous simple argument fails.
We instead separate into three regions motivated from \eqref{eq:U-bd-1}:
(i) $r\in[A_{0}\br{\lmb}_{k+1},A_{0}^{-1}\br{\lmb}_{k}]$, (ii) $r\in(0,A_{0}\br{\lmb}_{k+1}]$,
and (iii) $r\in[A_{0}^{-1}\br{\lmb}_{k},+\infty)$.

\emph{\uline{Case (i).}} In this region, by \eqref{eq:U-bd-1}
and \eqref{eq:f-06}, we have 
\[
\NL_{U}(g)\cdot[\lan y\ran^{-2D}]_{\ul{;k}}\aleq|g|^{2}\cdot\tfrac{1}{\lmb_{k}^{2}}|[\lan y\ran^{-2D}]_{;k}|^{p-1}\aleq r^{-4}|g|^{2}
\]
and hence 
\[
\|\chf_{[A_{0}\br{\lmb}_{k+1},A_{0}^{-1}\br{\lmb}_{k}]}\NL_{U}(g)\cdot[\lan y\ran^{-2D}]_{\ul{;k}}\|_{L^{1}}\aleq\|r^{-2}g\|_{L^{2}}^{2}\aleq\|g\|_{\dot{H}^{2}}^{2}.
\]

\emph{\uline{Case (ii).}} In this region, by \eqref{eq:f-04},
Young's inequality, and $\frac{2}{2-p}=\frac{2D}{D-2}$, we have 
\[
\NL_{U}(g)\cdot[\lan y\ran^{-2D}]_{\ul{;k}}\aleq|g|^{p}\cdot|[\lan y\ran^{-2D}]_{\ul{;k}}|\aleq\frac{1}{r^{4}}\Big\{|g|^{p}\cdot\frac{r^{4}}{\lmb_{k}^{D+2}}\Big\}\aleq\frac{1}{r^{4}}\Big\{|g|^{2}+\frac{1}{\lmb_{k}^{2D}}\Big(\frac{r}{\lmb_{k}}\Big)^{\frac{8D}{D-2}}\Big\}
\]
so 
\[
\|\chf_{(0,A_{0}\br{\lmb}_{k+1}]}|g|^{p}\cdot[\lan y\ran^{-2D}]_{\ul{;k}}\|_{L^{1}}\aleq\|r^{-2}g\|_{L^{2}}^{2}+\frac{\br{\lmb}_{k+1}^{2D-2}}{\lmb_{k}^{2D}}\mu_{k+1}^{\frac{4D}{D-2}}\aeq\|r^{-2}g\|_{L^{2}}^{2}+\frac{\mu_{k+1}^{D+3+\frac{8}{D-2}}}{\lmb_{k}^{2}}.
\]

\emph{\uline{Case (iii).}} In this region, similarly as in the
previous case, we have 
\[
\NL_{U}(g)\cdot[\lan y\ran^{-2D}]_{\ul{;k}}\aleq|g|^{p}\cdot|[\lan y\ran^{-2D}]_{\ul{;k}}|\aleq\frac{1}{r^{4}}|g|^{p}\cdot\frac{\lmb_{k}^{D-2}}{r^{2D-4}}\aleq\frac{1}{r^{4}}\Big(|g|^{2}+\frac{\lmb_{k}^{2D}}{r^{4D}}\Big)
\]
so 
\[
\|\chf_{[A_{0}^{-1}\br{\lmb}_{k},+\infty)}|g|^{p}\cdot[\lan y\ran^{-2D}]_{\ul{;k}}\|_{L^{1}}\aleq\|r^{-2}g\|_{L^{2}}^{2}+\frac{\mu_{k}^{D+1}}{\lmb_{k}^{2}}.
\]
This completes the proof of \eqref{eq:h_k-est-05}.

\textbf{Step 3.} Off diagonal summation estimates.

We show \eqref{eq:off-diag-1}--\eqref{eq:off-diag-3}. The summation
over $j>k$ follows from
\begin{align*}
\chf_{j>k}\Big(\frac{\lmb_{j}}{\lmb_{k}}\Big)^{D}\frac{1}{\lmb_{j}} & =\chf_{j>k}\Big(\frac{\lmb_{j}}{\lmb_{k}}\Big)^{D-1}\frac{1}{\lmb_{k}}\aleq\frac{\mu_{k+1}^{D-1}}{\lmb_{k}},\\
\chf_{j>k}\Big(\frac{\lmb_{j}}{\lmb_{k}}\Big)^{D}\frac{\mu_{j+1}^{D}}{\lmb_{j}^{2}} & =\chf_{j>k}\frac{(\lmb_{j}/\lmb_{k})^{D-1}}{\lmb_{k}}\frac{\mu_{j+1}^{D}}{\lmb_{j}}\aleq\frac{\mu_{k+1}^{D}}{\lmb_{k+1}}\frac{\mu_{j+1}^{D}}{\lmb_{j}}\aleq o(\calD),\\
\chf_{j>k}\Big(\frac{\lmb_{j}}{\lmb_{k}}\Big)^{D}\frac{\mu_{j}^{D}}{\lmb_{j}^{2}} & =\chf_{j=k+1}\frac{\mu_{k+1}^{2D-2}}{\lmb_{k}^{2}}+\chf_{j>k+1}\Big(\frac{\lmb_{j-1}}{\lmb_{k}}\Big)^{D}\frac{\mu_{j}^{2D}}{\lmb_{j}^{2}}\aleq o\Big(\frac{\mu_{k+1}^{D}}{\lmb_{k}^{2}}+\calD\Big).
\end{align*}
The summation over $j<k$ follows from 
\begin{align*}
\chf_{j<k}\Big(\frac{\lmb_{k}}{\lmb_{j}}\Big)^{D-2}\frac{1}{\lmb_{j}} & =\chf_{j<k}\Big(\frac{\lmb_{k}}{\lmb_{j}}\Big)^{D-1}\frac{1}{\lmb_{k}}\aleq\frac{\mu_{k}^{D-1}}{\lmb_{k}},\\
\chf_{j<k}\Big(\frac{\lmb_{k}}{\lmb_{j}}\Big)^{D-2}\frac{\mu_{j+1}^{D}}{\lmb_{j}^{2}} & =\chf_{j=k-1}\frac{\mu_{k}^{2D}}{\lmb_{k}^{2}}+\chf_{j<k-1}\Big(\frac{\lmb_{k}}{\lmb_{j+1}}\Big)^{D-2}\frac{\mu_{j+1}^{2D-2}}{\lmb_{j}^{2}}\aleq o\Big(\frac{\mu_{k}^{D}}{\lmb_{k}^{2}}+\calD\Big),\\
\chf_{j<k}\Big(\frac{\lmb_{k}}{\lmb_{j}}\Big)^{D-2}\frac{\mu_{j}^{D}}{\lmb_{j}^{2}} & \aleq\mu_{j}^{D}\frac{\mu_{k}^{D}}{\lmb_{k}^{2}}\aleq o(1)\cdot\frac{\mu_{k}^{D}}{\lmb_{k}^{2}}.
\end{align*}
This completes the proof.
\end{proof}
We are now ready to state and prove the modulation estimates.
\begin{prop}[Modulation estimates]
\label{prop:Modulation-Est}We have 
\begin{gather}
\Big|\frac{(\lmb_{k})_{t}}{\lmb_{k}}\Big|\aleq\frac{\|g\|_{\dot{H}^{2}}}{\lmb_{k}}+\frac{\mu_{k+1}^{D}+\mu_{k}^{D}}{\lmb_{k}^{2}}+\calD+\|g\|_{\dot{H}^{2}}^{2},\label{eq:First-Mod}\\
\Big|\frac{(\lmb_{k})_{t}}{\lmb_{k}}+\frac{d}{dt}o(1)-\br{\iota}_{k}\frac{\kappa\mu_{k}^{D}}{\lmb_{k}^{2}}\Big|\aleq\frac{\mu_{k+1}^{D}+o(1)\cdot\mu_{k}^{D}}{\lmb_{k}^{2}}+\calD+\|g\|_{\dot{H}^{2}}^{2}.\label{eq:Ref-Mod}
\end{gather}
\end{prop}

\begin{proof}
We generalize the one-bubble case of \cite{CollotMerleRaphael2017CMP}
to the multi-bubble case.

\textbf{Step 1.} Algebraic computation.

We begin with the evolution equation for $g$ (cf. \eqref{eq:sptime-2-3})
\[
\rd_{t}g=-\rd_{t}U-\sum_{j=1}^{J}\br{\iota}_{j}\frac{\kappa\mu_{j}^{D}}{\lmb_{j}^{2}}\Lmb W_{;j}-H_{U}g+\NL_{U}(g)+\Psi.
\]
Recall that we have set $\mu_{1}=0$ so that the summation above starts
from $j=2$. For $\psi\in\{\calZ,\Lmb W,\Lmb_{-1}\Lmb W\}$, we compute
\begin{align*}
\rd_{t}\lan g,\psi_{\ul{;k}}\ran & =\lan\rd_{t}g,\psi_{\ul{;k}}\ran-\frac{\lmb_{k,t}}{\lmb_{k}}\lan g,\Lmb_{-1}\psi_{\ul{;k}}\ran\\
 & =\Blan-\rd_{t}U-\sum_{j=1}^{J}\br{\iota}_{j}\frac{\kappa\mu_{j}^{D}}{\lmb_{j}^{2}}\Lmb W_{;j},\psi_{\ul{;k}}\Bran-\frac{\lmb_{k,t}}{\lmb_{k}}\lan g,\Lmb_{-1}\psi_{\ul{;k}}\ran\\
 & \peq+\lan-H_{U}g+\NL_{U}(g)+\Psi,\psi_{\ul{;k}}\ran.
\end{align*}
Note that 
\[
-\rd_{t}U-\sum_{j=1}^{J}\br{\iota}_{j}\frac{\kappa\mu_{j}^{D}}{\lmb_{j}^{2}}\Lmb W_{;j}=\sum_{j=1}^{J}\Big(\frac{\lmb_{j,t}}{\lmb_{j}}-\br{\iota}_{j}\frac{\kappa\mu_{j}^{D}}{\lmb_{j}^{2}}\Big)\Lmb W_{;j}-\sum_{j=1}^{J}\frac{\lmb_{j,t}}{\lmb_{j}}(\lmb_{j}\rd_{\lmb_{j}}\td U).
\]
Recalling the notation \eqref{eq:def-calA-B-etc}, we arrive at 
\begin{equation}
\rd_{t}\lan g,\psi_{\ul{;k}}\ran=\sum_{j=1}^{J}\Big\{\Big(\frac{\lmb_{j,t}}{\lmb_{j}}-\br{\iota}_{j}\frac{\kappa\mu_{j}^{D}}{\lmb_{j}^{2}}\Big)\calA_{jk}^{(\psi)}+\frac{\lmb_{j,t}}{\lmb_{j}}(\calB_{jk}^{(\psi)}+\chf_{j=k}g_{k}^{(\psi)})\Big\}+h_{k}^{(\psi)}.\label{eq:ModEst-1-1}
\end{equation}

\textbf{Step 2.} Proof of \eqref{eq:First-Mod}.

Substituting $\psi=\calZ$ into \eqref{eq:ModEst-1-1} and using the
orthogonality \eqref{eq:orthog}, we obtain 
\[
0=\sum_{j=1}^{J}\frac{\lmb_{j,t}}{\lmb_{j}}(\calA_{jk}^{(\calZ)}+\calB_{jk}^{(\calZ)}+\chf_{j=k}g_{k}^{(\calZ)})-\sum_{j=1}^{J}\br{\iota}_{j}\frac{\kappa\mu_{j}^{D}}{\lmb_{j}^{2}}\calA_{jk}^{(\calZ)}+h_{k}^{(\calZ)}.
\]
Dividing the above by $\calA_{kk}^{(\calZ)}+\calB_{kk}^{(\calZ)}+g_{k}^{(\calZ)}$
and using $\calA_{kk}^{(\calZ)}=1$ and $|\calB_{kk}^{(\calZ)}|+|g_{k}^{(\calZ)}|=o(1)$
(see Lemma~\ref{lem:Mod-Prelim}), we obtain 
\begin{equation}
\sum_{j=1}^{J}\frac{\lmb_{j,t}}{\lmb_{j}}(\chf_{j=k}+\calM_{jk})=\wh h_{k},\label{eq:ModEst-2-1}
\end{equation}
where 
\[
\calM_{jk}\coloneqq\begin{cases}
0 & \text{if }j=k,\\
\frac{\calA_{jk}^{(\calZ)}+\calB_{jk}^{(\calZ)}}{1+\calB_{kk}^{(\calZ)}+g_{k}^{(\calZ)}} & \text{if }j\neq k,
\end{cases}\qquad\wh h_{k}\coloneqq\frac{\sum_{j=1}^{J}\br{\iota}_{j}\frac{\kappa\mu_{j}^{D}}{\lmb_{j}^{2}}\calA_{jk}^{(\calZ)}-h_{k}^{(\calZ)}}{1+\calB_{kk}^{(\calZ)}+g_{k}^{(\calZ)}}.
\]

First, we estimate $\wh h_{k}$. The denominator has size $\aeq1$
due to \eqref{eq:ABg-small}. For the numerator, we apply \eqref{eq:calA-diag},
\eqref{eq:calA-off-diag}, \eqref{eq:off-diag-3} (for the summation),
and \eqref{eq:h_k-est-1} (for $h_{k}^{(\calZ)}$): 
\begin{equation}
|\wh h_{k}|\aleq\frac{\|g\|_{\dot{H}^{2}}}{\lmb_{k}}+\frac{\mu_{k+1}^{D}+\mu_{k}^{D}}{\lmb_{k}^{2}}+\calD+\|g\|_{\dot{H}^{2}}^{2}.\label{eq:ModEst-2-2}
\end{equation}
Next, by \eqref{eq:ABg-small}, \eqref{eq:calA-off-diag}, and \eqref{eq:calB-est},
we observe that $\calM_{jk}$ satisfies
\[
|\calM_{jk}|\aleq o(1)\quad\text{and}\quad|\calM_{jk}|\aleq\chf_{j>k}\Big(\frac{\lmb_{j}}{\lmb_{k}}\Big)^{D}+\chf_{j<k}\Big(\frac{\lmb_{k}}{\lmb_{j}}\Big)^{D-2}.
\]
This together with \eqref{eq:off-diag-1}, \eqref{eq:off-diag-2},
and \eqref{eq:off-diag-3} implies a contraction property of $\calM_{jk}$:
\begin{equation}
\begin{aligned}\sum_{j=1}^{J}\calM_{jk} & \Big(\frac{\|g\|_{\dot{H}^{2}}}{\lmb_{j}}+\frac{\mu_{j+1}^{D}+\mu_{j}^{D}}{\lmb_{j}^{2}}+\calD+\|g\|_{\dot{H}^{2}}^{2}\Big)\\
 & \aleq o(1)\cdot\Big(\frac{\|g\|_{\dot{H}^{2}}}{\lmb_{k}}+\frac{\mu_{k+1}^{D}+\mu_{k}^{D}}{\lmb_{k}^{2}}+\calD+\|g\|_{\dot{H}^{2}}^{2}\Big).
\end{aligned}
\label{eq:ModEst-2-3}
\end{equation}
Substituting \eqref{eq:ModEst-2-2} and \eqref{eq:ModEst-2-3} into
the system \eqref{eq:ModEst-2-1} gives \eqref{eq:First-Mod}.

\textbf{Step 3.} Proof of \eqref{eq:Ref-Mod}.

First, we claim 
\begin{equation}
\bigg|\frac{\lmb_{k,t}}{\lmb_{k}}-\br{\iota}_{k}\frac{\kappa\mu_{k}^{D}}{\lmb_{k}^{2}}-\frac{\rd_{t}\lan g,\tfrac{1}{\lmb_{k}^{2}}\Lmb W_{;k}\ran}{\|\Lmb Q\|_{L^{2}}^{2}+g_{k}^{(\Lmb W)}}\bigg|\aleq\frac{\mu_{k+1}^{D}+o(1)\cdot\mu_{k}^{D}}{\lmb_{k}^{2}}+\calD+\|g\|_{\dot{H}^{2}}^{2}.\label{eq:ModEst-3-1}
\end{equation}
Substituting $\psi=\Lmb W$ into \eqref{eq:ModEst-1-1} and using
$\calA_{kk}^{(\Lmb W)}=\|\Lmb W\|_{L^{2}}^{2}$, we obtain 
\begin{equation}
\begin{aligned}\rd_{t}\lan g,\tfrac{1}{\lmb_{k}^{2}}\Lmb W_{;k}\ran & =\Big(\frac{\lmb_{k,t}}{\lmb_{k}}-\br{\iota}_{k}\frac{\kappa\mu_{k}^{D}}{\lmb_{k}^{2}}\Big)\|\Lmb W\|_{L^{2}}^{2}+\frac{\lmb_{k,t}}{\lmb_{k}}g_{k}^{(\Lmb W)}\\
 & +\sum_{j=1}^{J}\Big\{\Big(\frac{\lmb_{j,t}}{\lmb_{j}}-\br{\iota}_{j}\frac{\kappa\mu_{j}^{D}}{\lmb_{j}^{2}}\Big)\chf_{j\neq k}\calA_{jk}^{(\Lmb W)}+\frac{\lmb_{j,t}}{\lmb_{j}}\calB_{jk}^{(\Lmb W)}\Big\}+h_{k}^{(\Lmb W)}.
\end{aligned}
\label{eq:ModEst-3-2}
\end{equation}
We show that the terms in the second line are errors. First, using
\eqref{eq:First-Mod}, \eqref{eq:calA-off-diag}, and \eqref{eq:calB-est},
we have for $\psi\in\{\Lmb W,\Lmb_{-1}\Lmb W\}$ 
\begin{align*}
 & \bigg|\sum_{j=1}^{J}\Big\{\Big(\frac{\lmb_{j,t}}{\lmb_{j}}-\br{\iota}_{j}\frac{\kappa\mu_{j}^{D}}{\lmb_{j}^{2}}\Big)\chf_{j\neq k}\calA_{jk}^{(\psi)}+\frac{\lmb_{j,t}}{\lmb_{j}}\calB_{jk}^{(\psi)}\Big\}\bigg|\\
 & \aleq\sum_{j=1}^{J}\Big(\frac{\|g\|_{\dot{H}^{2}}}{\lmb_{j}}+\frac{\mu_{j+1}^{D}+\mu_{j}^{D}}{\lmb_{j}^{2}}+\calD+\|g\|_{\dot{H}^{2}}^{2}\Big)\\
 & \qquad\qquad\times\Big(\chf_{j>k}\Big(\frac{\lmb_{j}}{\lmb_{k}}\Big)^{D}+\chf_{j=k}(\mu_{k+1}^{D}+\mu_{k}^{D}|\log\mu_{k}|)+\chf_{j<k}\Big(\frac{\lmb_{k}}{\lmb_{j}}\Big)^{D-2}\Big).
\end{align*}
Applying \eqref{eq:off-diag-1} and \eqref{eq:off-diag-3} when $j\neq k$,
the previous display continues as 
\begin{align*}
 & \aleq\frac{\mu_{k+1}^{D-1}+\mu_{k}^{D-1}}{\lmb_{k}}\|g\|_{\dot{H}^{2}}+o(1)\cdot\frac{\mu_{k-1}^{D}+\mu_{k}^{D}}{\lmb_{k}^{2}}+\calD+\|g\|_{\dot{H}^{2}}^{2}\\
 & \aleq o(1)\cdot\frac{\mu_{k+1}^{D}+\mu_{k}^{D}}{\lmb_{k}^{2}}+\calD+\|g\|_{\dot{H}^{2}}^{2}.
\end{align*}
Therefore, we have proved for $\psi\in\{\Lmb W,\Lmb_{-1}\Lmb W\}$
that 
\begin{equation}
\bigg|\sum_{j=1}^{J}\Big\{\Big(\frac{\lmb_{j,t}}{\lmb_{j}}-\br{\iota}_{j}\frac{\kappa\mu_{j}^{D}}{\lmb_{j}^{2}}\Big)\chf_{j\neq k}\calA_{jk}^{(\psi)}+\frac{\lmb_{j,t}}{\lmb_{j}}\calB_{jk}^{(\psi)}\Big\}\bigg|\aleq o(1)\cdot\frac{\mu_{k+1}^{D}+\mu_{k}^{D}}{\lmb_{k}^{2}}+\calD+\|g\|_{\dot{H}^{2}}^{2}.\label{eq:ModEst-3-3}
\end{equation}
Now the claim \eqref{eq:ModEst-3-1} follows from substituting \eqref{eq:ModEst-3-3}
and \eqref{eq:h_k-est-2} into \eqref{eq:ModEst-3-2}.

To complete the proof of \eqref{eq:Ref-Mod}, we will integrate by
parts the last term of LHS\eqref{eq:ModEst-3-1}. Let $\td c=-\frac{N}{2}+3\neq0$
(as $N\neq6$) be such that 
\[
\Lmb_{-1}\Lmb W=\td c\Lmb W+R\quad\text{with}\quad R\aleq\frac{1}{\lan y\ran^{N}}.
\]
Define a function $h:[-\frac{1}{2},\frac{1}{2}]\to\bbR$ such that
\[
h(a)\coloneqq\int_{0}^{a}\frac{a}{(1-a)^{2}}da\quad\text{and}\quad|h(a)|\aleq|a|^{2}.
\]
Now we can manipulate the last term of LHS\eqref{eq:ModEst-3-1} as
\begin{align}
 & \frac{\rd_{t}\lan g,\Lmb W_{\ul{;k}}\ran}{\|\Lmb W\|_{L^{2}}^{2}-\lan g,\Lmb_{-1}\Lmb W_{\ul{;k}}\ran}\nonumber \\
 & =\rd_{t}\Big\{\frac{\lan g,\Lmb W_{\ul{;k}}\ran}{\|\Lmb W\|_{L^{2}}^{2}-\lan g,\Lmb_{-1}\Lmb W_{\ul{;k}}\ran}\Big\}-\frac{\lan g,\Lmb W_{\ul{;k}}\ran\rd_{t}\lan g,\Lmb_{-1}\Lmb W_{\ul{;k}}\ran}{(\|\Lmb W\|_{L^{2}}^{2}-\lan g,\Lmb_{-1}\Lmb W_{\ul{;k}}\ran)^{2}}\nonumber \\
 & =\rd_{t}\Big\{\frac{\lan g,\Lmb W_{\ul{;k}}\ran}{\|\Lmb W\|_{L^{2}}^{2}-\lan g,\Lmb_{-1}\Lmb W_{\ul{;k}}\ran}-\frac{1}{\td c}h\Big(\frac{\lan g,\Lmb_{-1}\Lmb W_{\ul{;k}}\ran}{\|\Lmb W\|_{L^{2}}^{2}}\Big)\Big\}\label{eq:ModEst-3-4}\\
 & \peq+\frac{\lan g,R_{\ul{;k}}\ran\rd_{t}\lan g,\Lmb_{-1}\Lmb W_{\ul{;k}}\ran}{\td c(\|\Lmb W\|_{L^{2}}^{2}-\lan g,\Lmb_{-1}\Lmb W_{\ul{;k}}\ran)^{2}}.\nonumber 
\end{align}
The first line of RHS\eqref{eq:ModEst-3-4} is of the form $\frac{d}{dt}o(1)$
due to \eqref{eq:ABg-small}. Thus it suffices to show that the last
term of RHS\eqref{eq:ModEst-3-4} is of size $O(\calD+\|g\|_{\dot{H}^{2}}^{2})$.
For this purpose, we note (due to the improved decay $R\aleq\lan y\ran^{-N}$)
\[
|\lan g,R_{\ul{;k}}\ran|\aleq\min\{o(1),\lmb_{k}\|g\|_{\dot{H}^{2}}\}
\]
and (using \eqref{eq:ModEst-3-3} and \eqref{eq:h_k-est-1}) 
\begin{align*}
 & \big|\rd_{t}\lan g,\Lmb_{-1}\Lmb W_{\ul{;k}}\ran\big|\\
 & \aleq\Big|\frac{\lmb_{k,t}}{\lmb_{k}}\Big|+\frac{\mu_{k}^{D}}{\lmb_{k}^{2}}+\sum_{j=1}^{J}\Big\{\Big(\frac{\lmb_{j,t}}{\lmb_{j}}-\br{\iota}_{j}\frac{\kappa\mu_{j}^{D}}{\lmb_{j}^{2}}\Big)\chf_{j\neq k}\calA_{jk}^{(\Lmb_{-1}\Lmb W)}+\frac{\lmb_{j,t}}{\lmb_{j}}\calB_{jk}^{(\Lmb_{-1}\Lmb W)}\Big\}+|h_{k}^{(\Lmb_{-1}\Lmb W)}|\\
 & \aleq\frac{\|g\|_{\dot{H}^{2}}}{\lmb_{k}}+\frac{\mu_{k+1}^{D}+\mu_{k}^{D}}{\lmb_{k}^{2}}+\calD+\|g\|_{\dot{H}^{2}}^{2}.
\end{align*}
Combining the previous two displays, we obtain 
\[
\big|\lan g,R_{\ul{;k}}\ran\rd_{t}\lan g,\Lmb_{-1}\Lmb W_{\ul{;k}}\ran\big|\aleq\calD+\|g\|_{\dot{H}^{2}}^{2},
\]
saying that the second line of RHS\eqref{eq:ModEst-3-4} is of size
$O(\calD+\|g\|_{\dot{H}^{2}}^{2})$. Therefore, we have proved that
the last term of LHS\eqref{eq:ModEst-3-1} is of the form $\frac{d}{dt}o(1)+O(\calD+\|g\|_{\dot{H}^{2}}^{2})$.
Substituting this into \eqref{eq:ModEst-3-1} completes the proof
of \eqref{eq:Ref-Mod}.
\end{proof}

\subsection{\label{subsec:Integration}Integration}

We complete the proof of Theorem~\ref{thm:main}.
\begin{proof}
First, it suffices to show $\br{\iota}_{j}=-1$ for \eqref{eq:NLH-rad},
$\br{\iota}_{j}=+1$ for \eqref{eq:HMHF-equiv}, and the scales $\lmb_{j}(t)$
are asymptotically given by \eqref{eq:main-thm-scale} with $T=+\infty$.
Once the sign properties are shown, $J=|m|$ and $\iota_{j}=\mathrm{sgn}(m)$
for \eqref{eq:HMHF-equiv} follow from the  constraint $u\in\calH_{0,m}$.
To classify the signs and scales, we integrate \eqref{eq:Ref-Mod}.

\textbf{Step 1.} Two-sided differential inequalities for integration.

We claim that there exist functions $\lmb_{k}^{+}=e^{o(1)}\lmb_{k}$
and $\lmb_{k}^{-}=e^{o(1)}\lmb_{k}$ such that 
\begin{align}
\frac{d}{dt}\log\lmb_{k}^{+} & \leq\br{\iota}_{k}\frac{\kappa\mu_{k}^{D}}{\lmb_{k}^{2}}\cdot(1+o(1)),\label{eq:Integ-1-1}\\
\frac{d}{dt}\log\lmb_{k}^{-} & \geq\br{\iota}_{k}\frac{\kappa\mu_{k}^{D}}{\lmb_{k}^{2}}\cdot(1+o(1)).\label{eq:Integ-1-2}
\end{align}
First, by \eqref{eq:Ref-Mod}, there exist functions $\wh{\lmb}_{k}(t)=e^{o(1)}\lmb_{k}(t)$
such that 
\begin{equation}
\frac{d}{dt}\log\wh{\lmb}_{k}(t)=\br{\iota}_{k}\frac{\kappa\mu_{k}^{D}}{\lmb_{k}^{2}}\cdot(1+o(1))+O\Big(\frac{\mu_{k+1}^{D+2}}{\lmb_{k+1}^{2}}+\calD+\|g\|_{\dot{H}^{2}}^{2}\Big).\label{eq:integ-1-3}
\end{equation}
The above display also gives 
\[
\frac{d}{dt}\log\wh{\lmb}_{k-1}(t)=O\Big(\frac{\mu_{k-1}^{D}}{\lmb_{k-1}^{2}}+\frac{\mu_{k}^{D+2}}{\lmb_{k}^{2}}+\calD+\|g\|_{\dot{H}^{2}}^{2}\Big).
\]
Letting $\wh{\mu}_{j}\coloneqq\wh{\lmb}_{j}/\wh{\lmb}_{j-1}=\mu_{j}(1+o(1))$
for $j=2,\dots,J$, the above two displays give an auxiliary estimate
\[
\frac{d}{dt}\log\wh{\mu}_{j}=\br{\iota}_{j}\frac{\kappa\mu_{j}^{D}}{\lmb_{j}^{2}}\cdot(1+o(1))+O\Big(\frac{\mu_{j-1}^{D}}{\lmb_{j-1}^{2}}+\frac{\mu_{j+1}^{D+2}}{\lmb_{j+1}^{2}}+\calD+\|g\|_{\dot{H}^{2}}^{2}\Big).
\]
Fix $0<s_{2}<\dots<s_{J}<2$. The previous display with $\wh{\mu}_{j}=e^{o(1)}\mu_{j}$
implies 
\begin{equation}
\frac{d}{dt}\frac{\br{\iota}_{j}}{\kappa s_{j}}\wh{\mu}_{j}^{s_{j}}=\frac{\mu_{j}^{D+s_{j}}}{\lmb_{j}^{2}}\cdot(1+o(1))+O\Big(\frac{\mu_{j-1}^{D+s_{j}}}{\lmb_{j-1}^{2}}+\frac{\mu_{j+1}^{D+2}}{\lmb_{j+1}^{2}}+\calD+\|g\|_{\dot{H}^{2}}^{2}\Big),\label{eq:integ-1-4}
\end{equation}
where we used 
\[
\frac{\mu_{j-1}^{D}}{\lmb_{j-1}^{2}}\mu_{j}^{s_{j}}\aleq\frac{\mu_{j-1}^{D+s_{j}}+\mu_{j}^{D+s_{j}}}{\lmb_{j-1}^{2}}=\frac{\mu_{j-1}^{D+s_{j}}}{\lmb_{j-1}^{2}}+o(1)\cdot\frac{\mu_{j}^{D+s_{j}}}{\lmb_{j}^{2}}.
\]
We add \eqref{eq:integ-1-3} and \eqref{eq:integ-1-4} and use the
property $0<s_{2}<\dots<s_{J}<2$ to obtain 
\begin{align*}
 & \frac{d}{dt}\{\log\wh{\lmb}_{k}\mp\sum_{j=k+1}^{J}\frac{\br{\iota}_{j}}{\kappa s_{j}}\wh{\mu}_{j}^{s_{j}}\}\\
 & \quad=\br{\iota}_{k}\frac{\kappa\mu_{k}^{D}}{\lmb_{k}^{2}}\cdot\big\{1+o(1)+O(\mu_{k}^{s_{k+1}})\big\}\\
 & \quad\peq\mp\sum_{j=k+1}^{J}\frac{\mu_{j}^{D+s_{j}}}{\lmb_{j}^{2}}\cdot\big\{1+o(1)+O(\mu_{j}^{2-s_{j}}+\mu_{j}^{s_{j+1}-s_{j}})\big\}+O(\calD+\|g\|_{\dot{H}^{2}}^{2})\\
 & \quad=\br{\iota}_{k}\frac{\kappa\mu_{k}^{D}}{\lmb_{k}^{2}}\cdot(1+o(1))\mp\sum_{j=k+1}^{J}\frac{\mu_{j}^{D+s_{j}}}{\lmb_{j}^{2}}\cdot(1+o(1))+O(\calD+\|g\|_{\dot{H}^{2}}^{2}).
\end{align*}
Notice that the middle summation term is sign definite. Recalling
$\calD+\|g\|_{\dot{H}^{2}}^{2}\in L_{t}^{1}$ from \eqref{eq:sptime-est},
the claim \eqref{eq:Integ-1-1}--\eqref{eq:Integ-1-2} follows.

\textbf{Step 2.} Global existence and stabilization of $\lmb_{1}$.

As $\mu_{1}\equiv0$ by definition, \eqref{eq:Integ-1-1}--\eqref{eq:Integ-1-2}
reads 
\[
\frac{d}{dt}\{\log\lmb_{1}+o(1)\}\leq0\leq\frac{d}{dt}\{\log\lmb_{1}+o(1)\}.
\]
Integrating this, there exists $L_{\infty}\in(0,\infty)$ such that
\[
\lmb_{1}(t)=L_{\infty}\cdot(1+o(1)).
\]
Since $\lmb_{1}(t)$ does not converge to $0$, we have $T=+\infty$.

\textbf{Step 3.} Signs and asymptotics for the other bubbles.

In this last step, we finish the proof. By scaling invariance, we
may assume $L_{\infty}=1$. Recall \eqref{eq:def-lmb-ex-ell} and
denote $\lmb_{k}^{\mathrm{ex}}(t)=\lmb_{k,1}^{\mathrm{ex}}(t)$. We
need to show $\br{\iota}_{k}=-1$ for \eqref{eq:NLH-rad}, $\br{\iota}_{k}=+1$
for \eqref{eq:HMHF-equiv}, and $\lmb_{k}(t)=\lmb_{k}^{\mathrm{ex}}(t)\cdot(1+o(1))$
for all $k=2,\dots,J$.

We use induction on $k$; suppose for some $k\geq2$ that we proved
$\br{\iota}_{j}=-1$ and $\lmb_{j}=\lmb_{j}^{\mathrm{ex}}(t)\cdot(1+o(1))$
for $j=1,\dots k-1$. We substitute the inductive hypothesis and \eqref{eq:def-lmb-ex-ell}
into \eqref{eq:Integ-1-2} for $k$ to obtain 
\[
\frac{d}{dt}\log\lmb_{k}^{-}\geq\br{\iota}_{k}\frac{\kappa\mu_{k}^{D}}{\lmb_{k}^{2}}\cdot(1+o(1))\geq\br{\iota}_{k}\kappa\lmb_{k}^{D-2}\Big(\frac{t^{\alp_{k-1}}}{\beta_{k-1}}\Big)^{D}\cdot(1+o(1)),
\]
which gives 
\[
\frac{d}{dt}\Big(\frac{1}{(\lmb_{k}^{-})^{D-2}}\Big)\leq-\frac{\br{\iota}_{k}\kappa(D-2)}{\beta_{k-1}^{D}}\cdot t^{D\alp_{k-1}}\cdot(1+o(1)).
\]
As $\kappa\neq0$ and $\alp_{k-1}\geq0$, integrating the above gives
\[
\frac{1}{(\lmb_{k}^{-})^{D-2}}\leq\Big(\frac{-\br{\iota}_{k}\kappa(D-2)}{\beta_{k-1}^{D}(D\alp_{k-1}+1)}+o(1)\Big)\cdot t^{D\alp_{k-1}+1}.
\]
Since $\lmb_{k}^{-}>0$ for all time, we must have $\br{\iota}_{k}\kappa<0$.
For \eqref{eq:NLH-rad}, $\kappa>0$ so we have $\br{\iota}_{k}=-1$.
For \eqref{eq:HMHF-equiv}, $\kappa<0$ so we have $\br{\iota}_{k}=+1$.
Once we determined the sign, taking the power of $-\frac{1}{D-2}$
and using 
\[
\frac{D\alp_{k-1}+1}{D-2}=\alp_{k}\quad\text{and}\quad\beta_{k-1}^{\frac{D}{D-2}}\Big(\frac{\alp_{k}}{|\kappa|}\Big)^{\frac{1}{D-2}}=\beta_{k}
\]
from \eqref{eq:def-lmb-ex-ell} and \eqref{eq:def-beta_j}, we arrive
at the lower bound 
\[
\lmb_{k}^{-}(t)\geq\frac{\beta_{k}}{t^{\alp_{k}}}\cdot(1+o(1)).
\]
Proceeding similarly with \eqref{eq:Integ-1-1} instead of \eqref{eq:Integ-1-2},
we have the upper bound 
\[
\lmb_{k}^{+}(t)\leq\frac{\beta_{k}}{t^{\alp_{k}}}\cdot(1+o(1)).
\]
The previous two displays imply 
\[
\lmb_{k}(t)=\frac{\beta_{k}}{t^{\alp_{k}}}\cdot(1+o(1)),
\]
completing the induction step. This completes the proof of Theorem~\ref{thm:main}.
\end{proof}
As mentioned at the beginning of Section~\ref{subsec:Strategy},
Theorem~\ref{thm:main-HMHF} for \eqref{eq:HMHF-equiv} and Theorem~\ref{thm:main-NLH}
for \eqref{eq:NLH-rad} are direct corollaries of Theorem~\ref{thm:main}.

\appendix

\section{\label{sec:Proof-of-R1}Proof of \eqref{eq:R1-01}--\eqref{eq:R1-04}}

In this appendix, we prove the estimates for $\calR_{J+1}^{(1)}[\vphi]$
in the proof of Lemma~\ref{lem:contract-recur}, i.e., \eqref{eq:R1-01}--\eqref{eq:R1-04}.
Recall the notation of Section~\ref{subsec:Notation-for-rigorous-construction}
and Steps~0 and 1 of the proof of Lemma~\ref{lem:contract-recur}.
\begin{proof}[Proof of \eqref{eq:R1-01}]
We write 
\[
\calR_{J+1}^{(1)}[\vphi]=r^{-2}\{f'(Q_{;J+1}+P_{J})-f'(Q_{;J+1})\}\chi_{\dlt_{0}\lmb_{J}}\vphi+\NL_{W_{;J+1}+U_{J}}(\chi_{\dlt_{0}\lmb_{J}}\vphi).
\]
We apply \eqref{eq:f-03} and \eqref{eq:f-03-HM} to the first term
and \eqref{eq:f-04} and \eqref{eq:f-06-HM} to the second to obtain
\begin{align*}
\calR_{J+1}^{(1)}[\vphi] & \aleq\chf_{(0,2\dlt_{0}\lmb_{J}]}\{r^{-2}|P_{J}|^{p-1}|\vphi|+r^{-(D+2)}|\rng{\vphi}|^{p}\}.\\
 & \aeq\chf_{(0,2\dlt_{0}\lmb_{J}]}\{|U_{J}|^{p-1}+|\vphi|^{p-1}\}\cdot|\vphi|.
\end{align*}
As \eqref{eq:1-2} and \eqref{eq:1-5} imply 
\begin{equation}
\chf_{(0,2\dlt_{0}\lmb_{J}]}\{|U_{J}|^{p-1}+|\vphi|^{p-1}\}\aleq\chf_{(0,2\dlt_{0}\lmb_{J}]}\{\lmb_{J}^{-2}+K_{0}^{p-1}\mu_{J+1}^{p}r^{-2}\},\label{eq:R1-tmp-01}
\end{equation}
we obtain \eqref{eq:R1-01}.
\end{proof}
\begin{proof}[Proof of \eqref{eq:R1-02}]
Letting $\vphi^{(t)}=t\vphi+(1-t)\vphi'$ for $t\in[0,1]$, we have
the identity 
\begin{align*}
 & \calR_{J+1}^{(1)}[\vphi]-\calR_{J+1}^{(1)}[\vphi']\\
 & \quad=\tint 01r^{-2}\{f'(Q_{;J+1}+P_{J}+\chi_{\dlt_{0}\lmb_{J}}\rng{\vphi}^{(t)})-f'(Q_{;J+1})\}dt\cdot\chi_{\dlt_{0}\lmb_{J}}(\vphi-\vphi').
\end{align*}
We then apply \eqref{eq:f-03}, \eqref{eq:f-03-HM}, \eqref{eq:R1-tmp-01},
and $K_{0}^{p-1}\mu_{J+1}^{p}<\dlt_{0}^{2}$ to have 
\begin{equation}
\begin{aligned} & \chf_{(0,2\dlt_{0}\lmb_{J}]}r^{-2}\{f'(Q_{;J+1}+P_{J}+\chi_{\dlt_{0}\lmb_{J}}\rng{\vphi}^{(t)})-f'(Q_{;J+1})\}\\
 & \quad\aleq\chf_{(0,2\dlt_{0}\lmb_{J}]}\{|U_{J}|^{p-1}+|\vphi^{(t)}|^{p-1}\}\aleq\chf_{(0,2\dlt_{0}\lmb_{J}]}\dlt_{0}^{2}r^{-2},
\end{aligned}
\label{eq:R1-tmp-02}
\end{equation}
which gives \eqref{eq:R1-02}.
\end{proof}
\begin{proof}[Proof of \eqref{eq:R1-03}]
Note that for $j=1,\dots,J$ 
\begin{align}
 & \lmb_{j}\rd_{\lmb_{j}}(\calR_{J+1}^{(1)}[\vphi])\nonumber \\
 & =r^{-2}\{f'(Q_{;J+1}+P_{J}+\chi_{\dlt_{0}\lmb_{J}}\rng{\vphi})-f'(Q_{;J+1})\}\cdot\lmb_{j}\rd_{\lmb_{j}}(\chi_{\dlt_{0}\lmb_{J}}\vphi)\label{eq:R1-4-02}\\
 & \peq+r^{-2}\{f'(Q_{;J+1}+P_{J}+\chi_{\dlt_{0}\lmb_{J}}\rng{\vphi})-f'(Q_{;J+1}+P_{J})\}\cdot\lmb_{j}\rd_{\lmb_{j}}U_{J}.\label{eq:R1-4-03}
\end{align}
We apply \eqref{eq:R1-tmp-02} to estimate 
\[
\eqref{eq:R1-4-02}\aleq\chf_{(0,2\dlt_{0}\lmb_{J}]}\dlt_{0}^{2}r^{-2}\{|\lmb_{j}\rd_{\lmb_{j}}\vphi|+\chf_{j=J}|\vphi|\}\aleq K_{0}\dlt_{0}^{2}\lmb_{j}^{-D}\cdot r^{-2}\omg_{J+1}.
\]
Estimating \eqref{eq:R1-4-03} is more delicate. In the \eqref{eq:HMHF-equiv}
case, we apply \eqref{eq:f-03-HM} to obtain 
\begin{align*}
 & r^{-2}\{f'(Q_{;J+1}+P_{J}+\chi_{\dlt_{0}\lmb_{J}}\rng{\vphi})-f'(Q_{;J+1}+P_{J})\}\\
 & \quad\aleq r^{-2}|\rng{\vphi}|\aleq r^{-2}\cdot r^{D}K_{0}\lmb_{J}^{-D}\omg_{J+1}\aleq K_{0}\dlt_{0}^{D}\cdot r^{-2}\omg_{J+1}.
\end{align*}
In the \eqref{eq:NLH-rad} case, by \eqref{eq:1-2}, we can choose
$A>1$ such that $|W_{;J+1}+U_{J}|\aeq W_{\lmb_{J+1}}$ on $(0,\frac{1}{A}\br{\lmb}_{J+1}]$
and $|W_{;J+1}+U_{J}|\aeq W_{\lmb_{J}}$ on $[A\br{\lmb}_{J+1},2\dlt_{0}\lmb_{J}]$.
In these regions, we apply \eqref{eq:f-02} to estimate 
\begin{align*}
 & \chf_{(0,\frac{1}{A}\br{\lmb}_{J+1}]}\{f'(W_{;J+1}+U_{J}+\chi_{\dlt_{0}\lmb_{J}}\vphi)-f'(W_{;J+1}+U_{J})\}\aleq\chf_{(0,\frac{1}{A}\br{\lmb}_{J+1}]}W_{\lmb_{J+1}}^{p-2}|\vphi|\\
 & \aleq(\chf_{(0,\lmb_{J+1}]}\lmb_{J+1}^{D-2}+\chf_{[\lmb_{J+1},\frac{1}{A}\br{\lmb}_{J+1}]}\lmb_{J+1}^{-(D-2)}r^{2D-4})\cdot K_{0}\lmb_{J}^{-D}\omg_{J+1}\aleq K_{0}\mu_{J+1}\cdot r^{-2}\omg_{J+1}
\end{align*}
and 
\begin{align*}
 & \chf_{[A\br{\lmb}_{J+1},2\dlt_{0}\lmb_{J}]}\{f'(W_{;J+1}+U_{J}+\chi_{\dlt_{0}\lmb_{J}}\vphi)-f'(W_{;J+1}+U_{J})\}\\
 & \aleq\chf_{[A\br{\lmb}_{J+1},2\dlt_{0}\lmb_{J}]}W_{\lmb_{J}}^{p-2}|\vphi|\aleq\chf_{[A\br{\lmb}_{J+1},2\dlt_{0}\lmb_{J}]}\lmb_{J}^{D-2}\cdot K_{0}\lmb_{J}^{-D}\omg_{J+1}\aleq K_{0}\dlt_{0}^{2}\cdot r^{-2}\omg_{J+1}.
\end{align*}
In the intermediate region, we use \eqref{eq:f-03} and $\omg_{J+1}\aeq\mu_{J+1}$
at $r\aeq\br{\lmb}_{J+1}$ to estimate 
\begin{align*}
 & \chf_{[\frac{1}{A}\br{\lmb}_{J+1},A\br{\lmb}_{J+1}]}\{f'(W_{;J+1}+U_{J}+\chi_{\dlt_{0}\lmb_{J}}\vphi)-f'(W_{;J+1}+U_{J})\}\\
 & \aleq\chf_{[\frac{1}{A}\br{\lmb}_{J+1},A\br{\lmb}_{J+1}]}|\vphi|^{p-1}\aleq\chf_{[\frac{1}{A}\br{\lmb}_{J+1},A\br{\lmb}_{J+1}]}K_{0}^{p-1}\lmb_{J}^{-2}\omg_{J+1}^{p-1}\aleq K_{0}^{p-1}\mu_{J+1}^{p}\cdot r^{-2}\omg_{J+1}.
\end{align*}
The previous four displays and parameter dependence \eqref{eq:profile-param-dep}
imply that in either the \eqref{eq:NLH-rad} or \eqref{eq:HMHF-equiv}
cases, we have 
\begin{equation}
r^{-2}\{f'(Q_{;J+1}+P_{J}+\chi_{\dlt_{0}\lmb_{J}}\rng{\vphi})-f'(Q_{;J+1}+P_{J})\}\aleq\chf_{(0,2\dlt_{0}\lmb_{J}]}K_{0}\dlt_{0}^{2}\cdot r^{-2}\omg_{J+1}.\label{eq:R1-4-01}
\end{equation}
This estimate together with \eqref{eq:1-4} gives 
\[
\eqref{eq:R1-4-03}\aleq\chf_{(0,2\dlt_{0}\lmb_{J}]}K_{0}\dlt_{0}^{2}\lmb_{j}^{-D}\cdot r^{-2}\omg_{J+1}.
\]
This completes the proof of \eqref{eq:R1-03}.
\end{proof}
\begin{proof}[Proof of \eqref{eq:R1-04}]
We write 
\begin{align}
 & \lmb_{j}\rd_{\lmb_{j}}(\calR_{J+1}^{(1)}[\vphi]-\calR_{J+1}^{(1)}[\vphi'])\nonumber \\
 & =r^{-2}\{f'(Q_{;J+1}+P_{J}+\chi_{\dlt_{0}\lmb_{J}}\rng{\vphi})-f'(Q_{;J+1})\}\cdot\lmb_{j}\rd_{\lmb_{j}}(\chi_{\dlt_{0}\lmb_{J}}(\vphi-\vphi'))\label{eq:R1-6-01}\\
 & \peq+r^{-2}\{f'(Q_{;J+1}+P_{J}+\chi_{\dlt_{0}\lmb_{J}}\rng{\vphi})-f'(Q_{;J+1}+P_{J}+\chi_{\dlt_{0}\lmb_{J}}\rng{\vphi}')\}\label{eq:R1-6-02}\\
 & \qquad\qquad\times\{\lmb_{j}\rd_{\lmb_{j}}(\chi_{\dlt_{0}\lmb_{J}}\vphi')+\lmb_{j}\rd_{\lmb_{j}}U_{J}\}.\nonumber 
\end{align}
We apply \eqref{eq:R1-tmp-02} and then use $\|\vphi\|_{X},\|\vphi'\|_{X}\leq K_{0}\lmb_{J}^{-D}$
to obtain 
\begin{align*}
\eqref{eq:R1-6-01} & \aleq\chf_{(0,2\dlt_{0}\lmb_{J}]}\dlt_{0}^{2}\{\|\lmb_{j}\rd_{\lmb_{j}}(\vphi-\vphi')\|_{X}+\chf_{j=J}\|\vphi-\vphi'\|_{X}\}\cdot r^{-2}\omg_{J+1},\\
 & \aleq\chf_{(0,2\dlt_{0}\lmb_{J}]}\dlt_{0}^{2}\{\|\lmb_{j}\rd_{\lmb_{j}}(\vphi-\vphi')\|_{X}+K_{0}^{2-p}\lmb_{J}^{-(D-2)}\|\vphi-\vphi'\|_{X}^{p-1}\}\cdot r^{-2}\omg_{J+1}.
\end{align*}
To estimate \eqref{eq:R1-6-02}, we simply use \eqref{eq:f-03}, \eqref{eq:f-03-HM},
and \eqref{eq:1-6} to obtain a rough bound 
\begin{align*}
 & r^{-2}\{f'(Q_{;J+1}+P_{J}+\chi_{\dlt_{0}\lmb_{J}}\rng{\vphi})-f'(Q_{;J+1}+P_{J}+\chi_{\dlt_{0}\lmb_{J}}\rng{\vphi}')\}\\
 & \quad\aleq\chf_{(0,2\dlt_{0}\lmb_{J}]}|\vphi-\vphi'|^{p-1}\aleq\dlt_{0}^{2}\lmb_{J+1}^{2}\mu_{J+1}^{-D}\cdot\|\vphi-\vphi'\|_{X}^{p-1}\cdot r^{-2}\omg_{J+1},
\end{align*}
which together with \eqref{eq:1-4} and $\mu_{J+1}^{-D}\lmb_{j}^{-D}\leq\lmb_{J+1}^{-D}$
implies 
\[
\eqref{eq:R1-6-02}\aleq K_{0}\dlt_{0}^{2}\lmb_{J+1}^{-(D-2)}\|\vphi-\vphi'\|_{X}^{p-1}\cdot r^{-2}\omg_{J+1}.
\]
This completes the proof of \eqref{eq:R1-04}.
\end{proof}
\bibliographystyle{abbrv}
\bibliography{References}

\end{document}